\documentclass[12pt]{amsart}

\usepackage[mathscr]{eucal}
\usepackage{amscd}
\usepackage{amsfonts}
\usepackage{amsmath, amsthm, amssymb}
\usepackage{latexsym}
\usepackage[dvips]{graphics}

\usepackage[english]{babel}
\usepackage[latin1]{inputenc}
\usepackage[all]{xy}

\theoremstyle{plain} 

\newcommand{\Rh}{\mathbf{R}\text{Hom}}
\newcommand{\Lt}{\otimes^\mathbf{L}}
\newcommand{\Th}{\text{Hom}^\bullet}
\newcommand{\Tt}{\otimes^\bullet}
\newcommand{\C}{\mathcal{C}}
\newcommand{\K}{\mathcal{H}}
\newcommand{\D}{\mathcal{D}}
\newcommand{\T}{\mathbf{T}}
\newcommand{\RH}{\mathbf{H}}
\newcommand{\cp}{\coprod_{i\in I}}

\newtheorem{teor}{Theorem}[section]
\newtheorem{defi}{Definition}
\newtheorem{lemma}[teor]{Lemma}
\newtheorem{prop}[teor]{Proposition}
\newtheorem{cor}[teor]{Corollary}
\newtheorem{rem}[teor]{Remark}
\newtheorem{rems}[teor]{Remarks}
\newtheorem{exem}[teor]{Example}
\newtheorem{exems}[teor]{Examples}
\newtheorem{ques}[teor]{Question}
\newtheorem{quess}[teor]{Questions}

\def\Ext {\mathop{\rm Ext}\nolimits}

\numberwithin{equation}{section}

\title{CLASSICAL DERIVED FUNCTORS AS FULLY FAITHFUL EMBEDDINGS 
}
\author{Pedro Nicol\'as and Manuel Saor\'in} 

\address{
\begin{flushleft}
        \hspace{0.3cm}  Departamento de Matem\'aticas\\
         \hspace{0.3cm}  Universidad de Murcia. Aptdo 4021\\
         \hspace{0.3cm}  30100 Espinardo, Murcia\\
         \hspace{0.3cm}  SPAIN\\
\end{flushleft}
}
\email{pedronz@um.es} 
\email{msaorinc@um.es}

\thanks{The final version of this paper will be submitted for publication elsewhere.}

\begin{document}

\maketitle


\begin{abstract}

Given associative unital algebras $A$ and $B$ and a complex
$T^\bullet$ of $B-A-$bi\-modules, we give necessary and sufficient
conditions for the total derived functors, $\Rh_A(T^\bullet
,?):\D(A)\longrightarrow\D(B)$ and $?\Lt_BT^\bullet
:\D(B)\longrightarrow\D(A)$, to be fully faithful. We also give
criteria for these functors to be one of the fully faithful functors
appearing in a recollement of derived categories. In the case when
$T^\bullet$ is just a $B-A-$bimodule, we connect the results with
(infinite dimensional) tilting theory and show that some  open
question on the fully faithfulness of $\Rh_A(T,?)$ is related to the
classical Wakamatsu tilting problem.

 \bigskip


{\it $2000$ Mathematics Subject Classification{\rm : 18Gxx, 16Wxx, 16Gxx}} 

\end{abstract}


\section{Introduction}
In October 2013, the second named author was invited to the
\emph{46th Japan Symposium on Ring and Representation Theory} and
the title of one of his talks was exactly the title of this paper,
which tries to be a much expanded version of that talk. Several
results that we will present here are particular cases of results
given in \cite{NS} in the language of dg categories and will be
published elsewhere. For different reasons, the language of dg
categories tends to be difficult to understand by people working
both in Ring Theory and Representation Theory,  and it is specially
so for beginners in the field. The main motivation of the present
work is to isolate the material of \cite{NS} which applies to
ordinary (always associative unital) algebras and rings, and use it
to go further in its applications in terms of recollement situations
that were only indirectly considered in \cite{NS}. We hope in this
way that the results that we present interest ring and
representation theorists. Only minor references to dg algebras will
be needed, but the bulk of the contents stays within the scope of
ordinary algebras and rings.

Apart from the extraordinary hospitality of the organizers, the most
captivating thing for the mentioned author was the very active,
lively and enthusiastic Japanese youth community in the field, who
presented their recent work,  sometimes impressive. This paper is
written  thinking mainly on them. Aimed at beginners, in the initial
sections of the paper we have tried to be as self-contained as
possible, referring to the written literature only for technical
definitions, some proofs and specific details.

All throughout the paper the term 'algebra' will denote an
associative unital algebra  over a ground commutative ring $k$,
fixed in the sequel. Unless otherwise stated, 'module' will mean
'right module' and the corresponding category of  modules over an
algebra $A$ will be denoted by $\text{Mod}-A$. Left $A$-modules will
be looked at as right modules over the opposite algebra $A^{op}$.
Then $\D(A)$ and $\D(A^{op})$ will denote the derived categories of
the categories of right and left $A$-modules, respectively. On what
concerns set-theoretical matters, unlike \cite{NS}, in this paper we
will avoid the universe axiom and, instead, we will distinguish
between 'sets' and '(proper) classes'. All families will be
set-indexed families and an expression of the sort 'it has
(co)products' will always mean 'it has set-indexed (co)products'.

By now, the following is a classical result due to successive
contributions by Happel , Rickard and Keller (see \cite{H1},
\cite{R}, \cite{R2} and \cite{Ke}). We refer the reader to sections
2 and 3 for the pertinent definitions.

\begin{teor} \label{teor.Happel-Rickard-Keller}
Let $A$ and $B$ be ordinary algebras and let $T^\bullet$ be a
complex of $B-A-$bimodules. The following assertions are equivalent:

\begin{enumerate}
\item The functor $?\Lt_BT^\bullet :\D(B)\longrightarrow\D(A)$ is an
equivalence of categories;
\item The functor $\Rh_A(T^\bullet ,?):\D(A)\longrightarrow\D(B)$ is
an equivalence of categories;
\item $T^\bullet_A$ is a classical tilting object of $\D(A)$ such
that the canonical algebra morphism
$B\longrightarrow\text{End}_{\D(A)}(T^\bullet )$ is an isomorphism.
\end{enumerate}
\end{teor}

It is natural to ask what should be the substitute of assertion (3) in
this theorem when, in assertion (1) (resp. assertion (2)), we only
require that $?\Lt_BT^\bullet$ (resp. $\Rh_A(T^\bullet ,?)$) be
fully faithful. That is the first goal of this paper. Namely, given
a complex $T^\bullet$ of $B-A-$bimodules, we want to give necessary
and sufficient conditions for $?\Lt_BT^\bullet$ and $\Rh_A(T^\bullet
,?)$ to be fully faithful functors.

On the other hand, a weaker condition than the one in the theorem
appears when the algebras $A$ and $B$ admit a recollement situation
$$\xymatrix{\D^{'} \ar@<0ex>[rr]|{\hspace{0.1 cm}i_{\ast}=i_{!}
\hspace{0.05 cm}}  &&  \D
\ar@<0ex>[rr]|{\hspace{0.1cm}j^{!}=j^{\ast}} \ar@<1ex>[ll]^{i^{!}}
\ar@<-1 ex>[ll]_{i^{\ast}}  && \D^{''} \ar@<1ex>[ll]^{j_{\ast}}
\ar@<-1 ex>[ll]_{j_{!}} }$$  as defined by Beilinson, Bernstein and
Deligne (\cite{BBD}), where either $\{\D ,\D'\}=\{\D(A) ,\D(B)\}$ or
$\{\D ,\D''\}=\{\D(A) ,\D(B)\}$. In these cases, the functors
$i_*=i_!$, $j_!$ and $j_*$ are also fully faithful. This motivates
the  second goal of the paper. We want to give necessary and
sufficient conditions for those recollements to exist, but imposing
the condition that some of the functors in the picture be either
$?\Lt_BT^\bullet$ or $\Rh_A(T^\bullet ,?)$.

Finally, the following are natural questions for which we want to
have an answer.

\begin{quess} \label{ques.recollement situations}
Let $T^\bullet$  be a complex  of $B-A-$bimodules.
\begin{enumerate}
\item Suppose that $\Rh_A(T^\bullet ,?):\D(A)\longrightarrow\D(B)$
is fully faithful.

\begin{enumerate}
\item Is there a recollement $$\xymatrix{\D(A) \ar@<0ex>[rr]|{\hspace{0.1
cm}i_{\ast}=i_{!} \hspace{0.05 cm}}  &&  \D(B)
\ar@<0ex>[rr]|{\hspace{0.1cm}j^{!}=j^{\ast}} \ar@<1ex>[ll]^{i^{!}}
\ar@<-1 ex>[ll]_{i^{\ast}}  && \D^{''} \ar@<1ex>[ll]^{j_{\ast}}
\ar@<-1 ex>[ll]_{j_{!}} },$$ with $i_*=\Rh_A(T^\bullet ,?)$, for
some triangulated category
 $\D''$?

 \item Is there a recollement $$\xymatrix{\D^{'} \ar@<0ex>[rr]|{\hspace{0.1
cm}i_{\ast}=i_{!} \hspace{0.05 cm}}  &&  \D(B)
\ar@<0ex>[rr]|{\hspace{0.1cm}j^{!}=j^{\ast}} \ar@<1ex>[ll]^{i^{!}}
\ar@<-1 ex>[ll]_{i^{\ast}}  && \D(A) \ar@<1ex>[ll]^{j_{\ast}}
\ar@<-1 ex>[ll]_{j_{!}} },$$  whith $j_*=\Rh_A(T^\bullet ,?)$, for
some triangulated category $\D'$?

\end{enumerate}

\item Suppose that $?\Lt_BT^\bullet :\D(B)\longrightarrow\D(A)$ is
fully faithful.

\begin{enumerate}
\item Is there a recollement  $$\xymatrix{\D(B) \ar@<0ex>[rr]|{\hspace{0.1
cm}i_{\ast}=i_{!} \hspace{0.05 cm}}  &&  \D(A)
\ar@<0ex>[rr]|{\hspace{0.1cm}j^{!}=j^{\ast}} \ar@<1ex>[ll]^{i^{!}}
\ar@<-1 ex>[ll]_{i^{\ast}}  && \D^{''} \ar@<1ex>[ll]^{j_{\ast}}
\ar@<-1 ex>[ll]_{j_{!}} },$$ with $i_*=?\Lt_BT^\bullet$, for some
triangulated category
 $\D''$?

 \item Is there a recollement $$\xymatrix{\D^{'} \ar@<0ex>[rr]|{\hspace{0.1
cm}i_{\ast}=i_{!} \hspace{0.05 cm}}  &&  \D(A)
\ar@<0ex>[rr]|{\hspace{0.1cm}j^{!}=j^{\ast}} \ar@<1ex>[ll]^{i^{!}}
\ar@<-1 ex>[ll]_{i^{\ast}}  && \D(B) \ar@<1ex>[ll]^{j_{\ast}}
\ar@<-1 ex>[ll]_{j_{!}} },$$  whith $j_!=?\Lt_BT^\bullet$, for some
triangulated category $\D'$?
\end{enumerate}

\end{enumerate}
\end{quess}

The organization of the paper goes as follows. In section 2 we give
the preliminary results on triangulated categories  and the
corresponding terminology used in the paper. This part has been
prepared as an introductory material for beginners and, hence, tends
to  be as self-contained as possible.

Section 3 is specifically dedicated to the derived functors of
$\text{Hom}$ and the tensor product, but, due to the requirements of
some later proofs in the paper, the development is made for derived
categories of bimodules. In this context the material seems to be
unavailable in the literature. Special care is put on describing the
behavior of these derived functors when passing from the derived
categoy of bimodules to derived category of modules on one side. In
the final part of the section, we give a brief introduction to dg
algebras and give a generalization of Rickard theorem, in the case
of the derived category of a  $k$-flat dg algebra (Theorem
\ref{teor.Keller-Rickard generalization}).

Section 4 contains the main results in the paper. We first show that
the compact objects in the derived category of an algebra are
precisely those for which the associated derived tensor product
preserves products (proposition \ref{prop.tensor by compacts
preserves products}). We then go towards the mentioned goals of the
paper. Proposition \ref{prop.fully faithful RHom} gives a criterion
for the fully faithfulness of $\Rh_A(T^\bullet ,?)$, while
proposition \ref{prop.fully faithful derived tensor} gives criteria
for the fully faithfulness of $?\Lt_BT^\bullet$.  In a parallel way,
in corollary \ref{cor.recollement D(A)=D(B)=D by RHom}, theorem
\ref{teor.recollement D=D(B)=D(A) with Rh}, theorem
\ref{teor.recollement D(B)=D(A)=D with Lt} and corollary
\ref{cor.Chen-Xi} we give criteria for the existence of the
recollements mentioned in questions \ref{ques.recollement
situations} (1.a, 2.a, 2.b and 1.b, respectively). As a confluent
point, when $T^\bullet_A$ is exceptional in $\D(A)$ and the algebra
morphism $B\longrightarrow\text{End}_{\D(A)}(T^\bullet )$ is an
isomorphism, we show in theorem \ref{teor.main one} that $T^\bullet$
defines a recollement as in  question \ref{ques.recollement
situations}(1.b) if, only if, it defines a recollement as in
question \ref{ques.recollement situations}(2.b) on the derived
categories of left modules, and this is turn equivalent to saying
that $A$ is in the thick subcategory of $\D(A)$ generated by
$T^\bullet_A$.   We end the section by giving counterexamples to all
questions \ref{ques.recollement situations} and by proposing some
other   questions which remain open.

In the final section 5, we explicitly re-state some of the results
of section 4 in the particular case that $T^\bullet =T$ is just a
$B-A-$bimodule. One of the questions asked in section 4 asks whether
$\Rh_A(T^\bullet,?):\D(A)\longrightarrow\D(B)$ preserves compact
objects, when it is fully faithful,  $T^\bullet_A$ is exceptional in
$\D(A)$ and $B$ is isomorphic to $\text{End}_{\D(A)}(T^\bullet )$.
We end the paper by showing that, when $T^\bullet =T$ is a bimodule,
this question is related to the classical Wakamatsu tilting problem.

Some of our results are connected to recent results of
Bazzoni-Mantese-Tonolo \cite{BMT}, Bazzoni-Pavarin \cite{BP},
Chen-Xi (\cite{CX1}, \cite{CX2}), Han \cite{Ha} and Yang \cite{Y}.
All throughout sections 4 and 5, we give remarks showing these
connections.


\section{Preliminaries on triangulated categories and derived functors}

The results of this section are well-known, but they sometimes
appear scattered in the literature and with different notation. We
give them here for the convenience of the reader and, also, as a way
of unifying the terminology that we shall use throughout the paper.
Most of the material is an adaptation of Verdier's work (see
\cite{V}), but we will  refer also to several texts like \cite{N},
\cite{W}, \cite{H}, \cite{Ke1}, \cite{Kr1}..., for specific results
and proofs. As mentioned before, we will work over a fixed ground
commutative ring $k$. Then the term 'category' will mean always
'$k$-category'. If $\mathcal{C}$ is such a category, then the set of
morphisms $\mathcal{C}(C,D)$ has a structure of $k$-module, for all
$C,D\in\mathcal{D}$, and compositions of morphisms are $k$-bilinear.
Unless otherwise stated, subcategories will be always full and
closed under taking isomorphic objects.

The reader is referred to \cite[Chapter 10]{W} for the explicit
definition \emph{triangulated category}, although some of its
notation is changed. If $\D$ is such a category, then the
\emph{shifting}, also called \emph{suspension (or translation)
functor} $\D\longrightarrow\D$ will be denoted here by $?[1]$ and a
triangle in $\D$ will be denoted by $X\longrightarrow
Y\longrightarrow Z\stackrel{+}{\longrightarrow}$, although we will
also write $X\longrightarrow Y\longrightarrow
Z\stackrel{h}{\longrightarrow}X[1]$ when the connecting morphism $h$
need be emphasized. Recall that $Z$ is determined by the morphism
$f:X\longrightarrow Y$ up to non-unique isomorphism. We will call
$Z$ the \emph{cone of $f$}. A functor $F:\D\longrightarrow\D'$
between triangulated categories will be called a \emph{triangulated
or triangle-preserving functor} when it takes triangles to
triangules.

\subsection{The triangulated structure of the stable category of a Frobenius exact category}

An exact category (in the sense of Quillen) is an additive category
$\mathcal{C}$, together with a class of short exact sequences,
called \emph{conflations} or \emph{admissible short exact
sequences}. If  $0\rightarrow
C\stackrel{f}{\longrightarrow}D\stackrel{g}{\longrightarrow}E\rightarrow
0$ is a conflation, then we say that $f$ is an \emph{inflation} or
\emph{admissible monomorphism} and that $g$ is a \emph{deflation} or
\emph{admissible epimorphism}. The class of conflations must satisfy
the following axioms \underline{and their duals} (see \cite{Ke1} and
\cite{B}):

\begin{enumerate}
\item[Ex0] The identity morphism $1_X$ is a deflation, for each
$X\in\text{Ob}(C)$;
\item[Ex1] The composition of two deflations is a deflation;
\item[Ex2] Pullbacks of deflations along any morphism exist and are deflations.
\end{enumerate}

Obviously the dual of an exact category is an exact category with
the 'same' class of conflations. An object $P$ of an exact category
is called \emph{projective} when the functor
$\mathcal{C}(P,?):\mathcal{C}\longrightarrow k-\text{Mod}$ takes
deflations to epimorphisms. The dual notion is that of
\emph{injective} object. We say that $\mathcal{C}$ \emph{has enough
projectives (resp. injectives)} when, for each object
$C\in\mathcal{C}$, there is a deflation $P\longrightarrow C$ (resp.
inflation $C\longrightarrow I$), where $P$ (resp. $I$) is a
projective (resp. injective) object. The exact category
$\mathcal{C}$ is said to be \emph{Frobenius exact} when it has
enough projectives and enough injectives and the classes of
projective and injective objects coincide.

Given any exact Frobenius category $\mathcal{C}$, we can form its
\emph{stable category}, denoted $\underline{\mathcal{C}}$. Its
objects are those of $\mathcal{C}$, but we have
$\underline{\C}(C,D)=\frac{\mathcal{C}(C,D)}{\mathcal{P}(C,D)}$,
where $\mathcal{P}(C,D)$ is the $k$-submodule of $\mathcal{C}(C,D)$
consisting of those morphisms which factor through some
projective(=injective) object. The new category comes with a
\emph{projection functor}
$p_\mathcal{C}:\C\longrightarrow\underline{\C}$, which has the
property that each functor $F:\C\longrightarrow\D$ which vanishes on
the projective (=injective) objects factors through $p_\C$ in a
unique way.

The so-called \emph{(first) syzygy functor} $\Omega
:\underline{C}\longrightarrow\underline{\mathcal{C}}$  assigns to
each object $C$ the kernel of any deflation (=admissible
epimorphism) $\epsilon :P\twoheadrightarrow C$ in $\mathcal{C}$,
where $P$ is a projective object. Up to isomorphism in
$\underline{\mathcal{C}}$, the object $\Omega (C)$ does not depend
on the projective object $P$ or the deflation $\epsilon$. Moreover,
$\Omega$ is an equivalence of categories and its quasi-inverse is
called the \emph{(first) cosyzygy functor}
$\Omega^{-1}:\underline{C}\longrightarrow\underline{C}$.

If $0\rightarrow
C\stackrel{f}{\longrightarrow}D\stackrel{g}{\longrightarrow}E\rightarrow
0$ a conflation in the Frobenius exact category $\mathcal{C}$, then
we have  the following commutative diagram, where the rows are
conflations and $I$ is a projective(=injective) object:

\[
\xymatrix{
0\ar[r] & C\ar[r]^{f}\ar@{=}[d] & D\ar[r]^{g}\ar@{.>}[d] & E\ar[r]\ar[d]^{h} & 0 \\
0\ar[r] & C\ar[r] & I\ar[r] & \Omega^{-1}C\ar[r] & 0 }
\]

The following result is fundamental (see \cite[Section I.2]{H}):

\begin{prop}[Happel] \label{prop.Happel}
If $\mathcal{C}$ is a Frobenius exact category, then its stable
category admits a structure of triangulated category, where:

\begin{enumerate}
\item the suspension functor is the first cosyzygy functor;
\item the distinguished triangles are those isomorphic in
$\underline{\mathcal{C}}$ to a sequence of morphisms

\begin{center}
$C\stackrel{f}{\longrightarrow}D\stackrel{g}{\longrightarrow}E\stackrel{h}{\longrightarrow}\Omega^{-1}(C)$
\end{center}
coming from a commutative diagram in $\mathcal{C}$ as above.
\end{enumerate}
\end{prop}

\subsection{The Frobenius exact structure on the category of chain complexes. The homotopy category}

In the rest of this section  $\mathcal{A}$ will be an abelian
category. The \emph{graded category} associated to $\mathcal{A}$,
denoted $\mathcal{A}^{\mathbb{Z}}$ has as objects the
$\mathbb{Z}$-indexed families $X^\bullet:=(X^n)_{n\in\mathbb{Z}}$,
where $X_n\in\mathcal{A}$ for each $n\in\mathbb{Z}$. If $X^\bullet
,Y^\bullet\in\mathcal{A}^\mathbb{Z}$ then $\mathcal{A}^\mathbb{Z}$
consists of the families $f=(f^n)_{n\in\mathbb{Z}}$, where
$f^n\in\mathcal{A}(X^n,Y^n)$ for each $n\in\mathbb{Z}$. This
category is abelian and comes with a canonical self-equivalence
$?[1]:\mathcal{A}^\mathbb{Z}\longrightarrow\mathcal{A}^\mathbb{Z}$,
called the \emph{suspension} or \emph{(1-)shift},  given by the rule
$X^\bullet [1]^n=X^{n+1}$, for all $n\in\mathbb{Z}$. Keeping the
same class of objects, we can increase the class of morphisms and
form a new  category $\text{GR}\mathcal{A}$, where
$\text{GR}\mathcal{A}(X^\bullet ,Y^\bullet
)=\oplus_{n\in\mathbb{Z}}\mathcal{A}^\mathbb{Z}(X^\bullet,Y^\bullet
[n])$. This is obviously a graded category where a \emph{morphisms
of degree $n$} is just a morphism $X^\bullet\longrightarrow
Y^\bullet[n]$ in $\mathcal{A}^\mathbb{Z}$, and where $g\circ f$ is
the composition $X^\bullet\stackrel{f}{\longrightarrow}Y^\bullet
[m]\stackrel{g[m]}{\longrightarrow}Z[m+n]$ in
$\mathcal{A}^\mathbb{Z}$, whenever $f$ and $g$ are morphisms in
$GR\mathcal{A}$ of degrees $m$ and $n$.

A \emph{chain complex} of objects of $\mathcal{A}$ is a pair
$(X^\bullet,d)$ consisting of an object $X^\bullet$ of
$\text{GR}\mathcal{A}$ together with a morphism
$d:X^\bullet\longrightarrow X^\bullet$ in $GR\mathcal{A}$ of degree
$+1$ such that $d\circ d=0$. The \emph{category of chain complexes}
of objects of $\mathcal{A}$ will be denoted by
$\mathcal{C}(\mathcal{A})$. It has as objects the chain complexes of
objects of $\mathcal{A}$ and a morphism $f:X^\bullet\longrightarrow
Y^\bullet$, usually called a \emph{chain map},  will be just a
morphism of degree $0$ in $\text{GR}\mathcal{A}$ such that $f\circ
d=d\circ f$. The category $\mathcal{C}(\mathcal{A})$ is abelian,
with pointwise calculation of limits and colimits, and we have an
obvious forgetful functor
$\mathcal{C}(\mathcal{A})\longrightarrow\mathcal{A}^\mathbb{Z}$
which is faithful and dense, but not full. What is even more
important for us is that $\mathcal{C}(\mathcal{A})$ admits a
structure of Quillen exact category, usually called the
\emph{semi-split exact structure}. Recall that a short exact
sequence $0\rightarrow
X^\bullet\stackrel{f}{\longrightarrow}Y^\bullet\stackrel{g}{\longrightarrow}Z^\bullet\rightarrow
0$ in $\mathcal{C}(\mathcal{A})$ is called \emph{semi-split} when
its image by the forgetful functor
$\mathcal{C}(\mathcal{A})\longrightarrow\mathcal{A}^\mathbb{Z}$ is a
split exact sequence of $\mathcal{A}^\mathbb{Z}$. That is,  when the
sequence $0\rightarrow X^n\stackrel{f^n}{\longrightarrow}
Y^n\stackrel{g^n}{\longrightarrow}Z^n\rightarrow 0$ splits in
$\mathcal{A}$, for each $n\in\mathbf{Z}$.  For the semi-split exact
structure in $\mathcal{C}(\mathcal{A})$ the conflations  are
precisely the semi-split exact sequences. The suspension functor of
$\mathcal{A}^\mathbb{Z}$ induces a corresponding suspension functor
$?[1]:\mathcal{C}(\mathcal{A})\stackrel{\cong}{\longrightarrow}\mathcal{C}(\mathcal{A})$,
which is a self-equivalence and take $(X^\bullet
,d_X)\rightsquigarrow (X^\bullet [1],d_{X[1]}$, where $d_{X^\bullet
[1]}^n=-d_X^{n+1}$ for each $n\in\mathbb{Z}$.

It is well-known (see, e.g., \cite[Section I.3.2]{H}) that
$\mathcal{C}(\mathcal{A})$ is a \emph{Frobenius exact category} when
considered with this exact structure.  Its projective (=injective)
objects with respect to the conflations are precisely the
\emph{contractible complexes}. These are those complexes isomorphic
in $\mathcal{C}(\mathcal{A})$  to a coproduct of complexes of the
form $\text{C}_n(X):...0\longrightarrow
X\stackrel{1_X}{\longrightarrow}X\longrightarrow 0....$, where
$C_n(X)$ is concentrated in degrees $n-1,n$ for all $n\in\mathbb{Z}$
and all $X\in\mathcal{A}$. For each
$X^\bullet\in\mathcal{C}(\mathcal{A})$, we have a conflation

\begin{center}
$0\rightarrow X^\bullet
[-1]\longrightarrow\coprod_{n\in\mathbb{Z}}C_n(X_{n-1})\longrightarrow
X^\bullet\rightarrow 0$.
\end{center}
Then the syzygy functor of $\mathcal{C}(\mathcal{A})$ with respect
to the semi-split exact structure is identified with the inverse
$?[-1]:\mathcal{C}(\mathcal{A})\longrightarrow\mathcal{C}(\mathcal{A})$
of the suspension functor. The stable category of
$\mathcal{C}(\mathcal{A})$ has a very familiar description due to
the following result (see \cite[p. 28]{H}):

\begin{lemma} \label{lem.nullhomotopic=factors through contractible}
Let $(X^\bullet,d_X)$ and $(Y^\bullet ,d_Y)$ be chain complexes of
objects of $\mathcal{A}$ and let $f:X^\bullet\longrightarrow
Y^\bullet$ be a chain map. The following assertions are equivalent:

\begin{enumerate}
\item $f$ factors through a contractible complex;
\item $f$ is \emph{null-homotopic}, i.e., there exists a morphism  $\sigma :X^\bullet\longrightarrow
Y^\bullet$ of degree $-1$ in $\text{GR}\mathcal{A}$ such that
$f=\sigma\circ d_X+d_Y\circ\sigma$.
\end{enumerate}

\end{lemma}

As a consequence of the previous result, the morphisms in
$\mathcal{C}(\mathcal{A})$ which factor through a projective
(=injective) object, with respect to the semi-split exact structure,
are precisely the null-homotopic chain maps. As a consequence the
associated stable category $\underline{\mathcal{C}(\mathcal{A})}$ is
the \emph{homotopy category} of $\mathcal{A}$, denoted
$\mathcal{H}(\mathcal{A})$ in the sequel. We then get:

\begin{cor} \label{cor.triang. structure of homotopy category}
The homotopy category $\mathcal{H}(\mathcal{A})$ has a structure of
triangulated category such that

\begin{enumerate}
\item the suspension functor
$?[1]:\mathcal{H}(\mathcal{A})\longrightarrow\mathcal{H}(\mathcal{A})$
is induced from the suspension functor of
$\mathcal{C}(\mathcal{A})$;
\item each distinguished triangle in $\mathcal{H}(\mathcal{A})$
comes from a semi-split exact sequence $0\rightarrow
X^\bullet\stackrel{f}{\longrightarrow}Y^\bullet\stackrel{g}{\longrightarrow}Z^\bullet\rightarrow
0$ in $\mathcal{C}(\mathcal{A})$ in the form described in
proposition \ref{prop.Happel}.
\end{enumerate}
\end{cor}

Recall that if $(\mathcal{D},?[1])$ is a triangulated category and
$\mathcal{A}$ is an abelian category, then a \emph{cohomological
functor} $H:\mathcal{D}\longrightarrow\mathcal{A}$ is an additive
functor such that if

\begin{center}
$X\stackrel{f}{\longrightarrow}Y\stackrel{g}{\longrightarrow}Z\stackrel{h}{\longrightarrow}X[1]$
\end{center}
is any triangle in $\mathcal{D}$ and one puts $H^k=H\circ (?[k])$,
for each $k\in\mathbb{Z}$, then we get a long exact sequence

\begin{center}
$...H^{k-1}(Z)\stackrel{H^{k-1}(h)}{\longrightarrow}H^k(X)\stackrel{H^k(f)}{\longrightarrow}H^k(Y)\stackrel{H^k(g)}{\longrightarrow}H^k(Z)\stackrel{H^k(h)}{\longrightarrow}
H^{k+1}(X)\stackrel{H^{k+1}(f)}{\longrightarrow}H^{k+1}(Y)...$
\end{center}

Recall that if $X^\bullet$ is a chain complex and $k\in\mathbb{Z}$,
then the \emph{$k$-th object of homology of $X$} is $H^k(X^\bullet
)=\frac{\text{Ker}(d^k)}{\text{Im}(d^{k-1})}$, where
$d:X^\bullet\longrightarrow X^\bullet$ is the differential. The
assignment $X^\bullet\rightsquigarrow H^k(X^\bullet )$ is the
definition on objects of a functor
$H^k:\mathcal{C}(\mathcal{A})\longrightarrow\mathcal{A}$. The
following is well-known:

\begin{cor} \label{cor.0-homology functor}
Let $\mathcal{A}$ be any abelian  category,
$p:\mathcal{C}(\mathcal{A})\longrightarrow\mathcal{H}(\mathcal{A})$
the projection functor and $k\in\mathbb{Z}$ be an integer. The
functor $H^k:\mathcal{C}(\mathcal{A})\longrightarrow\mathcal{A}$
vanishes on null-homotopic chain maps and there is a unique
$k$-linear functor
$\bar{H}^k:\mathcal{H}(\mathcal{A})\longrightarrow\mathcal{A}$ such
that $\bar{H}^k\circ p=H^k$. Moreover, $\bar{H}^k$ is a
cohomological functor.
\end{cor}

\begin{rem}
We will forget the overlining of $H$ and will still denote by $H^k$
the functor $\bar{H}^k:\K(\mathcal{A})\longrightarrow\mathcal{A}$.
\end{rem}

\subsection{Localization of triangulated categories}

As in the previous sections, the symbol $\mathcal{A}$ will denote a
fixed abelian category. We will use the term \emph{big category} to
denote a concept
 defined as an usual category, but where we do not require that the
 morphsims between two objects form a set.  The following is well-known
 (see \cite[Chapter 1]{GZ}):

 \begin{prop} \label{prop.Gabriel-Zisman}
 Given a category $\mathcal{C}$ and a class $\mathcal{S}$ of
 morphisms in $\mathcal{C}$,  there is a big category
 $\mathcal{C}[\mathcal{S}^{-1}]$, together with a dense functor
 $q:\mathcal{C}\longrightarrow\mathcal{C}[\mathcal{S}^{-1}]$
 satisfying the following properties:

 \begin{enumerate}
 \item $q(s)$ is an isomorphism,  for each $s\in\mathcal{S}$;
 \item if $F:\mathcal{C}\longrightarrow\mathcal{D}$ is any functor
 between categories such that $F(s)$ is an isomorphism for each
 $s\in\mathcal{S}$, then there is a unique functor
 $\bar{F}:\mathcal{C}[\mathcal{S}^{-1}]\longrightarrow\mathcal{D}$ such that $\bar{F}\circ
 q=F$.
 \end{enumerate}
 \end{prop}

\begin{defi} \label{def.localization of a category}
$\mathcal{C}[\mathcal{S}^{-1}]$ is called the \emph{localization of
$\mathcal{C}$ with respect to $\mathcal{S}$}.
\end{defi}

\begin{rems} \label{rem.big category-enough for usual category}
\begin{enumerate}
\item The pair $(\mathcal{C}[\mathcal{S}^{-1}],q)$ is
uniquely determined up to equivalence.
\item A sufficient condition to guarantee that
$\mathcal{C}[\mathcal{S}^{-1}]$ is an usual category is that the
functor $q:\mathcal{C}\longrightarrow\mathcal{C}[\mathcal{S}^{-1}]$
has a left or a right adjoint. If, say,
$R:\mathcal{C}[\mathcal{S}^{-1}]\longrightarrow\mathcal{C}$ is right
adjoint to $q$, then we have a bijection
$\mathcal{C}[\mathcal{S}^{-1}](q(X),q(Y))\cong\mathcal{C}(X,Rq(Y))$,
for all $X,Y\in\mathcal{C}$. This proves that the morphisms between
two objects of $\mathcal{C}[\mathcal{S}^{-1}]$ form a set since $q$
is dense. A dual argument works in case $q$ has a left adjoint.
\end{enumerate}
\end{rems}

The explicit definition of $\mathcal{C}[\mathcal{S}^{-1}]$ was given
in \cite[Chapter 1]{GZ}, but the morphisms in this category are
intractable in general. Then Gabriel and Zisman introduced some
condition on $\mathcal{S}$ which makes much more tractable the
morphisms in $\mathcal{C}[\mathcal{S}^{-1}]$.

\begin{defi} \label{def.calculus of fractions}
We shall say that $\mathcal{S}$ \emph{admits a calculus of left
fractions} when it satisfies the following conditions:

\begin{enumerate}
\item $1_X\in\mathcal{S}$, for all $X\in\mathcal{C}$;
\item for each diagram
$X'\stackrel{f'}{\longrightarrow}Y'\stackrel{s'}{\longleftarrow}Y$
in $\mathcal{C}$, with $s'\in\mathcal{S}$, there exists a diagram
$X'\stackrel{s}{\longleftarrow}X\stackrel{f}{\longrightarrow}Y$ such
that $s\in\mathcal{S}$ and $f'\circ s=s'\circ f$.
\item if $f,g:X\longrightarrow Y$ are morphisms in $\mathcal{C}$ and
there exists $t\in\mathcal{S}$ such that $t\circ f=t\circ g$, then
there exists $s\in\mathcal{S}$ such that $f\circ s=g\circ s$.
\end{enumerate}
We say that $\mathcal{S}$ \emph{admits a calculus of right
fractions} when it satisfies the duals of properties 1-3 above.
Finally, we will say that $\mathcal{S}$ \emph{admits a calculus of
fractions}, or that $\mathcal{S}$ is a \emph{multiplicative system
of morphisms}, when it admits both a calculus of left fractions and
a calculus of right fractions.
\end{defi}

When $\mathcal{S}$ admits a calculus of left fractions, the
morphisms in $\mathcal{C}[\mathcal{S}^{-1}]$ have a more tractable
form. Indeed, if
$X,Y\in\text{Ob}(C)=\text{Ob}(\mathcal{C}[\mathcal{S}^{-1}])$, then
$\mathcal{C}[\mathcal{S}^{-1}](X,Y)$ consists of the formal left
fractions $s^{-1}f$. Such a formal left fraction is the equivalence
class of the pair $(s,f)$ with respect to some equivalence relation
defined in the class of diagrams
$X\stackrel{s}{\longleftarrow}X'\stackrel{f}{\longrightarrow}Y$,
with $s\in\mathcal{S}$.  We refer the reader to \cite[Chapter 1]{GZ}
for the precise definition of the equivalence relation and the
composition of morphisms in $\mathcal{C}[\mathcal{S}^{-1}]$.

The process of localizing a category with respect to a class of
morphism was developed in the context of triangulated categories by
Verdier (see \cite[Section II.2]{V}).

\begin{defi} \label{def.multipl.system compatible with triang.}
Let $\mathcal{D}$ be a triangulated category. A multiplicative
system $\mathcal{S}$ in $\mathcal{D}$ is said to be \emph{compatible
with the triangulation} when the following properties hold:

\begin{enumerate}
\item if $s:X\longrightarrow Y$ is a morphism in $\mathcal{S}$ and
$n\in\mathbb{Z}$, then $s[n]:X[n]\longrightarrow Y[n]$ is in
$\mathcal{S}$;
\item each commutative diagram in $\mathcal{D}$

\[
\xymatrix{
X\ar[r]^{f}\ar[d]^{s} & Y\ar[r]^{g}\ar[d]^{s'}&Z\ar[r]^{h}\ar[d]^{s''}&X[1]\ar[d]^{s[1]} \\
X'\ar[r]^{f'}&Y'\ar[r]^{g'}&Z'\ar[r]^{h'}& X'[1] }
\]
\noindent where the rows are triangles and where
$s,s'\in\mathcal{S}$,  can be completed commutatively by an arrow
$s''$ which is in $\mathcal{S}$.
\end{enumerate}
\end{defi}

As we can expect, one gets:

\begin{prop}[Verdier, Th\'eor\`eme 2.2.6] \label{prop.Verdier}
Let $\mathcal{D}$ be a triangulated category and $\mathcal{S}$ be a
multiplicative system compatible with the triangulation in
$\mathcal{D}$. There exists a unique structure of triangulated
category in $\mathcal{D}[\mathcal{S}^{-1}]$ such that the canonical
functor $q:\mathcal{D}\longrightarrow\mathcal{D}[\mathcal{S}^{-1}]$
is triangulated. Moreover, if $\mathcal{A}$ is an abelian category
and $H:\mathcal{D}\longrightarrow\mathcal{A}$ is a cohomological
functor such that $H(s)$ is an isomorphism, for each
$s\in\mathcal{S}$, then there is a unique cohomological funtor
$\tilde{H}:\mathcal{D}[\mathcal{S}^{-1}]\longrightarrow\mathcal{A}$
such that $\tilde{H}\circ q=H$.
\end{prop}

\subsection{Semi-orthogonal decompositions. Brown representability theorem}

Let $\mathcal{D}$ be a triangulated category. A \emph{t-structure}
in $\mathcal{D}$ (see \cite{BBD}) is a pair
$(\mathcal{U},\mathcal{W})$ of full subcategories which satisfy the
 following  properties:

\begin{enumerate}
\item[i)] $\mathcal{D}(U,W[-1])=0$, for all
$U\in\mathcal{U}$ and $W\in\mathcal{W}$;
\item[ii)] $\mathcal{U}[1]\subseteq\mathcal{U}$;
\item[iii)] For each $X\in Ob(\mathcal{D})$, there is a triangle $U\longrightarrow X\longrightarrow
V\stackrel{+}{\longrightarrow}$ in $\mathcal{D}$, where
$U\in\mathcal{U}$ and $V\in\mathcal{W}[-1]$.
\end{enumerate}
It is easy to see that in such case $\mathcal{W}=\mathcal{U}^\perp
[1]$ and $\mathcal{U}={}^\perp (\mathcal{W}[-1])={}^\perp
(\mathcal{U}^\perp )$. For this reason, we will write a t-structure
as $(\mathcal{U},\mathcal{U}^\perp [1])$. We will call $\mathcal{U}$
and $\mathcal{U}^\perp$ the \emph{aisle} and the \emph{co-aisle} of
the t-structure, respectively. The objects $U$ and $V$ in the above
triangle are uniquely determined by $X$, up to isomorphism, and
define functors
$\tau_\mathcal{U}:\mathcal{D}\longrightarrow\mathcal{U}$ and
$\tau^{\mathcal{U}^\perp}:\mathcal{D}\longrightarrow\mathcal{U}^\perp$
which are right and left adjoints to the respective inclusion
functors. We call them the \emph{left and right truncation functors}
with respect to the given t-structure.

Recall that a subcategory $\mathcal{X}$ of a triangulated category
$\D$ is \emph{closed under extensions} when, given any triangle
$X'\longrightarrow Y\longrightarrow X\stackrel{+}{\longrightarrow}$
in $\D$, with $X,X'\in\mathcal{X}$, the object $Y$ is also in
$\mathcal{X}$.

\begin{defi}
Let $\D$ be a triangulated category and $\mathcal{U}\subseteq\D$ a
full subcategory. We say that $\mathcal{U}$ is

\begin{enumerate}
\item \emph{suspended} when it is closed under extensions in $\D$ and $\mathcal{U}[1]\subseteq\mathcal{U}$;
\item \emph{triangulated} when it is closed under extensions  and $\mathcal{U}[1]=\mathcal{U}$
\item \emph{thick} when it is triangulated and closed under taking
direct summands.
\end{enumerate}
\end{defi}

{\bf Notation and terminology.-} Given a class $\mathcal{S}$ of
objects of $\D$, we shall denote by $\text{susp}_\D(\mathcal{S})$
(resp. $\text{tria}_\D(\mathcal{S})$, resp.
$\text{thick}_\D(\mathcal{S})$) the smallest suspended (resp.
triangulated, resp. thick) subcategory of $\D$ containing
$\mathcal{S}$. If $\mathcal{U}$ is any suspended (resp.
triangulated, resp. thick) subcategory of $\D$ and
$\mathcal{U}=\text{susp}_\D(\mathcal{S})$ (resp.
$\mathcal{U}=\text{tria}_\D(\mathcal{S})$, resp.
$\text{thick}_\D(\mathcal{S})$), we will say that $\mathcal{U}$ is
the \emph{suspended (resp. triangulated, resp. thick) subcategory of
$\D$ generated by $\mathcal{S}$}. When $\D$ has coproducts, we will
denote by $\text{Susp}_\D(\mathcal{S})$ (resp.
$\text{Tria}_\D(\mathcal{S})$) the smallest suspended (resp.
triangulated) subcategory of $D$ containing $\mathcal{S}$ which is
closed under taking coproducts in $\D$.

\begin{rem}
If $\D$ has coproducts and $\mathcal{T}$ is a full subcategory
closed under taking coproducts, then it is a triangulated
subcategory if, and only if, it is thick (see \cite[Prop.1.6.8 and
its proof]{N}).
\end{rem}

\vspace*{0.3cm}

It is an easy exercise to prove now the following useful result.

\begin{lemma} \label{lem.triang.functors preserves susp, tria..}
Let $F:\D\longrightarrow\D'$ be a triangulated functor between
triangulated categories. If $\mathcal{S}\subseteq\D$ is any class of
objects, then
$F(\text{susp}_\D(\mathcal{S}))\subseteq\text{susp}_{\D'}(F(\mathcal{S}))$
(resp.
$F(\text{tria}_\D(\mathcal{S}))\subseteq\text{tria}_{\D'}(F(\mathcal{S}))$,
resp.
$F(\text{thick}_\D(\mathcal{S}))\subseteq\text{thick}_{\D'}(F(\mathcal{S}))$).
When $\D$ and $\D'$ have coproducts and $F$ preserves coproducts, we
also have that
$F(\text{Susp}_\D(\mathcal{S}))\subseteq\text{Susp}_{\D'}(F(\mathcal{S}))$
(resp.
$F(\text{Tria}_\D(\mathcal{S}))\subseteq\text{Tria}_{\D'}(F(\mathcal{S}))$).
\end{lemma}

The following result of Keller and Vossieck \cite[Proposition
1.1]{KV} is fundamental to deal with t-structures and
semi-orthogonal decompositions.

\begin{prop} \label{prop.Keller-Vossieck}
A full subcategory $\mathcal{U}$ of $\D$ is the aisle of a
t-structure if, and only if, it is a suspended subcategory such that
the inclusion functor $\mathcal{U}\hookrightarrow\D$ has a right
adjoint.
\end{prop}

The type of t-structure which is most useful to us in this paper is
the following.

\begin{defi}
A \emph{semi-orthogonal decomposition} or \emph{Bousfield
localization pair} in $\D$ is a t-structure
$(\mathcal{U},\mathcal{U}^\perp [1])$ such that
$\mathcal{U}[1]=\mathcal{U}$ (equivalently, such that
$\mathcal{U}^\perp =\mathcal{U}^\perp [1]$). That is, a
semi-orthogonal decomposition is a t-structure such that
$\mathcal{U}$ (resp. $\mathcal{U}^\perp$) is a triangulated
subcategory of $\D$. In such case we will use
$(\mathcal{U},\mathcal{U}^\perp )$ instead of
$(\mathcal{U},\mathcal{U}^\perp [1])$ to denote the semi-orthogonal
decomposition.
\end{defi}

Certain adjunctions of triangulated functors provide semi-orthogonal
decompositions.

\begin{prop} \label{prop.semi-orthogonal pair from adjunction}
Let $F:\D\longrightarrow\D'$ and $G:\D'\longrightarrow\D$ be
triangulated functors between triangulated categories such that
$(F,G)$ is an adjoint pair. The following assertions hold:

\begin{enumerate}
\item If $F$ is fully faithful, then $(\text{Im}(F),\text{Ker}(G))$
is a semi-orthogonal decomposition of $\D'$;
\item If $G$ is fully faitfhful, then $(\text{Ker}(F),\text{Im}(G))$
is a semi-orthogonal decomposition of $\D$.
\end{enumerate}
\end{prop}
\begin{proof}
Assertion (2) follows from assertion (1) by duality. To prove (1), note
that the unit $\lambda :1_\D\longrightarrow G\circ F$ is an
isomorphism (see \cite[Proposition II.7.5]{HS}) and, by the
adjunction equations, we then get that $G(\delta )$ is also an
isomorphism, where $\delta :F\circ G\longrightarrow 1_{\D'}$ is the
counit. This implies that if $M\in\D'$ is any object and we complete
$\delta_M$ to a triangle

\begin{center}
$(F\circ G)(M)\stackrel{\delta_M}{\longrightarrow}M\longrightarrow
Y_M\stackrel{+}{\longrightarrow}$,
\end{center}
then $Y_M\in\text{Ker}(G)$.

But, by the adjunction, we have that $\D'(F(D),Y)=\D(D,G(Y))=0$, for
each $Y\in\text{Ker}(G)$. It follows that
$(\text{Im}(F),\text{Ker}(G))$ is a semi-orthogonal decomposition of
$\D'$.
\end{proof}

Given a triangulated subcategory $\mathcal{T}$ of $\D$, we shall
denote by $\Sigma_\mathcal{T}$ the class of morphisms
$s:X\longrightarrow Y$ in $\D$ whose cone is an object of
$\mathcal{T}$. The following is a fundamental result of Verdier:

\begin{prop} \label{prop.Verdier localization by triangulated}
Let $\D$ be a triangulated category and $\mathcal{T}$ a thick
subcategory. The following assertions hold:

\begin{enumerate}
\item $\Sigma_\mathcal{T}$ is a multiplicative system of $\D$
compatible with the triangulation. The category
$\D[\Sigma_\mathcal{T}^{-1}]$ is denoted by $\D/\mathcal{T}$ and
called the \emph{quotient category of $\D$ by $\mathcal{T}$}.

\item The canonical functor $q:\D\longrightarrow\D/\mathcal{T}$
satisfies the following universal property:

(*) For each triangulated category $\D'$ and each triangulated
functor $F:\D\longrightarrow\D'$ such that $F(\mathcal{T})=0$, there
is a triangulated functor
$\bar{F}:\D/\mathcal{T}\longrightarrow\D'$, unique up to natural
isomorphism,  such that $\bar{F}\circ q\cong F$.

\item The functor $q:\D\longrightarrow\D/\mathcal{T}$ has a right
adjoint if, and only if, $(\mathcal{T},\mathcal{T}^{\perp} )$ is a
semi-orthogonal decomposition in $\D$. In this case, the functor
$\tau^{\mathcal{T}^\perp }:D\longrightarrow\mathcal{T}^\perp$
induces an equivalence of triangulated categories
$\D/\mathcal{T}\stackrel{\cong}{\longrightarrow}\mathcal{T}^\perp$.

\item The functor $q:\D\longrightarrow\D/\mathcal{T}$ has a left
adjoint if, and only if, $({}^\perp\mathcal{T},\mathcal{T})$ is a
semi-orthogonal decomposition in $\D$. In this case, the functor
$\tau_{{}^\perp\mathcal{T}}:D\longrightarrow{}^\perp\mathcal{T}$
induces an equivalence of triangulated categories
$\D/\mathcal{T}\stackrel{\cong}{\longrightarrow}{}^\perp\mathcal{T}$.
\end{enumerate}
\end{prop}
\begin{proof}
Assertions (1) is  \cite[Proposition II.2.1.8]{V} while assertion (2) is
included in \cite[Corollaire II.2.2.11]{V}. Assertions (3) and (4) are
dual to each other. Assertion (3) is implicit in \cite[Proposition
II.2.3.3]{V}. For an explicit proof, see \cite[Proposition
4.9.1]{Kr1} and  use  proposition \ref{prop.Keller-Vossieck}.
\end{proof}

In many  situations, we will need criteria for a given triangulated
functor to have adjoints. The main tool is the following.

\begin{defi} \label{def.Brown repres. theorem}
Let $\mathcal{D}$ be a triangulated category with coproducts. We
shall say that $\D$ \emph{satisfies Brown representability theorem}
when any cohomological contravariant functor
$H:\D\longrightarrow\text{Mod}-k$ which takes coproducts to products
is representable. That is, there exists an object $Y$ of $\D$ such
that $H$ is naturally isomorphic to $\D(?,Y)$.
\end{defi}

The key point is the following.

\begin{prop} \label{prop.Brown representability theorem}
Let $\D$ be a triangulated subcategory which satisfies Brown
representability theorem and let $\D'$ be any triangulated category.
Each triangulated functor $F:\D\longrightarrow\D'$ which preserves
coproducts has a right adjoint.
\end{prop}
\begin{proof}
See \cite[Theorem 8.8.4]{N}.
\end{proof}

\begin{defi} \label{def.compact and compactly generated}
Let $\D$ have coproducts. An object $X$ of $\D$ is called
\emph{compact} when the functor
$\D(X,?):\D\longrightarrow\text{Mod}-k$ preserves coproducts. The
category $\D$ is called \emph{compactly generated} when there is a
\underline{set} $\mathcal{S}$ of compact objects such that
$\text{Tria}_\D(\mathcal{S})=\D$. We then say that $\mathcal{S}$ is
a \emph{set of compact generators} of $\D$.
\end{defi}

\begin{cor} \label{cor.adjoints with comp.generated domain}
The following assertions hold, for any compactly generated
triangulated category $\D$ and any covariant triangulated functor
$F:\D\longrightarrow\D'$:

\begin{enumerate}
\item $F$ preserves coproducts if, and only if, it has a right
adjoint;
\item $F$ preserves products if, and only if, it has a left adjoint.
\end{enumerate}
\end{cor}
\begin{proof}
See \cite[Proposition 5.3.1]{Kr1}.
\end{proof}

The following lemma, whose proof can be found in \cite{NS}, is
rather useful.

\begin{lemma} \label{lem.fully-faitful right iff left}
Let $F:\D\longrightarrow\D'$ be a triangulated functor between
triangulated categories and suppose that it has a left adjoint $L$
and a right adjoint $R$. Then $L$ is fully faithful if, and only if,
so is $R$.
\end{lemma}

\subsection{Recollements and TTF triples}

The following concept, introduced in \cite{BBD}, is fundamental in
the theory of triangulated categories.

\begin{defi} \label{def.recollement}
A \emph{recollement} of triangulated categories consists of a triple
$(\D' ,\D,\D'')$ of triangulated categories and of six triangulated
functors between them, assembled as follows
$$\xymatrix{\D^{'} \ar@<0ex>[rr]|{\hspace{0.1 cm}i_{\ast}=i_{!} \hspace{0.05 cm}}  &&  \D  \ar@<0ex>[rr]|{\hspace{0.1cm}j^{!}=j^{\ast}} \ar@<1ex>[ll]^{i^{!}} \ar@<-1 ex>[ll]_{i^{\ast}}  && \D^{''} \ar@<1ex>[ll]^{j_{\ast}} \ar@<-1 ex>[ll]_{j_{!}} },$$
which satisfy the following conditons:

\begin{enumerate}
\item Each functor in the picture is left adjoint to the one
immediately below, when it exists;
\item The composition $i^!\circ j_*$ (and hence also  $j^!\circ i_!=j^*\circ i_*$ and $i^*\circ
j_!$) is the zero functor;
\item The functors $i_*$, $j_*$ and $j_!$ are fully faithful (and
hence the unit maps $1_{\D'}\longrightarrow i^!\circ i_*=i^!\circ
i_!$, $1_{\D''}\longrightarrow j^*\circ j_!=j^!\circ j_!$ and the
counit maps $i^*\circ i_!=i^*\circ i_*\longrightarrow 1_{\D'}$,
$j^!\circ j_*=j^*\circ j_*\longrightarrow 1_{\D''}$ are all
isomorphisms);
\item The remaining unit and counit maps of the different
adjunctions give rise, for each object $X\in\D$, to triangules

\begin{center}
$(i_!\circ i^!)(X)\longrightarrow X\longrightarrow (j_*\circ
j^*)(X)\stackrel{+}{\longrightarrow}$

and

$(j_!\circ j^!)(X)\longrightarrow X\longrightarrow (i_*\circ
i^*)(X)\stackrel{+}{\longrightarrow}$.
\end{center}
\end{enumerate}

In such situation, we will say that \emph{$\D$ is a recollement of
$\D'$ and $\D''$}.
\end{defi}

A less familiar concept, coming from torsion theory in module
categories, is the following (see \cite{NS2}):

\begin{defi} \label{def.TTF triple}
Given a triangulated category $\D$, a triple
$(\mathcal{X},\mathcal{Y},\mathcal{Z})$ of full subcategories is
called a \emph{TTF triple} in $\D$ when the pairs
$(\mathcal{X},\mathcal{Y})$ and $(\mathcal{Y},\mathcal{Z})$ are both
semi-orthogonal decompositions of $\D$.
\end{defi}

As shown in \cite[Section 2.1]{NS2}, it turns out that recollements
and TTF triples are equivalent concepts in the following sense:

\begin{prop} \label{prop.TTF=recollement}
Let $\D$ be a triangulated category. The following assertions hold:

\begin{enumerate}
\item If $$\xymatrix{\D^{'} \ar@<0ex>[rr]|{\hspace{0.1 cm}i_{\ast}=i_{!} \hspace{0.05 cm}}  &&  \D  \ar@<0ex>[rr]|{\hspace{0.1cm}j^{!}=j^{\ast}} \ar@<1ex>[ll]^{i^{!}} \ar@<-1 ex>[ll]_{i^{\ast}}  && \D^{''} \ar@<1ex>[ll]^{j_{\ast}} \ar@<-1 ex>[ll]_{j_{!}} }$$
is a recollement of the triangulated category $\D$, then
$(\text{Im}(j_!),\text{Im}(i_*),\text{Im}(j_*))$ is a TTF triple in
$\D$;
\item If $(\mathcal{X},\mathcal{Y},\mathcal{Z})$ is a TTF triple in
$\D$, then

$$\xymatrix{\mathcal{Y} \ar@<0ex>[rr]|{\hspace{0.1 cm}i_{\ast}=i_{!} \hspace{0.05 cm}}  &&  \D  \ar@<0ex>[rr]|{\hspace{0.1cm}j^{!}=j^{\ast}} \ar@<1ex>[ll]^{i^{!}} \ar@<-1 ex>[ll]_{i^{\ast}}  && \mathcal{X} \ar@<1ex>[ll]^{j_{\ast}} \ar@<-1 ex>[ll]_{j_{!}} }$$
 is a recollement, where:

\begin{enumerate}
\item $i_*=i_!:\mathcal{Y}\longrightarrow\D$ and
$j_!:\mathcal{X}\longrightarrow\D$ are the inclusion functors;
\item $i^*=\tau^\mathcal{Y}:\D\longrightarrow\mathcal{Y}$ is the
right truncation with respect to the semi-orthogonal decomposition
$(\mathcal{X},\mathcal{Y})$ and
$i^!=\tau_\mathcal{Y}:\D\longrightarrow\mathcal{Y}$ is the left
truncation with respect to the semi-orthogonal decomposition
$(\mathcal{Y},\mathcal{Z})$;
\item $j^!=j^*=\tau_\mathcal{X}:\D\longrightarrow\mathcal{X}$ is the
left truncation with respect to the semi-orthogonal decomposition
$(\mathcal{X},\mathcal{Y})$;
\item $j_*$ is the composition
$\mathcal{X}\hookrightarrow\D\stackrel{\tau^\mathcal{Z}}{\longrightarrow}\mathcal{Z}\hookrightarrow\D$,
where the hooked arrows are the inclusions and $\tau^\mathcal{Z}$ is
the right truncation with respect to the semi-orthogonal
decomposition $(\mathcal{Y},\mathcal{Z})$.
\end{enumerate}
\end{enumerate}
\end{prop}

\begin{rem} \label{rem.notation for recollements}
In the rest of the paper, whenever
$$\xymatrix{\D^{'} \ar@<0ex>[rr]|{\hspace{0.1 cm}i_{\ast}=i_{!}
\hspace{0.05 cm}}  &&  \D
\ar@<0ex>[rr]|{\hspace{0.1cm}j^{!}=j^{\ast}} \ar@<1ex>[ll]^{i^{!}}
\ar@<-1 ex>[ll]_{i^{\ast}}  && \D^{''} \ar@<1ex>[ll]^{j_{\ast}}
\ar@<-1 ex>[ll]_{j_{!}} }$$ is a recollement of triangulated
categories, we will simply write $\D'\equiv\D\equiv\D''$ and the six
functors of the recollement will be understood.
\end{rem}

We now give a criterion for a triangulated functor to be one of the
two central arrows of a recollement.

\begin{prop} \label{prop.functor in a recollement}
The following assertions hold:

\begin{enumerate}
\item Let $F:\D'\longrightarrow\D$ be a triangulated functor between
triangulated categories. There is a recollement
$\D'\equiv\D\equiv\D''$, with $F=i_*$, for some triangulated
category $\D''$, if and only if $F$ is fully faithful and has both a
left and a right adjoint;
\item Let $G:\D\longrightarrow\D''$ be a triangulated functor
between triangulated categories. There is a recollement
$\D'\equiv\D\equiv\D''$, with $G=j^*$, for some triangulated
category $\D'$, if and only if $G$ has both a left and a right
adjoint, and one of these adjoints is fully faithful.
\end{enumerate}
\end{prop}
\begin{proof}
(1) We just need to prove the 'if' part of the assertion. If $F$ is
fully faithful, then it induces an equivalence of categories
$\D'\stackrel{\cong}{\longrightarrow}\text{Im}(F)=:\mathcal{Y}$. The
fact that $F$ has both a left and a right adjoint implies that also
the inclusion functor $i_\mathcal{Y}:\mathcal{Y}\hookrightarrow\D$
has both a left and a right adjoint. By proposition
\ref{prop.Keller-Vossieck} and its dual, we  get that
$(\mathcal{Y},\mathcal{Y}^\perp)$ and
$({}^\perp\mathcal{Y},\mathcal{Y})$ are semi-orthogonal
decompositions of $\D$.  Therefore
$({}^\perp\mathcal{Y},\mathcal{Y},\mathcal{Y}^\perp )$ is a TTF
triple in $\D$ and, by proposition \ref{prop.TTF=recollement}, we
have a recollement $\mathcal{Y}\equiv\D\equiv {}^\perp\mathcal{Y}$,
with $i_*=i_\mathcal{Y}:\mathcal{Y}\hookrightarrow\D$ the inclusion
functor. Using now the equivalence
$\D'\stackrel{\cong}{\longrightarrow}\mathcal{Y}$ given by $F$, we
immediately get a recollement
$\D'\equiv\D\equiv{}^\perp\mathcal{Y}$, with $i_*=F$.

(2) Again, we just need to prove the 'if' part. If
$G:\D\longrightarrow\D''$ is a triangulated functor as stated and we
denote by $L$ and $R$ its left and right adjoint, respectively, then
lemma \ref{lem.fully-faitful right iff left} tells us that $L$ and
$R$ are both fully faithful. Then, by proposition
\ref{prop.semi-orthogonal pair from adjunction}, we get that
$(\text{Im}(L),\text{Ker}(G),\text{Im}(R))$ is a TTF triple in $\D$.

Since $L$ gives an equivalence of triangulated categories
$\tilde{L}:\D''\stackrel{\cong}{\longrightarrow}\text{Im}(L)=\mathcal{X}$,
we easily get that the left truncation functor $\tau_\mathcal{X}$
with respect to the semi-orthogonal decomposition
$(\mathcal{X},\mathcal{X}^\perp )=(\text{Im}(L),\text{Ker}(G))$ is
naturally isomorphic to $\tilde{L}\circ G$. Using proposition
\ref{prop.TTF=recollement}, we then get a recollement
$\text{Ker}(G)\equiv\D\equiv D''$, where $j^*=G$.

\end{proof}

\subsection{The derived category of an abelian category}
In this subsection $\mathcal{A}$ will be an abelian category. A
morphism $f:X^\bullet\longrightarrow Y^\bullet$ in
$\mathcal{C}(\mathcal{A})$ is called a \emph{quasi-isomorphism} when
the morphism $H^k(f):H^k(X^\bullet )\longrightarrow H^k(Y^\bullet )$
is an isomorphism in $\mathcal{A}$, for each $k\in\mathbb{Z}$. Our
main object of interest is the following category.

\begin{defi}
The \emph{derived category of $\mathcal{A}$}, denoted
$\mathcal{D}(\mathcal{A})$,  is the localization of
$\mathcal{C}(\mathcal{A})$ with respect to the class of
quasi-isomorphisms.
\end{defi}

Note that, in general,  $\mathcal{D}(\mathcal{A})$ is a big
category. Moreover, defined as above, we have the problem of the
intractability of its morphisms. But, fortunately,  this latter
obstacle is overcome:

\begin{prop}[Verdier] \label{prop.derived category}
Let $\mathcal{Q}$ be the class of quasi-isomoprhisms in
$\mathcal{C}(\mathcal{A})$. The following assertions hold:

\begin{enumerate}
\item The canonical functor
$q:\mathcal{C}(\mathcal{A})\longrightarrow\mathcal{D}(\mathcal{A})$
factors through the projection functor
$p:\mathcal{C}(\mathcal{A})\longrightarrow\mathcal{H}(\mathcal{A})$.
More concretely, there is a unique functor, up to natural
isomorphism,
$\bar{q}:\mathcal{H}(\mathcal{A})\longrightarrow\mathcal{D}(\mathcal{A})$
such that $\bar{q}\circ p=q$.
\item $\overline{\mathcal{Q}}:=p(\mathcal{Q})$ is a multiplicative
system in $\mathcal{H}(\mathcal{A})$ compatible with the
triangulation and the functor $\bar{q}$ induces an equivalence of
categories
$\mathcal{H}(\mathcal{A})[\overline{Q}^{-1}]\stackrel{\cong}{\longrightarrow}\mathcal{D}(\mathcal{A})$.
\end{enumerate}
\end{prop}
\begin{proof}
See D\'efinition 1.2.2, Proposition 1.3.5 and Remarque 1.3.7 in
\cite[Chapitre III]{V}.
\end{proof}

\begin{cor}
$\mathcal{D}(\mathcal{A)}$ admits a unique structure of triangulated
category such that the functor
$\bar{q}:\mathcal{H}(\mathcal{A})\longrightarrow\mathcal{D}(\mathcal{A})$
is triangulated. Moreover, for each $k\in\mathbf{Z}$, the
cohomological functor
$H^k:\mathcal{H}(\mathcal{A})\longrightarrow\mathcal{A}$ factors
thorough $\bar{q}$ in a unique way.
\end{cor}

It remains to settle the set-theoretical problem that
$\mathcal{D}(\mathcal{A})$ is a big category. Led by remark
\ref{rem.big category-enough for usual category}(2), we will
characterize when the functor
$q:\mathcal{H}(\mathcal{A})\longrightarrow\mathcal{D}(\mathcal{A})$
has an adjoint. Recall that an object
$X^\bullet\in\mathcal{C}(\mathcal{A})$ is called an \emph{acyclic
complex} when it has zero homology, i.e., when $H^k(X^\bullet )=0$,
for all $k\in\mathbb{Z}$. Note that, when $\mathcal{A}$ is AB4, the
$k$-th homology functor
$H^k:\mathcal{C}(\mathcal{A})\longrightarrow\mathcal{A}$ preserves
coproducts. In this case, if $(X_i^\bullet )_{i\in I}$ is a family
of acyclic complexes which has a coproduct in
$\mathcal{C}(\mathcal{A})$ (equivalently, in
$\mathcal{H}(\mathcal{A})$), then $\coprod_{i\in I}X_i^\bullet$ is
also an acyclic complex. Viewed as a full subcategory of
$\K(\mathcal{A})$, it follows that the class $\mathcal{Z}$ of
acyclic complexes is a triangulated subcategory closed under taking
coproducts. The dual fact applies to products when $\mathcal{A}$ is
AB4*. The following result (resp. its dual), which is a direct
consequence of proposition \ref{prop.Verdier localization by
triangulated},  gives a criterion for the canonical functor
$q:\mathcal{H}(\mathcal{A})\longrightarrow\mathcal{D}(\mathcal{A})$
to have a right (resp. left) adjoint.

\begin{prop} \label{prop.homotopically injective resolution}
Let $\mathcal{A}$ be an AB4 abelian category and denote by
$\mathcal{Z}$ the full subcategory of $\K(\mathcal{A})$ whose
objects are the (images by the quotient functor
$p:\mathcal{C}(\mathcal{A})\longrightarrow\mathcal{H}(\mathcal{A})$
of) acyclic complexes. The following assertions are equivalent:

\begin{enumerate}
\item The  pair $(\mathcal{Z},\mathcal{Z}^\perp )$ is a semi-orthogonal decomposition in  $\mathcal{H}(\mathcal{A})$.
\item The canonical functor
$q:\mathcal{H}(\mathcal{A})\longrightarrow\mathcal{D}(\mathcal{A})$
has a right adjoint.
\end{enumerate}
In such case, there is an equivalence of categories
$\mathcal{D}(\mathcal{A})\cong\mathcal{Z}^\perp$, so that
$\D(\mathcal{A})$ is a real category and not just a big one.
\end{prop}

\begin{defi}
A chain complex $Q^\bullet\in\mathcal{C}(\mathcal{A})$ is called
\emph{homotopically injective} (resp. \emph{homotopically
projective}) when $p(Q^\bullet )\in\mathcal{Z}^\perp$ (resp.
$p(Q^\bullet )\in {}^\perp\mathcal{Z}$), where $\mathcal{Z}$ is as
in the previous proposition and
$p:\mathcal{C}(\mathcal{A})\longrightarrow\mathcal{H}(\mathcal{A})$
is the projection functor. A chain complex $X^\bullet$ is said to
have a \emph{homotopically injective resolution} (resp.
\emph{homotopically projective resolution}) in
$\mathcal{H}(\mathcal{A})$ when there is quasi-isomorphism
$\iota:X^\bullet\longrightarrow I^\bullet_X$ (resp. $\pi
:P^\bullet_X\longrightarrow X^\bullet$), where $I^\bullet_X$ (resp.
$P^\bullet_X$) is a homotopically injective (resp. homotopically
projective) complex.
\end{defi}

Note that $X^\bullet$ has a homotopically injective resolution if,
and only if, there is a triangle $Z^\bullet\longrightarrow
X^\bullet\stackrel{\iota}{\longrightarrow}I^\bullet\stackrel{+}{\longrightarrow}$
in $\K(\mathcal{A})$ such that $Z^\bullet\in\mathcal{Z}$ and
$I^\bullet\in\mathcal{Z}^\perp$ (i.e. $Z^\bullet$ is acyclic and
$I^\bullet$ is homotopically injective). A dual fact is true about
the existence of a homotopically projective resolution.

As a direct consequence of the definition of semi-orthogonal
decomposition and of proposition \ref{prop.homotopically injective
resolution}, we get the following result. The statement of the dual
result is left to the reader.

\begin{cor} \label{cor.homotopically projective and injective}
Let $\mathcal{A}$ be an AB4 abelian category.  The canonical functor
$q:\mathcal{H}(\mathcal{A})\longrightarrow\mathcal{D}(\mathcal{A})$
has a right adjoint if, and only if, each chain complex $X^\bullet$
has a homotopically injective resolution
$\iota_X:X^\bullet\longrightarrow I^\bullet_X$. In such case
$I^\bullet_X$ is uniquely determined, up to isomorphism in
$\mathcal{H}(\mathcal{A})$, and the assignment $X\rightsquigarrow
I^\bullet_X$ is the definition on objects of a triangulated functor
$\mathbf{i}_\mathcal{A}:\mathcal{D}(\mathcal{A})\longrightarrow\mathcal{H}(\mathcal{A})$
which is right adjoint to $q$.
\end{cor}

\begin{defi} \label{def.homotopically inj.proj.resol.functor}
In the situation of the previous corollary,
$\mathbf{i}_\mathcal{A}:\D(\mathcal{A})\longrightarrow\K(\mathcal{A})$
is called the \emph{homotopically injective resolution functor}.
When the dual result holds, one defines the \emph{homotopically
projective resolution functor}
$\mathbf{p}_\mathcal{A}:\D(\mathcal{A})\longrightarrow\K(\mathcal{A})$,
which is left adjoint to $q:\K(\mathcal{A})\longrightarrow\D(A)$.
\end{defi}

\begin{rem}
When the homotopically projective (resp. homotopically injective)
resolution functor is defined, the counit
$q\circ\mathbf{i}_\mathcal{A}\longrightarrow 1_{\D(A)}$ (resp. the
unit $1_{\D(\mathcal{A})}\longrightarrow
q\circ\mathbf{p}_\mathcal{A}$) of the adjunction pair
$(q,\mathbf{i}_\mathcal{A})$ (resp. $(\mathbf{p}_\mathcal{A},q)$) is
a natural isomorphism. We will denote by $\iota
:1_{\K(\mathcal{A})}\longrightarrow\mathbf{i}_\mathcal{A}\circ q$
(resp. $\pi :\mathbf{p}_\mathcal{A}\circ q\longrightarrow
1_{\K(A)}$) the unit (resp. counit) of the first (resp. second)
adjoint pair. But since $q$ 'is' the identity on objects, we will
simply write
$\iota=\iota_{X^\bullet}:X^\bullet\longrightarrow\mathbf{i}_\mathcal{A}X^\bullet$
 and $\pi =\pi_{X^\bullet}:\mathbf{p}_\mathcal{A}X^\bullet\longrightarrow
 X^\bullet$. Both of them are quasi-isomorphisms.
\end{rem}

The canonical examples that we should keep in mind are the
following:

\begin{exems}
\begin{enumerate}
\item If $\mathcal{A}=\mathcal{G}$ is a Grothendieck category, then
each chain complex of objects of $\mathcal{G}$ admits a
homotopically injective resolution (\cite[Theorem 5.4]{AJSo1}, see
also \cite[Theorem D]{S}).

\item If $\mathcal{A}=\text{Mod}-A$ is the category of (right)
modules over a $k$-algebra $A$, then each chain complex of
$A$-modules admits both a homotopically projective and a
homotopically injective resolution (see \cite[Theorem C]{S}). Note
that if $A$ and $B$ are $k$-algebras, then a $B-A-$bimodule is the
same as a right $B^{op}\otimes A$-module, by the rule $m(b^o\otimes
a)=bma$, for all $a\in A$, $b\in B$ and $m\in M$. Therefore the
existence of homotopically injective and homotopically projective
resolutions also applies when taking
$\mathcal{A}=\text{Mod}-(B^{op}\otimes A)$ to be the category of
$B-A-$bimodules.

\item Since we have canonical equivalences
$\mathcal{C}(\mathcal{A}^{op})\cong\mathcal{C}(\mathcal{A})^{op}$
and
$\mathcal{H}(\mathcal{A}^{op})\cong\mathcal{H}(\mathcal{A})^{op}$,
the opposite category of a (bi)module category is another example of
abelian category $\mathcal{A}$ over which every complex admits both
a homotopically projective resolution and a homotopically injective
one.
\end{enumerate}
\end{exems}

A consequence of Brown representability theorem is now  the
following result.

\begin{prop} \label{prop.triangulated subcategory is aisle}
Let $A$ be an algebra and $\mathcal{S}$ be any \underline{set} of
objects of $\D(A)$. The following assertions hold:

\begin{enumerate}
\item $\mathcal{U}:=\text{Susp}_{\D(A)}(\mathcal{S})$ is the aisle of
a t-structure in $\D(A)$. The co-aisle $\mathcal{U}^\perp$ consists
of the complexes $Y^\bullet$ such that $\D(A)(S^\bullet
[k],Y^\bullet )=0$, for all $S^\bullet\in\mathcal{S}$ and and all
integers $k\geq 0$;
\item $\mathcal{T}:=\text{Tria}_{\D(A)}(\mathcal{S})$ is the aisle of
a semi-orthogonal decomposition in $\D(A)$. In this case
$\mathcal{T}^\perp$ consists of the complexes $Y^\bullet$ such that
$\D(A)(S^\bullet [k],Y^\bullet )=0$, for all
$S^\bullet\in\mathcal{S}$ and  $k\in\mathbb{Z}$.
\end{enumerate}
\end{prop}
\begin{proof}
Assertion 1 is proved in \cite[Proposition 3.2]{AJSo} (see also
\cite{So} and \cite[Theorem 12.1]{KN} for more general versions).
Assertion 2 follows from 1 since
$\text{Tria}_{\D(A)}(\mathcal{S})=\text{Susp}_{\D(A)}(\bigcup_{k\in\mathbb{Z}}\mathcal{S}[k])$.
\end{proof}

\subsection{Derived functors and adjunctions}

\begin{lemma} \label{lem.adjunction pass to the stable category}
Let $\mathcal{C}$ and $\mathcal{D}$ be Frobenius exact categories
and let $F:\mathcal{C}\longrightarrow\mathcal{D}$ and
$G:\mathcal{D}\longrightarrow\mathcal{C}$ be functors which take
conflations to conflations. Then the following statements hold true:

\begin{enumerate}
\item If $F$ takes projective objects to projective objects, then
there is a  triangulated functor
$\underline{F}:\underline{C}\longrightarrow\underline{D}$, unique up
to natural isomorphism,  such that $p_\mathcal{D}\circ
F=\underline{F}\circ p_\mathcal{C}$.

\item If $(F,G)$ is an adjoint pair, then the
following assertions hold:

\begin{enumerate}
\item Both $F$ and $G$ preserve projective objects.
\item The induced triangulated functors
$\underline{F}:\underline{\mathcal{C}}\longrightarrow\underline{\mathcal{D}}$
and
$\underline{G}:\underline{\mathcal{D}}\longrightarrow\underline{\mathcal{C}}$
form an adjoint pair $(\underline{F},\underline{G})$.
\end{enumerate}
\end{enumerate}
\end{lemma}
\begin{proof}
(1) The existence of the functor is immediate and the fact that it is
triangulated is due to the fact that all triangles in
$\underline{C}$ and $\underline{D}$ are 'image' of conflations in
$\mathcal{C}$ and $\mathcal{D}$ by the respective projection
functors.

(2) (a) The proof is identical to the one which proves, for arbitrary
categories,  that each left (resp. right) adjoint of a functor which
preserves epimorphisms (resp. monomorphisms)  preserves projective
(resp. injective) objects. Then use the fact that the injective and
projective objects coincide.

(b) Let us fix an isomorphism $\eta
:\mathcal{D}(F(?),?)\stackrel{\cong}{\longrightarrow}\mathcal{C}(?,G(?))$
natural on both variables and let $C\in\mathcal{C}$ and
$D\in\mathcal{D}$ be any objects. If $\alpha\in\mathcal{D}(F(C),D)$
is a morphism which factors through a projective object $Q$ of
$\mathcal{D}$, then we have a decomposition
$F(C)\stackrel{\beta}{\longrightarrow}Q\stackrel{\gamma}{\longrightarrow}D$.
Due to the naturality of $\eta$, we then have a commutative diagram:

\[
\xymatrix{
\D(F(C),Q)\ar[rr]^{\eta_{C,Q}}\ar[d]^{\gamma_{*}} && \C(C,G(Q))\ar[d]^{G(\gamma )_{*}} \\
\D(F(C),D)\ar[rr]^{\eta_{C,D}} && \C(C,G(D)) }
\]
with the obvious meaning of the vertical arrows. It follows that

\begin{center}
$\eta_{C,D}(\alpha
)=\eta_{C,D}(\gamma\circ\beta)=(\eta_{C,D}\circ\gamma_*)(\beta)=G(\gamma)_*\circ\eta_{C,Q}(\beta)=G(\gamma
)\circ\eta_{C,Q}(\beta)$,
\end{center}
which proves that $\eta_{C,D}(\alpha )$ factors through the
projective object $G(Q)$ of $\mathcal{C}$. We then get and induced
map
$\bar{\eta}_{C,D}:\underline{D}(\underline{F}(C),D)\longrightarrow\underline{C}(C,\underline{G}(D))$,
which is natural on both variables. Applying the symmetric argument
to $\eta^{-1}$, we obtain a map
$\overline{\eta^{-1}}_{C,D}:\underline{C}(C,\underline{G}(D))\longrightarrow\underline{D}(\underline{F}(C),D)$,
natural on both variables, which is clearly inverse to $\bar{\eta}$.
Then $(\underline{F},\underline{G})$ is an adjoint pair, as desired.
\end{proof}

Bearing in mind that
$\underline{\mathcal{C}(\mathcal{A})}=\mathcal{H}(\mathcal{A})$, for
any abelian category category $\mathcal{A}$, the following
definition makes sense.

\begin{defi} \label{def.derived funtor}
Let $\mathcal{A}$ be an AB4 abelian category such that every chain
complex of objects of $\mathcal{A}$ admits a homotopically injective
resolution, and let
$\mathbf{i}_\mathcal{A}:\mathcal{D}(\mathcal{A})\longrightarrow\mathcal{H}(\mathcal{A})$
be the homotopically injective resolution functor. If $\mathcal{B}$
is another abelian category and
$F:\mathcal{C}(\mathcal{A})\longrightarrow\mathcal{C}(\mathcal{B})$ is a
$k$-linear functor which takes conflations to conflations and
contractible complexes to contractible complexes, then the
composition of triangulated funtors

\begin{center}
$\mathcal{D}(\mathcal{A})\stackrel{\mathbf{i}_\mathcal{A}}{\longrightarrow}\mathcal{H}(\mathcal{A})\stackrel{\underline{F}}{\longrightarrow}
\mathcal{H}(\mathcal{B})\stackrel{q_\mathcal{B}}{\longrightarrow}\mathcal{D}(\mathcal{B})$
\end{center}
is called the \emph{right derived functor of $F$}, usually denoted
$\mathbf{R}F$.

Dually, one defines the the \text{left derived functor of $F$},
denoted $\mathbf{L}F$,  whenever $\mathcal{A}$ is an AB4* abelian
category on which every chain complex admits a homotopically
projective resolution.
\end{defi}

We now get:

\begin{prop} \label{prop.adjunction of derived functors}
Let $\mathcal{A}$ and $\mathcal{B}$ be abelian categories. Suppose
that $\mathcal{A}$ is AB4* on which each chain complex has a
homotopically projective resolution and $\mathcal{B}$ is AB4 on
which each chain complex has a homotopically injective resolution.
If
$F:\mathcal{C}(\mathcal{A})\longrightarrow\mathcal{C}(\mathcal{B})$
and
$G:\mathcal{C}(\mathcal{B})\longrightarrow\mathcal{C}(\mathcal{A})$
are $k$-linear functors which take conflations to conflations and
form an adjoint pair $(F,G)$, then the derived functors
$\mathbf{L}F:\mathcal{D}(\mathcal{A})\longrightarrow\mathcal{D}(\mathcal{B})$
and
$\mathbf{R}F:\mathcal{D}(\mathcal{B})\longrightarrow\mathcal{D}(\mathcal{A})$
form an adjoint pair $(\mathbf{L}F,\mathbf{R}G)$.
\end{prop}
\begin{proof}
By definition, we have

\begin{center}
$\mathbf{L}F:\mathcal{D}(\mathcal{A})\stackrel{\mathbf{p}_\mathcal{A}}{\longrightarrow}\mathcal{H}(\mathcal{A})\stackrel{\underline{F}}{\longrightarrow}
\mathcal{H}(\mathcal{B})\stackrel{q_\mathcal{B}}{\longrightarrow}\mathcal{D}(\mathcal{B})$

and

$\mathbf{R}G:\mathcal{D}(\mathcal{B})\stackrel{\mathbf{i}_\mathcal{B}}{\longrightarrow}\mathcal{H}(\mathcal{A})\stackrel{\underline{G}}{\longrightarrow}
\mathcal{H}(\mathcal{A})\stackrel{q_\mathcal{A}}{\longrightarrow}\mathcal{D}(\mathcal{A})$.
\end{center}
The result is an immediate consequence of the fact that
$(\mathbf{p}_\mathcal{A},q_\mathcal{A})$,
$(\underline{F},\underline{G})$ and
$(q_\mathcal{B},\mathbf{i}_\mathcal{B})$ are adjoint pairs.
\end{proof}


\section{Classical derived functors defined by complexes of bimodules}

\subsection{Definition and main adjunction properties}

If $A$ is a $k$-algebra, we will write $\mathcal{C}(A)$,
$\mathcal{H}(A)$ and $\mathcal{D}(A)$ instead of
$\mathcal{C}(\text{Mod}A)$, $\mathcal{H}(\text{Mod}A)$ and
$\mathcal{D}(\text{Mod}A)$, respectively. This rule applies also to
the algebra $B^{op}\otimes A$, for any algebras $A$ and $B$.

Let $A$, $B$ and $C$ be $k$-algebras. Given  complexes $T^\bullet$
and $M^\bullet$,  of $B-A-$bimodules and $A-C-$bimodules
respectively, we shall associate to them several functors between
categories of complexes of bimodules. The material can be found in
\cite{W}, but the reader is warned on the difference of indization
of complexes with respect to that book.
 The \emph{total tensor product of $T^\bullet$
and $M^\bullet$}, denoted $T^\bullet\otimes^\bullet_AM^\bullet$, is
the complex of $B-C-$bimodules given as follows (see \cite[Section
2.7]{W}):

\begin{enumerate}
\item
$(T^\bullet\otimes^\bullet_AM^\bullet)^n=\oplus_{i+j=n}T^i\otimes_AM^j$,
 for each $n\in\mathbb{Z}$.

\item The differential
$d^n:(T^\bullet\otimes^\bullet_AM^\bullet)^n\longrightarrow
 (T^\bullet\otimes^\bullet_AM^\bullet)^{n+1}$ takes $t\otimes m\rightsquigarrow d_T(t)\otimes m+(-1)^it\otimes
 d_M(m)$, whenever $t\in T^i$ and $m\in M^j$
 \end{enumerate}

 When $T^\bullet$ is fixed, the assignment
$M^\bullet\rightsquigarrow T^\bullet\otimes^\bullet_AM^\bullet$ is
the definition on objects of a $k$-linear  functor
$T^\bullet\otimes_A^\bullet?:\mathcal{C}(A^{op}\otimes
C)\longrightarrow\mathcal{C}(B^{op}\otimes C)$. It acts as
$f\rightsquigarrow 1_T\otimes  f$ on morphisms, where $( 1_T\otimes
f)(t\otimes m)=t\otimes f(m)$, whenever $t\in T^\bullet$ and $m\in
M^\bullet$ are homogeneous elements. Note that, with a symmetric
argument, when $M^\bullet\in\mathcal{C}(A^{op}\otimes C)$ is fixed,
we also get another $k$-linear functor $?\otimes_A^\bullet M^\bullet
:\mathcal{C}(B^{op}\otimes
A)\longrightarrow\mathcal{C}(B^{op}\otimes C)$ which takes
$T^\bullet\rightsquigarrow T^\bullet\otimes_A^\bullet M^\bullet$ and
$g\rightsquigarrow g\otimes 1_M$. In this way, we get a bifunctor
$?\otimes_A^\bullet ?:\mathcal{C}(B^{op}\otimes
A)\times\mathcal{C}(A^{op}\otimes
C)\longrightarrow\mathcal{C}(B^{op}\otimes C)$.

Suppose now that $N^\bullet$ is a complex of $C-A-$bimodules. The
\emph{total Hom complex $\text{Hom}_A^\bullet (T^\bullet, N^\bullet
)$} is the complex of $C-B-$bimodules given as follows (see
\cite[Section 2.7]{W} or, in a more general context,  \cite[Section
1.2]{Ke}):

\begin{enumerate}
\item $\text{Hom}^\bullet_A(T^\bullet ,N^\bullet
)^n=\prod_{j-i=n}\text{Hom}_A(T^i,N^j)$, for each $n\in\mathbb{Z}$.
\item The differential $\text{Hom}^\bullet_A(T^\bullet ,N^\bullet
)^n\longrightarrow\text{Hom}^\bullet_A(T^\bullet ,N^\bullet )^{n+1}$
takes $f\rightsquigarrow d_N\circ f-(-1)^nf\circ d_M$
\end{enumerate}
The assigment
$N^\bullet\rightsquigarrow\text{Hom}^\bullet_A(T^\bullet ,N^\bullet
)$ is the definition on objects of a $k$-linear  functor
$\text{Hom}^\bullet_A(T^\bullet ,?):\mathcal{C}(C^{op}\otimes
A)\longrightarrow \mathcal{C}(C^{op}\otimes B)$. The functor acts on
morphism as $\alpha\rightsquigarrow \alpha_*$, where
$\alpha_*(f)=\alpha\circ f$, for each homogeneous element
$f\in\text{Hom}^\bullet_A(T^\bullet ,N^\bullet )$. Symmetrically, if
$N^\bullet\in\C(C^{op}\otimes A)$ is fixed, then the assignment
$T^\bullet\rightsquigarrow\Th_A(T^\bullet ,N^\bullet )$ is the
definition on objects of a functor
$\Th_A(?,N^\bullet):\C(B^{op}\otimes
A)^{op}\longrightarrow\C(C^{op}\otimes B)$. It acts on morphisms as
$\beta\rightsquigarrow\Th_A(\beta,N^\bullet )=:\beta^*$, where
$\beta^*(f )=f\circ\beta$, for each homogeneous element
$f\in\text{Hom}^\bullet_A(T^\bullet ,N^\bullet )$. In this way, we
get a bifunctor $\Th_A(?,?):\C (B^{op}\otimes
A)\times\C(C^{op}\otimes A)\longrightarrow\C(C^{op}\otimes B)$.

Obviously, we also get bifunctors:

\begin{center}
$?\Tt_B?:\C(C^{op}\otimes B)\times\C(B^{op}\otimes
A)\longrightarrow\C(C^{op}\otimes A)$, \hspace*{1cm} $(X^\bullet
,T^\bullet)\rightsquigarrow X^\bullet\Tt_BT^\bullet$;

$\Th_{B^{op}}(?,?):\C(B^{op}\otimes C)\times\C(B^{op}\otimes
A)\longrightarrow\C(C^{op}\otimes A)$, \hspace*{1cm} $(Y^\bullet
,T^\bullet )\rightsquigarrow\Th_{B^{op}}(Y^\bullet ,T^\bullet )$.

\end{center}

\begin{prop} \label{prop.adjunctions given by bimodule}
Let $A$, $B$ and $C$ be $k$-algebras. All the bifunctors $?\Tt_A?$,
$?\Tt_B?$, $\Th_A(?,?)$ and $\Th_{B^{op}}(?,?)$ defined above
preserve conflations (i.e. semi-split exact sequences) and
contractible complexes on each variable. Moreover, if $T^\bullet$ is
a complex of $B-A-$bimodules, then  the following assertions hold:

\begin{enumerate}

\item The pairs $(?\otimes_B^\bullet T^\bullet ,\text{Hom}^\bullet_A(T^\bullet
,?))$ and $(T^\bullet\otimes_A^\bullet
?,\text{Hom}^\bullet_{B^{op}}(T^\bullet ,?))$ are adjoint pairs.
\item The pair  $(\text{Hom}^\bullet_{B^{op}}(?,T^\bullet ):\C(B^{op}\otimes C)\longrightarrow\C(C^{op}\otimes A)^{op},\text{Hom}^\bullet_{A}(?,T^\bullet
):\C(C^{op}\otimes A)^{op}\longrightarrow\C(B^{op}\otimes C))$  is
an adjoint pair.
\end{enumerate}
\end{prop}
\begin{proof}
  It is clear that the $k$-linear functors
$T^\bullet\Tt_A?:\C(A^{op}\otimes C)\longrightarrow\C(B^{op}\otimes
C)$,  $\Th_A(T^\bullet ,?):\C(C^{op}\otimes
A)\longrightarrow\C(C^{op}\otimes B)$ and $\Th_A(?,T^\bullet
):\C(C^{op}\otimes A)\longrightarrow\C(B^{op}\otimes C)^{op}$, after
forgetting about the differentials, induce corresponding functors
between the associated graded categories.
 For instance, we have an induced
$k$-linear functor $?\otimes_B^\bullet T^\bullet
:(\text{Mod}(C^{op}\otimes
B))^\mathbb{Z}\longrightarrow(\text{Mod}(C^{op}\otimes
A))^\mathbb{Z}$. But then this latter functor preserves split exact
sequences. This is exactly saying that $?\otimes_B^\bullet T^\bullet
:\mathcal{C}(C^{op}\otimes
B)\longrightarrow\mathcal{C}(C^{op}\otimes A)$ takes conflations to
conflations. The corresponding argument works for all the other
functors.

So all the bifunctors in the list preserve conflations on each
component. That they also preserve contractible complexes on each
variable, will follow from lemma \ref{lem.adjunction pass to the
stable category} once we prove the adjunctions of assertions (1) and
(2).

(1) This goes as in the usual adjunction between the tensor product
and the $\text{Hom}$ functor in module categories. Concretely, we
define

\begin{center}
$\eta=\eta_{M,Y}:\mathcal{C}(C^{op}\otimes
A)(M^\bullet\otimes_A^\bullet T^\bullet ,Y^\bullet )\longrightarrow
\mathcal{C}(C^{op}\otimes B)(M^\bullet,\text{Hom}_A^\bullet
(T^\bullet ,Y^\bullet ))$
\end{center}
by the rule:  $[\eta (f)(m)](t)=f(m\otimes t)$, whenever $m\in M^i$
and $t\in T^j$, for some  $i,j\in\mathbb{Z}$. We leave to the reader
the routine task of checking  that $\eta (f)$ is a chain map
$M^\bullet\longrightarrow\text{Hom}_A^\bullet (T^\bullet ,Y^\bullet
)$. The naturallty of $\eta$ is then proved as in modules.

(2) We need to define an isomorphism

\begin{center}
$\xi =\xi_{M,Y}:\mathcal{C}(B^{op}\otimes C)(Y^\bullet
,\text{Hom}_A^\bullet (M^\bullet ,T^\bullet ))\longrightarrow
 \mathcal{C}(C^{op}\otimes A)^{op}(\text{Hom}_{B^{op}}^\bullet
(Y^\bullet ,T^\bullet ), M^\bullet )=\mathcal{C}(C^{op}\otimes
A)(M^\bullet ,\text{Hom}_{B^{op}}^\bullet (Y^\bullet ,T^\bullet ))$,
\end{center}
natural on $M^\bullet\in\mathcal{C}(C^{op}\otimes A)$ and $Y^\bullet
\in\mathcal{C}(B^{op}\otimes C)$. Our choice of $\xi$ is identified
by the equality $[\xi (f)(m)](y)=(-1)^{|m||y|}f(y)(m)$, for all
homogeneous elements $f\in\mathcal{C}(B^{op}\otimes C)(Y^\bullet
,\text{Hom}_A^\bullet (M^\bullet ,T^\bullet ))$, $m\in M^\bullet$
and $y\in Y^\bullet$.
\end{proof}

\vspace*{0.3cm}

With the notation of last proposition, we adopt the following
notation:
\begin{enumerate}
 \item[a)] $?\otimes_B^\mathbf{L}T:\mathcal{D}(C^{op}\otimes
B)\longrightarrow\mathcal{D}(C^{op}\otimes A)$ will denote the left
derived functor of $?\otimes_B^\bullet
T^\bullet:\mathcal{C}(C^{op}\otimes
B)\longrightarrow\mathcal{C}(C^{op}\otimes A)$

\item[b)] $\mathbf{R}\text{Hom}_A(T^\bullet ,?):\mathcal{D}(C^{op}\otimes
A)\longrightarrow\mathcal{D}(C^{op}\otimes B)$ will denote the right
derived functor of $\Th_A(T^\bullet ,?):\mathcal{C}(C^{op}\otimes
A)\longrightarrow\mathcal{C}(C^{op}\otimes B)$

\item[c)] $\mathbf{R}\text{Hom}_A(?,T^\bullet ):\mathcal{D}(C^{op}\otimes
A)^{op}\longrightarrow\mathcal{D}(B^{op}\otimes C)$ will denote the
right derived functor of $\text{Hom}_A^\bullet (?,T^\bullet
):\mathcal{C}(C^{op}\otimes
A)^{op}\longrightarrow\mathcal{D}(B^{op}\otimes C)$ or,
equivalently, the left derived functor of $\text{Hom}_A^\bullet
(?,T^\bullet ):\mathcal{C}(C^{op}\otimes
A)\longrightarrow\mathcal{D}(B^{op}\otimes C)^{op}$.

\item[d)] $\T_A(?,?):\D(B^{op}\otimes A)\times\D(A^{op}\otimes C)\longrightarrow\D(B^{op}\otimes
C)$ will denote the composition

\begin{center}
$\D(B^{op}\otimes A)\times\D(A^{op}\otimes
C)\stackrel{\mathbf{p}_{B^{op}\otimes}A\times\mathbf{p}_{B^{op}\otimes}A}{\longrightarrow}\K(B^{op}\otimes
A)\times\K(A^{op}\otimes
C)\stackrel{?\Tt_A?}{\longrightarrow}\K(B^{op}\otimes
C)\stackrel{q}{\longrightarrow}\D(B^{op}\otimes C)$,
\end{center}
where the central arrow is induced by the bifunctor
$\C(B^{op}\otimes A)\times\C(A^{op}\otimes
C)\stackrel{?\Tt_A?}{\longrightarrow}\C(B^{op}\otimes C)$ and is
well-defined due to proposition \ref{prop.adjunctions given by
bimodule}. The bifunctor $\T_A(?,?)$ is triangulated on each
variable.

\item[e)] $\RH_A(?,?):\D(B^{op}\otimes A)^{op}\times\D(C^{op}\otimes A)\longrightarrow\D(C^{op}\otimes
B)$ is the composition

\begin{center}
$\D(B^{op}\otimes A)^{op}\times\D(C^{op}\otimes
A)\stackrel{\mathbf{p}_{B^{op}\otimes
A}\times\mathbf{i}_{C^{op}\otimes
A}}{\longrightarrow}\K(B^{op}\otimes A)^{op}\times\K(C^{op}\otimes
A)\stackrel{\Th_A(?,?)}{\longrightarrow}\K(C^{op}\otimes
B)\stackrel{q}{\longrightarrow}\D(C^{op}\otimes B)$,
\end{center}
which is a bifunctor which is triangulated on each variable.

\item[e')] $\RH_{B^{op}}(?,?):\D(B^{op}\otimes C)^{op}\times\D(B^{op}\otimes A)\longrightarrow\D(C^{op}\otimes
A)$ is a bifunctor, triangulated on both variables, which is defined
in a similar way as that in e).

\end{enumerate}
Of course, there are left-right symmetric versions
$T^\bullet\otimes_A^\mathbf{L}?:\mathcal{D}(A^{op}\otimes
C)\longrightarrow\mathcal{D}(B^{op}\otimes C)$,
$\mathbf{R}\text{Hom}_{B^{op}}(T^\bullet
,?):\mathcal{D}(B^{op}\otimes
C)\longrightarrow\mathcal{D}(A^{op}\otimes C)$ and
$\mathbf{R}\text{Hom}_{B^{op}}(?,T^\bullet
):\mathcal{D}(B^{op}\otimes
C)^{op}\longrightarrow\mathcal{D}(C^{op}\otimes A)$ of the functors
in a), b) and c). Their precise definition is left to the reader.

As a direct consequence of propositions \ref{prop.adjunction of
derived functors} and \ref{prop.adjunctions given by bimodule}, we
get:

\begin{cor} \label{cor.classical adjunctions of derived functors}
Let $A$, $B$ and $C$ be $k$-algebras and let $T^\bullet$ be a
complex of $B-A-$bimodules. With the notation above, the following
pairs of triangulated functors are adjoint pairs:

\begin{enumerate}
\item $(?\otimes_B^\mathbf{L}T^\bullet ,\mathbf{R}\text{Hom}_A(T^\bullet
,?))$;
\item $(T^\bullet\otimes_A^\mathbf{L}?,\mathbf{R}\text{Hom}_{B^{op}}(T^\bullet
,?))$;
\item $(\mathbf{R}\text{Hom}_{B^{op}}(?,T^\bullet ):\mathcal{D}(B^{op}\otimes
C)\longrightarrow\mathcal{D}(C^{op}\otimes A)^{op},
\mathbf{R}\text{Hom}_{A}(?,T^\bullet ):\mathcal{D}(C^{op}\otimes
A)^{op}\longrightarrow\mathcal{D}(B^{op}\otimes C))$.
\end{enumerate}
\end{cor}

\begin{defi}
We will adopt the following terminology, referred to the adjunctions
of last corollary:

\begin{enumerate}
\item $\lambda :1_{\mathcal{D}(C^{op}\otimes B)}\longrightarrow\text{R}\text{Hom}_A(T^\bullet ,?)\circ
(?\otimes_B^\mathbf{L}T)$ and $\delta
:(?\otimes_B^\mathbf{L}T)\circ\text{R}\text{Hom}_A(T^\bullet
,?)\longrightarrow 1_{\mathcal{D}(C^{op}\otimes A)}$ will be the
unit and the counit of the first adjunction.
\item $\rho :1_{\mathcal{D}(A^{op}\otimes C)}\longrightarrow \text{R}\text{Hom}_{B^{op}}(T^\bullet ,?)\circ
(T^\bullet\otimes_A^\mathbf{L}?)$ and $\phi
:(T^\bullet\otimes_A^\mathbf{L}?)\circ\text{R}\text{Hom}_{B^{op}}(T^\bullet
,?)\longrightarrow 1_{\mathcal{D}(B^{op}\otimes C)}$ are the unit
and counit of the second adjunction.
\item $\sigma: 1_{\mathcal{D}(B^{op}\otimes C)}\longrightarrow\mathbf{R}\text{Hom}_A(?,T^\bullet )\circ\mathbf{R}\text{Hom}_{B^{op}}(?,T^\bullet
)$ and $\tau :1_{\mathcal{D}(C^{op}\otimes
A)}\longrightarrow\mathbf{R}\text{Hom}_{B^{op}}(?,T^\bullet
)\circ\mathbf{R}\text{Hom}_{A}(?,T^\bullet )$ are the unit and the
counit of the third adjunction (note that the last one is an arrow
in the opposite direction when the functors are considered as
endofunctors of $\mathcal{D}(C^{op}\otimes A)^{op}$).
\end{enumerate}
\end{defi}

The following lemma is very useful:

\begin{lemma} \label{lem.units are isomorphisms}
In the situation of last definition, let us take $C=k$ and consider
the following assertions:

\begin{enumerate}
\item $\lambda_B:B\longrightarrow\mathbf{R}\text{Hom}_A(T^\bullet ,?\otimes_B^\mathbf{L}T^\bullet )(B)\cong\mathbf{R}\text{Hom}_A(T^\bullet
,?)(T^\bullet)$ is an isomorphism in $\mathcal{D}(B)$;
\item $\delta_{T^\bullet}:[\mathbf{R}\text{Hom}_A(T^\bullet ,?)\otimes_B^\mathbf{L}T^\bullet ](T^\bullet )\longrightarrow
T^\bullet$ is an isomorphism in $\mathcal{D}(A)$;
\item $\sigma_B:B\longrightarrow [\mathbf{R}\text{Hom}_A(\mathbf{R}\text{Hom}_{B^{op}}(?,T^\bullet
),T^\bullet)](B)\cong\mathbf{R}\text{Hom}_A(?,T^\bullet )(T^\bullet
)$ is an isomorphism in $\mathcal{D}(B^{op})$;
\item $\tau_{T^\bullet}:T^\bullet\longrightarrow [\mathbf{R}\text{Hom}_{B^{op}}(\mathbf{R}\text{Hom}_{A}(?,T^\bullet
),T^\bullet)](T^\bullet )$ is an isomorphism in $\mathcal{D}(A)$.
\end{enumerate}
Then the implications $(2)\Longleftarrow (1)\Longleftrightarrow
(3)\Longrightarrow (4)$ hold true.
\end{lemma}
\begin{proof}
$(1)\Longrightarrow (2)$ Putting $F=?\otimes_B^\mathbf{L}T^\bullet$
and $G=\mathbf{R}\text{Hom}_A(T^\bullet ,?)$, the truth of the
implication is a consequence of the adjunction equation
$1_{F(B)}=\delta_{F(B)}\circ F(\lambda_B)$ and the fact that
$F(B)\cong T^\bullet$.

$(3)\Longrightarrow (4)$ also follows from the equations of the
adjunction $(\mathbf{R}\text{Hom}_{B^{op}}(?,T^\bullet ),
\mathbf{R}\text{Hom}_{A}(?,T^\bullet ))$ (see corollary
\ref{cor.classical adjunctions of derived functors}).

$(1)\Longleftrightarrow (3)$ Let
$\mathbf{p}_A,\mathbf{i}_A:\mathcal{D}(A)\longrightarrow\mathcal{H}(A)$
be the homotopically projective resolution functor and the
homotopically injective resolution functor, respectively. By
definition of the derived functors, we have
$\mathbf{R}\text{Hom}_A(T^\bullet
,?)(T^\bullet)=\text{Hom}^\bullet_A(T^\bullet ,\mathbf{i}_AT^\bullet
)$ and $\mathbf{R}\text{Hom}_A(?,T^\bullet )(T^\bullet
)=\text{Hom}_A^\bullet (\mathbf{p}_AT^\bullet ,T^\bullet)$. Let then
$\pi :\mathbf{p}_AT^\bullet\longrightarrow T^\bullet$ and $\iota
:T^\bullet\longrightarrow\mathbf{i}_AT^\bullet$ be the canonical
quasi-isomorphisms. We then have quasi-isomorphisms in
$\mathcal{C}(k)$:

\begin{center}
$\mathbf{R}\text{Hom}_A(T^\bullet
,?)(T^\bullet)=\text{Hom}^\bullet_A(T^\bullet ,\mathbf{i}_AT^\bullet
)\stackrel{\pi^*}{\longrightarrow}\text{Hom}^\bullet_A(\mathbf{p}_AT^\bullet
,\mathbf{i}_AT^\bullet
)\stackrel{\iota_*}{\longleftarrow}\text{Hom}^\bullet_A(\mathbf{p}_AT^\bullet
,T^\bullet )\cong \mathbf{R}\text{Hom}_A(?,T^\bullet )(T^\bullet )$
\end{center}
Note that $\lambda_B$ and $\sigma_B$ are the compositions:

\begin{center}
$\lambda_B:B\longrightarrow\text{Hom}_A^\bullet (T^\bullet
,T^\bullet
)\stackrel{\tilde{\iota}}{\longrightarrow}\text{Hom}^\bullet_A(T^\bullet
,\mathbf{i}_AT^\bullet )=\mathbf{R}\text{Hom}_A(T^\bullet
,?)(T^\bullet)$,

and

$\sigma_B:B\longrightarrow\text{Hom}_A^\bullet (T^\bullet ,T^\bullet
)\stackrel{\tilde{\pi}}{\longrightarrow}\text{Hom}^\bullet_A(\mathbf{p}_AT^\bullet
,T^\bullet )=\mathbf{R}\text{Hom}_A(?,T^\bullet )(T^\bullet )$,
\end{center}
where the first arrow, in both cases,  takes
$b\rightsquigarrow\lambda_b:t\rightsquigarrow bt$, and the other
arrows are $\tilde{\iota}=\text{Hom}_A^\bullet (T^\bullet ,\iota )$
and $\tilde{\pi}=\text{Hom}_A^\bullet (\pi ,T^\bullet )$.

A direct easy calculation shows that the equality
$\pi^*\circ\lambda_B=\iota_*\circ\sigma_B$ holds in
$\mathcal{C}(k)$. As a consequence, $\lambda_B$ is a
quasi-isomorphism if, and only if, so is $\sigma_B$.
\end{proof}

We explicitly state the left-right symmetric of the previous lemma
since it will be important for us:

\begin{lemma} \label{lem.units are isomorphisms2}
In the situation of last definition, let us take $C=k$ and consider
the following assertions:

\begin{enumerate}
\item $\rho_A:A\longrightarrow\mathbf{R}\text{Hom}_{B^{op}}(T^\bullet ,T^\bullet\otimes_A^\mathbf{L}?)(A)\cong\mathbf{R}\text{Hom}_{B^{op}}(T^\bullet
,?)(T^\bullet)$ is an isomorphism in $\mathcal{D}(A^{op})$;
\item $\phi_{T}:[T^\bullet\otimes_A^\mathbf{L}\otimes\mathbf{R}\text{Hom}_A(T^\bullet ,?)](T^\bullet )\longrightarrow
T^\bullet$ is an isomorphism in $\mathcal{D}(B^{op})$;
\item $\tau_A:A\longrightarrow [\mathbf{R}\text{Hom}_{B^{op}}(\mathbf{R}\text{Hom}_{A}(?,T^\bullet
),T^\bullet)](A)\cong\mathbf{R}\text{Hom}_{B^{op}}(?,T^\bullet
)(T^\bullet )$ is an isomorphism in $\mathcal{D}(A)$;
\item $\sigma_T:T^\bullet\longrightarrow [\mathbf{R}\text{Hom}_{A}(\mathbf{R}\text{Hom}_{B^{op}}(?,T^\bullet
),T^\bullet)](T^\bullet )$ is an isomorphism in
$\mathcal{D}(B^{op})$.
\end{enumerate}
Then the implications $(2)\Longleftarrow (1)\Longleftrightarrow
(3)\Longrightarrow (4)$ hold true.
\end{lemma}

\subsection{Homotopically flat complexes. Restrictions}
In order to deal with the derived tensor product, it is convenient
to introduce a class of chain complexes which is wider than that of
the homotopically projective ones.

\begin{defi} \label{def.homotopically flat}
Let $B$ be any algebra. A complex $F^\bullet\in\C(B)$ is called
\emph{homotopically flat} when the functor
$F^\bullet\Tt_B?:\C(B)\longrightarrow\C(k)$ preserves acyclic
complexes.
\end{defi}

The key points are assertion (3) and (4) of the following result.

\begin{lemma} \label{lem.homotopically flat resolution}
Let $A$ and $B$ be $k$-algebras. The following assertions hold:
\begin{enumerate}
\item If $Z^\bullet\in\C(B)$ is acyclic and homotopically flat, then
the essential image of
$Z^\bullet\Tt_B?:\K(B^{op})\longrightarrow\K(k)$ consists of acyclic
complexes;
 \item Each homotopically projective object of $\K(B)$ is
 homotopically flat;
 \item If
$F^\bullet\stackrel{s}{\longrightarrow}M^\bullet$ is a
quasi-isomorphism in $\C(B)$, where $F^\bullet$  is homotopically
flat, then, for each $N^\bullet\in\C(B^{op}\otimes A)$, we have an
isomorphicm $(?\Lt_BN^\bullet )(M^\bullet )\cong
F^\bullet\Tt_BN^\bullet$ in $\D(A)$;
\item If $F^\bullet$ is homotopically flat in $\K(B)$ and $N^\bullet$ is a complex  of  $B-A$-bimodules, then
$(F^\bullet\Lt_B?)(N^\bullet )=F^\bullet\Tt_BN^\bullet$ in $\D(A)$.
\end{enumerate}
\end{lemma}
\begin{proof}

All throughout the proof we will use the fact that   $\D(B)$ (resp.
$\D(B^{op})$) is compactly generated by $\{B\}$ (see proposition
\ref{prop.D(A) compactly generated} below).

(1) The given functor $T:=Z^\bullet\Tt_B?$ is triangulated and take
acyclic complexes to acyclic complexes. Then it preserves
quasi-isomorphisms. By construction of the derived category, we then
get a unique triangulated functor
$\bar{T}:\mathcal{D}(B^{op})\longrightarrow\mathcal{D}(k)$ such that
$\bar{T}\circ q_B=q_k\circ T$, where
$q_B:\mathcal{H}(B^{op})\longrightarrow\mathcal{D}(B^{op})$ and
$q_k:\mathcal{H}(k)\longrightarrow\mathcal{D}(k)$ are the canonical
functor.

We will prove that $\bar{T}$ is the zero functor and this will prove
the assertion. We  consider the full subcategory $\mathcal{X}$ of
$\mathcal{D}(B^{op})$ consisting of the complexes $M^\bullet$ such
that $\bar{T}(M^\bullet )=0$. It is  a triangulated subcategory,
closed under taking arbitrary coproducts,  which contains $B$. We
then have $\mathcal{X}=\mathcal{D}(B)$.

(2) By \cite[Theorem P]{Ke}, each homotopically projective complex is
isomorphic in $\K(A)$ to a complex $P^\bullet$ which admits a
countable filtration

\begin{center}
$P_0^\bullet\subset P_1^\bullet\subset ...\subset P_n^\bullet\subset
...$,
\end{center}
in $\mathcal{C}(B)$, where the inclusions are inflations and where
$P_0^\bullet$ and all the factors $P_n^\bullet/P_{n-1}^\bullet$ are
direct summands of direct sums of stalk complexes of the form
$B[k]$, with $k\in\mathbb{Z}$. We just need to check that this
$P^\bullet$ is homotopically flat.
 Due to the fact  that the bifunctor
$?\Tt_B?:\mathcal{C}(B)\times\mathcal{C}(B^{op})\longrightarrow\mathcal{C}(k)$
preserves direct limits on both variables, with an evident induction
argument,  the proof is easily reduced to the case when
$P^\bullet=B[r]$, for some $r\in\mathbb{Z}$, in which case it is
trivial.

(3) Let $P^\bullet
:=\mathbf{p}M^\bullet\stackrel{\pi}{\longrightarrow}M^\bullet$ be
the homotopically projective resolution. Then
$s^{-1}\circ\pi\in\D(B)(P^\bullet ,F^\bullet )\cong\K(B)(P^\bullet
,F^\bullet )$. Then we have a chain map $f:P^\bullet\longrightarrow
F^\bullet$ such that $s\circ f=\pi$ in $\K(B)$. In particular, $f$
is a quasi-isomoprhism between homotopically flat objects. Its cone
is then an acyclic and homotopically flat complex. By assertion (1),
we conclude that $f\Tt_B1_{N^\bullet} $ has an acyclic cone and,
hence, it is a quasi-isomorphism, for each
$N^\bullet\in\C(B^{op}\otimes A)$. We then have an isomorphism
$(?\Lt_BN^\bullet )(M^\bullet
)=P^\bullet\Tt_BN^\bullet\stackrel{f\Tt
1_{N^\bullet}}{\longrightarrow}F^\bullet\Tt_BN^\bullet$ in $\D(A)$.

(4) Let us consider the triangle in $\K(B^{op}\otimes A)$

\begin{center}
$\mathbf{p}_{B^{op}\otimes
A}N^\bullet\stackrel{\pi}{\longrightarrow}N^\bullet\longrightarrow
Z^\bullet\stackrel{+}{\longrightarrow}$
\end{center}
afforded by the homotopically projective resolution of $N^\bullet$.
Then $Z^\bullet$ is acyclic and the homotopically flat condition of
$F^\bullet$ tells us that $(F^\bullet\Lt_B?)(N^\bullet
)=F^\bullet\Tt_B\mathbf{p}_{B^{op}\otimes
A}N^\bullet\stackrel{1\otimes\pi}{\longrightarrow}F^\bullet\Tt_BN^\bullet$
is a quasi-isomorphism and, therefore, an isomorphism in $\D(A)$.
\end{proof}

 In order to understand triangulated functors between derived
categories of bimodules as 'one-sided' triangulated functors, it is
important to know how homotopically projective or injective
resolutions behave with respect to the restriction functors. The
following result and its left-right symmetric are important in this
sense.

\begin{lemma} \label{lem.restrictions of homotopically ...}
Let $A$ and $B$ be algebras. The following assertions hold:

\begin{enumerate}
\item If $A$ is $k$-flat, then the forgetful functor $\K(A^{op}\otimes
B)\longrightarrow\K(B)$ preserves homotopically injective complexes
and takes homotopically projective complexes to homotopically flat
ones.
\item If $A$ is $k$-projective, then the forgetful functor $\K(A^{op}\otimes
B)\longrightarrow\K(B)$ preserves homotopically projective
complexes.
\end{enumerate}
\end{lemma}
\begin{proof}
Assertion (2) and the part of assertion (1) concerning the preservation
of homotopically injective complexes are proved in \cite{NS} in a
much more general contex. Suppose now that $A$ is $k$-flat and let
$P^\bullet\in\K(A^{op}\otimes B)$ be homotopically projective. Using
\cite[Theorem P]{Ke}, we can assume that $P^\bullet$ admits a
countable filtration

\begin{center}
$P_0^\bullet\subset P_1^\bullet\subset ...\subset P_n^\bullet\subset
...$,
\end{center}
where the inclusions are inflations and where $P^\bullet _0$ and all
quotients $P^\bullet_n/P^\bullet_{n-1}$ are direct summands of coproducts of stalk
complexes of the form $A\otimes B[r]$. That $A$ is $k$-flat implies
that  it is the direct limit in $\text{Mod}-k$ of a direct system of
finitely generated free $k$-modules (see \cite[Th\'eor\'eme
1.2]{L}). It follows that, in $\C(B)$, $A\otimes B[r]$ is a direct
limit stalk complexes of the form $B^{(k)}[r]$, all of which are
homotopically flat in $\C(B)$. Bearing in mind that the bifunctor
$?\Tt_B?$ preserves direct limits on both variables, one gets that
each stalk complex $A\otimes B[r]$ is homotopically flat in $\C(B)$,
and then one easily proves by induction that each $P^\bullet_n$ in
the filtration is homotopically flat in $\C(B)$. But, by definition,
the direct limit in $\C(B)$ of homotopically flat complexes is
homotopically flat.
\end{proof}

\subsection{Classical derived functors as components of a  bifunctor}

It would be a natural temptation  to believe that if
$T^\bullet\in\C(B^{op}\otimes A)$ and $M^\bullet\in\C(A^{op}\otimes
C)$, then we have isomorphisms $\T_A(T^\bullet ,M^\bullet )\cong
(T^\bullet\Lt_A?)(M^\bullet )$ and $\T_A(T^\bullet ,M^\bullet )\cong
(?\Lt_AM^\bullet )(T^\bullet )$ in $\D(B^{op}\otimes C)$. Similarly,
one could be tempted to believe that if $T^\bullet$ is as above and
$N^\bullet\in\C(C^{op}\otimes A)$, then one has isomorphisms
$\RH_A(T^\bullet ,N^\bullet )\cong\Rh_A(T^\bullet ,?)(N^\bullet )$
and $\RH_A(T^\bullet ,N^\bullet )\cong\Rh_A(?,N^\bullet )(T^\bullet
)$ in $\D(C^{op}\otimes B)$. However,  we need extra hypotheses to
guarantee that.

\begin{prop} \label{prop.balance of derived bifunctors}
Let $A$, $B$ and $C$ be $k$-algebras and let ${}_BT^\bullet_A$,
${}_AM ^\bullet_C$ and ${}_CN ^\bullet_A$ complexes of bimodules
over the indicated algebras. There exist canonical natural
transformations of triangulated functors:

\begin{enumerate}
\item $\T_A(T^\bullet ,?)\longrightarrow T^\bullet\Lt_A?:\D(A^{op}\otimes C)\longrightarrow\D(B^{op}\otimes
C)$;
\item $\T_A(?,M)\longrightarrow ?\Lt_AM^\bullet:\D(B^{op}\otimes A)\longrightarrow\D(B^{op}\otimes
C)$;
\item $\Rh_A(T^\bullet ,?)\longrightarrow\RH_A(T^\bullet ,?):\D(C^{op}\otimes A)\longrightarrow\D(C^{op}\otimes
B)$;
\item $\Rh_A(?,N^\bullet )\longrightarrow\RH_A(?,N^\bullet ):\D(B^{op}\otimes A)^{op}\longrightarrow\D(C^{op}\otimes
B)$
\end{enumerate}

Moreover, the following assertions hold

\begin{enumerate}
\item[a)] If $C$ is $k$-flat, then the natural transformations 1 and
3 are isomorphism;
\item[b)] if $B$ is $k$-flat, then the natural transformation 2 is
an isomorphism;
\item[c)] if $B$ is $k$-projective, then the natural transformation
4 is an isomorphism.
\end{enumerate}
\end{prop}
\begin{proof}
(1) By definition, we have $\T_A(T^\bullet ,M^\bullet
)=(\mathbf{p}_{B^{op}\otimes A}T^\bullet
)\Tt_A(\mathbf{p}_{A^{op}\otimes C}M^\bullet )$ and
$(T^\bullet\Lt_A?)(M^\bullet
)=T^\bullet\Tt_A(\mathbf{p}_{A^{op}\otimes C}M^\bullet )$. If
$\pi_T:(\mathbf{p}_{B^{op}\otimes A}T^\bullet )\longrightarrow
T^\bullet$ is the homotopically projective resolution, we clearly
have a chain map

\begin{center}
$\T_A(T^\bullet ,M^\bullet )=(\mathbf{p}_{B^{op}\otimes A}T^\bullet
)\Tt_A(\mathbf{p}_{A^{op}\otimes C}M^\bullet
)\stackrel{\pi_T\Tt1}{\longrightarrow}
T^\bullet\Tt_A(\mathbf{p}_{A^{op}\otimes C}M^\bullet
)=(T^\bullet\Lt_A?)(M^\bullet )$
\end{center}
That this  map defines a natural transformation $\T_A(T^\bullet
,?)\longrightarrow T^\bullet\Lt_A?$ is routine.

(2) follows as (1), by applying a left-right symmetric argument.

(3), (4) By definition again, we have $\RH_A(T^\bullet ,N^\bullet
)=\Th_A(\mathbf{p}_{B^{op}\otimes A}T^\bullet$
, $\mathbf{i}_{C^{op}\otimes A}N^\bullet)$, $\Rh_A(T^\bullet
,?)(N^\bullet )=\Th_A(T^\bullet ,\mathbf{i}_{C^{op}\otimes
A}N^\bullet)$ and $\Rh_A(?,N^\bullet )(T^\bullet
)=\Th_A(\mathbf{p}_{B^{op}\otimes A}T^\bullet ,N^\bullet )$. If
$\pi_T$ is as above and
$\iota_N:N^\bullet\longrightarrow\mathbf{i}_{C^{op}\otimes
A}N^\bullet$ is the homotopically injective resolution, we then have
obvious chain maps

\begin{center}
$\Th_A(\mathbf{p}_{B^{op}\otimes A}T^\bullet ,N^\bullet
)\stackrel{(\iota_N)_*}{\longrightarrow}\Th_A(\mathbf{p}_{B^{op}\otimes
A}T^\bullet ,\mathbf{i}_{C^{op}\otimes
A}N^\bullet)\stackrel{\pi_T^*}{\longleftarrow}\Th_A(T^\bullet
,\mathbf{i}_{C^{op}\otimes A}N^\bullet)$.
\end{center}
It is again routine to see that they induce natural transformations
of triangulated functors

\begin{center}
$\Rh_A(?,N^\bullet )\longrightarrow\RH_A(? ,N^\bullet )$ and
$\RH_A(T^\bullet ,?)\longleftarrow\Rh_A(T^\bullet ,?)$.
\end{center}

On the other hand, when $C$ is $k$-flat, by lemma
\ref{lem.restrictions of homotopically ...},  we know that the
forgetful functor $\K(A^{op}\otimes C)\longrightarrow\K(A^{op})$
(resp. $\K(C^{op}\otimes A)\longrightarrow\K(A)$) takes
homotopically projective objects to homotopically flat objects
(resp. preserves homotopically injective objects). In particular,
the morphisms $\pi_T\Tt 1$ and $\pi_T^*$ considered above, are both
quasi-isomorphisms, which gives assertion a). Assertion b) follows
from the part of a) concerning the derived tensor product by a
symmetric argument.

Finally, if $B$ is $k$-projective, then, by lemma
\ref{lem.restrictions of homotopically ...},
 the forgetful functor $\K(B^{op}\otimes
A)\longrightarrow \K(A)$ preserves homotopically projective objects.
Then the map $(\iota_N)_*$ above is a quasi-isomorphism.

\end{proof}

\begin{lemma} \label{lem.from TH to H}
Let $A$, $B$ and $C$ be $k$-algebras. The following assertions hold:

\begin{enumerate}
\item The assignment $(T^\bullet ,X^\bullet
)\rightsquigarrow\text{Hom}_{B^{op}}^\bullet
(T^\bullet,B)\otimes_B^\bullet X^\bullet$ is the definition on
objects of a bifunctor

\begin{center}
$\text{Hom}^\bullet _{B^{op}}(?,B)\otimes_B^\bullet
?:\mathcal{C}(B^{op}\otimes A)^{op}\times \mathcal{C}(B^{op}\otimes
C)\longrightarrow\mathcal{C}(A^{op}\otimes C)$
\end{center}
which preserves conflations and contractible complexes on each
variable.

\item There is a natural transformation of bifunctors

\begin{center}
$\psi :\text{Hom}^\bullet _{B^{op}}(?,B)\otimes_B^\bullet
?\longrightarrow \Th_{B^{op}}(?,?)$.
\end{center}
\end{enumerate}
\end{lemma}
\begin{proof}
Assertion (1) is routine and  left to the reader. As for assertion (2),
let us fix $T^\bullet\in\mathcal{C}(B^{op}\otimes A)$ and
$X^\bullet\in \mathcal{C}(B^{op}\otimes C)$. We need to define a map

\begin{center}
$\psi=\psi_{T,X}:\text{Hom}^\bullet_{B^{op}}(T^\bullet
,B)\otimes_B^\bullet
X^\bullet\longrightarrow\text{Hom}^\bullet_{B^{op}}(T^\bullet
,X^\bullet)$.
\end{center}
Let  $f\in\text{Hom}^\bullet_{B^{op}}(T^\bullet ,B)$ and $x\in
X^\bullet$ be homogeneous elements, whose degrees are denoted by
$|f|$ and $|x|$. We define $\psi (f\otimes
x)(t)=(-1)^{|f|\cdot|x|}f(t)x$, for all homogeneous elements
$f\in\Th_{B^{op}}(T^\bullet ,B)$, $x\in X^\bullet$ and $t\in
T^\bullet$, and leave to the reader the routine task of checking
that $\psi$ is a chain map of complexes of $A-C-$bimodules which is
natural on both variables.
\end{proof}

Assertion 1 in the previous lemma allows us to define the following
bifunctor, which is triangulated on both variables:

\begin{center}
$\mathbf{TH}(?,?):\mathcal{D}(B^{op}\otimes A)^{op}\times
\mathcal{D}(B^{op}\otimes C)\stackrel{\mathbf{p}_{B^{op}\otimes
A}\times\mathbf{p}_{B^{op}\otimes
C}}{\longrightarrow}\mathcal{H}(B^{op}\otimes A)^{op}\times
\mathcal{H}(B^{op}\otimes C)\stackrel{\text{Hom}^\bullet
_{B^{op}}(?,B)\otimes_B^\bullet ?}{\longrightarrow}\K(A^{op}\otimes
C)\stackrel{q}{\longrightarrow}\D(A^{op}\otimes C)$.
\end{center}

We have the following property of this bifunctor.

\begin{prop} \label{prop.natural transformation TH to H}
Let $A$, $B$ and $C$ be $k$-algebras and let
$T^\bullet\in\C(B^{op}\otimes A)$ and $Y^\bullet\in\C(B^{op}\otimes
C)$ be complexes of bimodules. The natural transformation of lemma
\ref{lem.from TH to H} induces:

\begin{enumerate}
\item  A natural transformation of bifunctors $\theta
:\mathbf{TH}(?,?)\longrightarrow\RH_{B^{op}}(?,?)$ and  natural
transformation $\theta_{Y^\bullet}:(?\Lt_BY^\bullet
)\circ\Rh_{B^{op}}(?,B)\longrightarrow\Rh_{B^{op}}(?,Y^\bullet )$ of
triangulated functors $\D(B^{op}\otimes
A)\longrightarrow\D(A^{op}\otimes C)$. When $A$ is $k$-projective,
there is a natural transformation
$\upsilon_{Y^\bullet}:(?\Lt_BY^\bullet
)\circ\Rh_{B^{op}}(?,B)\longrightarrow\mathbf{TH}(?,Y^\bullet )$
such that $\theta_{Y^\bullet}\cong\theta (?,Y^\bullet
)\circ\upsilon_{Y^\bullet}$. If, in addition, $C$ is $k$-flat, then
$\upsilon_{Y^\bullet}$ is a natural isomorphism.

\item When $C$ is $k$-flat, $\theta$ induces a natural transformation  $\theta^{T^\bullet}:\Rh_{B^{op}}(T^\bullet ,B)\Lt_B?\cong\mathbf{TH}(T^\bullet ,?)
\longrightarrow\RH_{B^{op}}(T^\bullet ,?)\cong\Rh_{B^{op}}(T^\bullet
,?)$ of triangulated functors $\D(B^{op}\otimes
C)\longrightarrow\D(A^{op}\otimes C)$.

\end{enumerate}
\end{prop}
\begin{proof}
All throughout the proof, for simplicity, whenever $M^\bullet$ is a
complex of bimodules, we will denote by
$\pi=\pi_M:\mathbf{p}M^\bullet \longrightarrow M^\bullet$ and $\iota
=\iota_M :M^\bullet\longrightarrow\mathbf{i}M^\bullet$,
respectively, the homotopically projective resolution and the
homotopically injective resolution, without explicitly mentioning
over which algebras it is a complex of bimodules. In each case, it
should be clear from the context which algebras they are.

(1) We define $\theta_{(T,Y)}$ as the composition

\begin{center}
$\mathbf{TH}(T^\bullet ,Y^\bullet )=\Th_{B^{op}}(\mathbf{p}T^\bullet
,B)\Tt_B\mathbf{p}Y^\bullet\stackrel{\psi}{\longrightarrow}\Th_{B^{op}}(\mathbf{p}T^\bullet
,\mathbf{p}Y^\bullet
)\stackrel{(\pi_Y)_*}{\longrightarrow}\Th_{B^{op}}(\mathbf{p}T^\bullet
,Y^\bullet
)\stackrel{(\iota_Y)_*}{\longrightarrow}\Th_{B^{op}}(\mathbf{p}T^\bullet
,\mathbf{i}Y^\bullet )=\RH_{B^{op}}(T^\bullet ,Y^\bullet )$.
\end{center}
Checking its naturality on both variables is left to the reader.

As definition of $\theta_{Y^\bullet}$, we take the composition:

\begin{center}
$[(?\Lt_BY)\circ\Rh_{B^{op}}(?,B)](T^\bullet
)=\mathbf{p}\Th_{B^{op}}(\mathbf{p}T^\bullet
,B)\Tt_BY^\bullet\stackrel{\pi\Tt
1}{\longrightarrow}\Th_{B^{op}}(\mathbf{p}T^\bullet
,B)\Tt_BY^\bullet\stackrel{\psi}{\longrightarrow}\Th_{B^{op}}(\mathbf{p}T^\bullet
,Y^\bullet )=\Rh_{B^{op}}(? ,Y^\bullet )(T^\bullet )$.
\end{center}
When $A$ is $k$-projective,  we have a natural isomorphism
$\Rh_{B^{op}}(?,Y^\bullet )\cong\RH_{B^{op}}(?,Y^\bullet )$ (see
proposition \ref{prop.balance of derived bifunctors} c)). On the
other hand, we have morphisms in $\K(A^{op}\otimes B)$

\begin{center}

 $\mathbf{TH}(T^\bullet ,Y^\bullet )=\Th_{B^{op}}(\mathbf{p}T^\bullet,B)\Tt_B\mathbf{p}Y^\bullet\stackrel{\pi_H\Tt
1}{\longleftarrow}\mathbf{p}\Th_{B^{op}}(\mathbf{p}T^\bullet
,B)\Tt_B\mathbf{p}Y^\bullet\stackrel{1\Tt\pi_Y}{\longrightarrow}\mathbf{p}\Th_{B^{op}}(\mathbf{p}T^\bullet
,B)\Tt_BY^\bullet
=[(?\Lt_BY)\circ\Rh_{B^{op}}(?,B)](T^\bullet )$,
\end{center}
where $\pi_H$ is taken for $H^\bullet
:=\Th_{B^{op}}(\mathbf{p}T^\bullet,B)$. The right morphism  is a
quasi-isomorphism since the $k$-projectivity of $A$ guarantees that
$\mathbf{p}\Th_{B^{op}}(\mathbf{p}T^\bullet ,B)$ is homotopically
projective in $\K(B)$. The evaluation of the natural transformation
$\upsilon_Y$ at $T^\bullet$ is $\upsilon_{Y,T}=(\pi_H\Tt 1)\circ
(1\Tt\pi_Y)^{-1} :[(?\otimes\Lt_BY)\circ\Rh_{B^{op}}(?,B)](T^\bullet
)\longrightarrow\mathbf{TH}(T^\bullet ,Y^\bullet )$. This is an
isomorphism if, and only if, $\pi_H\Tt 1$ is an isomorphism in
$\D(A^{op}\otimes C)$. For this to happen it is enough that $C$ be
$k$-flat, for then $\mathbf{p}Y^\bullet$ is homotopically flat in
$\K(B^{op})$.

The fact that, when $A$ is $k$-projective, we have an equality
$\theta_Y=\theta (?,Y)\circ\upsilon_Y$ follows from the
commutativity in  $\K(A^{op}\otimes B)$ of the following diagram:

\[
\xymatrix{
&& \left[(?\otimes^{\bf L}_{B}Y^\bullet)\circ \Rh_{B^{op}}(?,B)\right](T^\bullet)\ar@{=}[d] \\
 {\bf p}\Th_{B^{op}}({\bf p}T^\bullet ,B)\otimes_{B}{\bf p}Y^\bullet\ar[rr]^{1\otimes\pi_{Y}}\ar[d]^{\pi_{H}\otimes 1}&&
 {\bf p}\Th_{B^{op}}({\bf p}T^\bullet ,B\otimes_{B}Y^\bullet\ar[d]^{\pi_{H}\otimes 1} \\
{\bf TH}(T^\bullet,Y^\bullet)= \operatorname{Hom}_{B^{op}}({\bf p}T^\bullet,B)\otimes {\bf p}Y^\bullet\ar[rr]^{1\otimes\pi_{Y}}\ar[d]^{\psi}&& \operatorname{Hom}_{B^{op}}({\bf p}T^\bullet,B)\otimes Y^\bullet\ar[d]^{\psi} \\
\operatorname{Hom}_{B^{op}}({\bf p}T^\bullet,{\bf p}Y^\bullet)\ar[rr]^{(\pi_{Y})_{*}}&&\operatorname{Hom}_{B^{op}}({\bf p}T^\bullet,Y^\bullet)= \Rh_{B^{op}}(?,Y^\bullet)(T^\bullet)\ar[d]^{(\iota_{Y})_{*}} \\
&& \operatorname{Hom}_{B^{op}}({\bf p}T^\bullet, \iota
Y^\bullet)=\mathbf{H}_{B^{op}}(T^\bullet, Y^\bullet) }
\]

(2) Note that, by definition of the involved functors,  we have an equality $\Rh_{B^{op}}(T,B)\Lt_B?=\mathbf{TH}(T^\bullet
 ,?)$. Then assertion (2) is a consequence of assertion (1) and proposition \ref{prop.balance of derived
 bifunctors} a).
\end{proof}

\begin{defi} \label{def.dual wrt B}
Let $A$ and $B$ be $k$-algebras. We shall denote by $(?)^*$ both
functors $\Rh_{B}(?,B):\D(A^{op}\otimes
B)^{op}\longrightarrow\D(B^{op}\otimes A)$ and
$\Rh_{B^{op}}(?,B):\D(B^{op}\otimes
A)^{op}\longrightarrow\D(A^{op}\otimes B)$.
\end{defi}
Note that, with the appropriate interpretation, the functors $(?)^*$
are adjoint to each other due to corollary \ref{cor.classical
adjunctions of derived functors}. In particular, we have the
corresponding unit $\sigma :1\longrightarrow (?)^{**}$ for the two
possible compositions. The following result is now crucial for us:

\begin{prop} \label{prop.all nat.transformations are isos}
Let $A$, $B$ and $C$ be $k$-algebras, let $T^\bullet$ be a complex
of $B-A-$bimodules such that ${}_BT^\bullet$  is compact in
$\D(B^{op})$ and suppose that $A$ is $k$-projective. The following
assertions hold:

\begin{enumerate}
\item For each complex $Y^\bullet$ of $B-C-$bimodules, the map $\theta_Y(T^\bullet ):(?\Lt_BY)(T^{\bullet *})=
[(?\Lt_BY)\circ\Rh_{B^{op}}(?,B)](T^\bullet )
\longrightarrow\Rh_{B^{op}}(?,Y^\bullet )(T^\bullet )$ is an
isomorphism in $\D(A^{op}\otimes C)$;

\item   $T^{\bullet *}_B$ is compact in $\D(B)$ and he map $\sigma_T:T^\bullet\longrightarrow T^{\bullet **}$ is an
isomophism in $\D(B^{op}\otimes A)$;

\item When in addition $C$ is $k$-flat, the following assertions
hold:

\begin{enumerate}
\item There is a natural isomorphism of triangulated functors $T^{\bullet
*}\Lt_B?\cong\Rh_{B^{op}}(T^\bullet ,?):\D(B^{op}\otimes
C)\longrightarrow\D(A^{op}\otimes C)$;
\item There is a natural isomorphism of triangulated functors $?\Lt_BT^\bullet
\cong\Rh_{B}(T^{\bullet *} ,?):\D(C^{op}\otimes
B)\longrightarrow\D(A^{op}\otimes C)$.
\end{enumerate}
\end{enumerate}
\end{prop}
\begin{proof}
All throughout the proof we use the fact that
$\text{per}(B)=\text{thick}_{\D(B)}(B)$ and similarly for $B^{op}$
(see proposition \ref{prop.compact objects in D(A)} below).

(1) For any $k$-projective algebra $A$, let us put
$F_A:=(?\Lt_BY^\bullet )\circ\Rh_{B^{op}}(?,B)$ and
$G_A:=\Rh_{B^{op}}(?,Y^\bullet)$, which are triangulated functors
$\D(B^{op}\otimes A)^{op}\longrightarrow\D(A^{op}\otimes C)$. Since
the forgetful functors $\K(B^{op}\otimes
A)\longrightarrow\K(B^{op})=\K(B^{op}\otimes k)$ and
$\K(A^{op}\otimes C)\longrightarrow\K(C)=\K(C\otimes k)$ preserve
homotopically projective objects, we have the following commutative
diagrams (one for $F$ and another one for $G$), where the vertical
arrows are the forgetful (or restriction of scalars) functors:

\[
\xymatrix{
\D(B^{op}\otimes A)^{op}\ar@<1ex>[rr]^{F_{A}}\ar@<-1ex>[rr]_{G_{A}}\ar[d]^{U_{B^{op}}} && \D(A^{op}\otimes C)\ar[d]^{U_{C}} \\
\D(B^{op})^{op}\ar@<1ex>[rr]^{F_{k}}\ar@<-1ex>[rr]_{G_{k}} && \D(C)
}
\]

We shall denote $\theta_{Y^\bullet}(T^\bullet )$ by
$\theta^A_{T^\bullet ,Y^\bullet }:F_A(T^\bullet )\longrightarrow
G_A(T^\bullet )$ to emphasize that there is one for each choice of a
$k$-projective algebra $A$. Note that, by proposition
\ref{prop.natural transformation TH to H},  if $C$ is not $k$-flat
we cannot guarantee that $\theta^A_{?,?}$ is a natural
transformation of bifunctors. Recall that $\theta^A_{T^\bullet
,Y^\bullet}$ is the composition

\begin{center}
$F_A(T^\bullet )=[(?\Lt_BY)\circ\Rh_{B^{op}}(?,B)](T^\bullet
)=\mathbf{p}\Th_{B^{op}}(\mathbf{p}T^\bullet
,B)\Tt_BY^\bullet\stackrel{\pi\Tt
1}{\longrightarrow}\Th_{B^{op}}(\mathbf{p}T^\bullet
,B)\Tt_BY^\bullet\stackrel{\psi}{\longrightarrow}\Th_{B^{op}}(\mathbf{p}T^\bullet
,Y^\bullet )=\Rh_{B^{op}}(? ,Y^\bullet )(T^\bullet )=G_A(T^\bullet
)$.
\end{center}
 When applying the forgetful functor
$U_C:\D(A^{op}\otimes C)\longrightarrow\D(C)$, we obtain
$U_C(\theta_{T,Y}^A):=(U_C\circ F_A)(T^\bullet)\longrightarrow
(U_C\circ G_A)(T^\bullet )$. Due to the commutativity of the above
diagram, this last morphism can be identified with the morphism
$\theta_{T ,Y }^{k}$, i.e., the version of $\theta_Y$ obtained when
taking $A=k$. But observe that $\theta_{B,Y}^k$ is an isomorphism.
This implies that $\theta_{X,Y}^k$ is an isomorphism, for each
$X^\bullet\in\text{thick}_{\D(B^{op})}(B)=\text{per}(B^{op})$. It
follows that $\theta_{T ,Y }^{k}=U_C(\theta_{T ,Y }^{A})$ is an
isomorphism since $_BT^\bullet\in\text{per}(B^{op})$. But then
$\theta_{T ,Y }^{A}$ is an isomorphism because  $U_C$ reflects
isomorphisms.

(2) Due to the $k$-projectivity of $A$, we have the following
commutative diagram, where  the lower horizontal arrow is the
version of $(?)^*$ when $A=k$.

\[
\xymatrix{
\D(B^{op}\otimes A)^{op}\ar[rr]^{(?)^{*}}\ar[d]^{U_{B^{op}}} && \D(A^{op}\otimes B)\ar[d]^{U_{B}}\\
\D(B^{op})^{op}\ar[rr]^{(?)^{*}}  && \D(B) }
\]

The task is hence reduced to check that the lower horizontal arrow
preserves compact objects. That is a direct consequence of the fact
that $(_BB)^*=\Rh_{B^{op}}(?,B)(B)\cong B_B$ and that the full
subcategory of $\D(B^{op})$ consisting of the $X^\bullet$ such that
$X^{\bullet *}\in\text{per}(B)$ is a thick subcategory of
$\D(B^{op})$.

In order to prove that $\sigma=\sigma_{T^\bullet}$ is an isomorphism
there is no loss of generality in assuming that $A=k$. Note that the
full subcategory $\mathcal{X}$ of $\D(B^{op})$ consisting of the
$X^\bullet$ for which $\sigma_{X^\bullet}:X^\bullet\longrightarrow
X^{\bullet
**}$ is an isomorphism is a thick subcategory. We just need to prove
that  ${}_BB\in\mathcal{X}$ since
${}_BT^\bullet\in\text{per}(B^{op})$. We do it by applying lemma
\ref{lem.units are isomorphisms} with ${}_BT^\bullet_A ={}_BB_B$.
Indeed, $\lambda_B:B\longrightarrow\Rh_B(B,?\Lt_BB)(B)$ is an
isomorphism in $\D(B)$.

(3) Assume now that $C$ is $k$-flat. Due to assertion (2), assertions
(3)(a) and (3)(b) are left-right symmetric. We then prove (3)(a).
 Proposition \ref{prop.natural
transformation TH to H} allows us to identify
$\theta_{Y^\bullet}=\theta_{?,Y^\bullet }$ and $\theta^{T^\cdot
}=\theta_{T^\bullet ,?}$, where $\theta
:\mathbf{TH}(?,?)\longrightarrow\mathbf{H}_{B^{op}}(?,?)$ is the
natural transformation of bifunctors given in that proposition. We
then have the following chain of double implications:

\begin{center}
$\theta^{T^\bullet}$ natural isomorphism  $\Longleftrightarrow$
$\theta_{T^\bullet ,Y^\bullet}$ isomorphism, for each
$Y^\bullet\in\D(B^{op}\otimes C)$,  $\Longleftrightarrow$
$\theta^{Y^\bullet }(T^\bullet )$ isomorphism, for each
$Y^\bullet\in\D(B^{op}\otimes C)$.
\end{center}
Now use assertion (1).
\end{proof}

As a direct consequence of the previous proposition, we have:

\begin{cor} \label{cor.duality between perfect complexes}
Let $A$ and $B$ be $k$-algebras, where $A$ is $k$-projective,  and
denote by $\mathcal{L}_{B,A}$ (resp $\mathcal{R}_{A,B}$) the full
subcategory of $\D(B^{op}\otimes A)$ (resp. $\D(A^{op}\otimes B)$)
consisting of the objects $X^\cdot$ such that ${}_BX^\bullet$ (resp.
$X^\bullet_B)$ is compact in $\D(B^{op})$ (resp. $\D(B)$). Then
$\mathcal{L}_{B,A}$ and $\mathcal{R}_{A,B}$ are thick subcategories
of $\D(B^{op}\otimes A)$ and $\D(A^{op}\otimes B)$, respectively.
Moreover, the assignments $X^\bullet\rightsquigarrow X^{\bullet *}$
define quasi-inverse triangulated dualities (=contravariant
equivalences)
$\mathcal{L}_{B,A}\stackrel{\cong}{\longleftrightarrow}\mathcal{R}_{A,B}$.
\end{cor}

\subsection{A few remarks on dg algebras and their derived categories}

In a few places in this paper, we will need some background on dg
algebras and their derived categories. We just give an outline of
the basic facts that we need. Interpreting dg algebras as dg
categories with just one object,  the material is a particular case
of the development in \cite{Ke} or \cite{Ke2}.

A \emph{dg algebra} is a ($\mathbb{Z}$-)graded algebra
$A=\oplus_{p\in\mathbb{Z}}A^p$ together with a \emph{differential}.
This is a graded $k$-linear map $d:A\longrightarrow A$ of degree $1$
such that $d(ab)=d(a)b+(-1)^{|a|}ad(b)$, for all homogeneous
elements $a,b\in A$, where $|a|$ denotes the degree of $a$, and such
that $d\circ d=0$. In such case a (right) \emph{dg $A$-module} is a
graded $A$-module $M^\bullet =\oplus_{p\in\mathbb{Z}}M^p$ together
with a $k$-linear map $d_M:M^\bullet\longrightarrow M^\bullet$ of
degree $+1$, called the differential of $M^\bullet$, such that
$d_M(xa)=d_M(x)a+(-1)^{|x|}xd(a)$, for all homogeneous elements
$x\in M^\bullet$ and $a\in A$, and such that $d_M\circ d_M=0$. It is
useful to look at each dg $A$-module  as a complex
$...M^p\stackrel{d_M^p}{\longrightarrow}M^{p+1}\longrightarrow..$ of
$k$-modules with some extra properties. Note that an ordinary
algebra is just a dg algebra with grading concentrated in degree
$0$, i.e.  $A^p=0$ for $p\neq 0$. In that case a dg $A$-module is
just a complex of $A$-modules.

Let $A$ be a dg algebra in the rest of this paragraph. We denote by
$\text{Gr}-A$ the category of graded $A$-modules (and morphisms of
degree $0$). Note that, when $A$ is an ordinary algebra, we have
$\text{Gr}-A=(\text{Mod}-A)^\mathbb{Z}$, with the notation of
subsection 2.2. We next define a category $\C(A)$ whose objects are
the dg $A$-modules as follows. A morphism
$f:M^\bullet\longrightarrow N^\bullet$ in $\C(A)$ is a morphism in
$\text{Gr}-A$ which is a chain map of complex of $k$-modules, i.e.,
such that $f\circ d_M=d_M\circ f$. This category is abelian and
comes with a canonical shifting $?[1]:\C(A)\longrightarrow\C(A)$
which comes from the canonical shifting of $\text{Gr}-A$, by
defining $d_{M[1]}^n=-d_M^{n+1}$, for each $n\in\mathbb{Z}$. Note
that we have an obvious faithful forgetful functor
$\C(A)\longrightarrow\text{Gr}-A$, which is also dense since we can
interpret each graded $A$-module as an object of $\C(A)$ with zero
differential. Viewing the objects of $C(A)$ as complexes of
$k$-modules, we clearly have,  for each $p\in\mathbb{Z}$,  the
\emph{$p$-th homology functor}
$H^p:\C(A)\longrightarrow\text{Mod}-k$. A morphism
$f:M^\bullet\longrightarrow N^\bullet$ in $\C(A)$ is called a
\emph{quasi-isomorphism} when $H^p(f)$ is an isomorphism, for all
$p\in\mathbb{Z}$. A dg $A$-module $M^\bullet$ is called
\emph{acyclic} when $H^p(M)=0$, for all $p\in\mathbb{Z}$.

Given any dg $A$-module $X^\bullet$, we denote by $P^\bullet_X$ the
dg $A$-module which, as a graded $A$-module, is equal to
$X^\bullet\oplus X^\bullet[1]$, and where the differential, viewed
as a morphism $X^\bullet\oplus X^\bullet
[1]=P^\bullet_X\longrightarrow P^\bullet_X[1]=X^\bullet [1]\oplus X^\bullet [2]$
in $\text{Gr}-k$, is the 'matrix' $\begin{pmatrix}d_X & 1_{X[1]} \\
0 & d_{X[1]}
\end{pmatrix}$. Note that we have a canonical exact sequence $0\rightarrow X^\bullet\longrightarrow P^\bullet_X\longrightarrow X^\bullet[1]\rightarrow
0$ in $\C(A)$,  which splits in $\text{Gr}-A$ but not in $\C(A)$. A
morphism $f:M^\bullet\longrightarrow N^\bullet$ in $\C(A)$ is called
\emph{null-homotopic} when there exists a morphism $\sigma
:M^\bullet\longrightarrow N^\bullet [-1]$ in $\text{Gr}-A$ such that
$\sigma\circ d_M+d_N\circ\sigma =f$.

The following is the fundamental fact (see \cite{Ke} and
\cite{Ke2}):

\begin{prop} \label{prop.fundamental for dg algebras}
Let $A$ be a dg algebra. The following assertions hold:

\begin{enumerate}
\item $\C(A)$ has a structure of exact category where the
conflations are those exact sequences which become split when
applying the forgetful functor $\C(A)\longrightarrow\text{Gr}-A$;
\item The projective objects with respect to this exact structure
coincide with the injective ones, and they are the direct sums of dg
$A$-modules of the form $P^\bullet_X$. In particular $\C(A)$ is a
Frobenius exact category;
\item A morphism $f:M^\bullet\longrightarrow N^\bullet$ in $\C(A)$ factors through a
projective object if, an only it, it is \emph{null-homotopic}. The
stable category $\underline{\C(A)}$ with respect to the given exact
structure is denoted by $\K(A)$ and called the \emph{homotopy
category of $A$}. It is a triangulated category, where $?[1]$ is the
suspension functor and where the triangles are the images of
conflations by the projection functor $p:\C(A)\longrightarrow\K(A)$;
\item If $Q$ denotes the class of quasi-isomorphisms in $\C(A)$ and
$\Sigma :=p(Q)$, then $\Sigma$ is a multiplicative system in $\K(A)$
compatible with the triangulation and
$\C(A)[Q^{-1}]\cong\K(A)[\Sigma^{-1}]$. In particular, this latter
category is triangulated. It is  denoted by $\D(A)$ and called the
\emph{derived category of $A$};
\item If $\mathcal{Z}$ denotes the full subcategory of $\K(A)$
consisting of the acyclic dg $A$-modules, then $\mathcal{Z}$ is a
triangulated subcategory closed under taking coproducts and products
in $\K(A)$. The category  $\D(A)$ is equivalent, as a triangulated
category, to the quotient category $\K(A)/\mathcal{Z}$.
\end{enumerate}
\end{prop}

As in the case of an ordinary algebra, we do not have
set-theoretical problems with the just defined derived category. The
reason is the following:

\begin{prop} \label{prop.D(A) compactly generated}
Let $A$ be a dg algebra and let $\mathcal{Z}$ be the full
subcategory of $\K(A)$ consisting of acyclic dg $A$-modules. Then
the pairs $(\mathcal{Z},\mathcal{Z}^\perp )$ and
$({}^\perp\mathcal{Z},\mathcal{Z})$ are semi-orthogonal
decompositions of $\K(A)$. In particular, the canonical functor
$q:\K(A)\longrightarrow\D(A)$ has both a left adjoint
$p_A:\D(A)\longrightarrow\K(A)$ and a right adjoint
$i_A:\D(A)\longrightarrow\K(A)$.

Moreover, the category $\D(A)$ is compactly generated by $\{A\}$.
\end{prop}

As in the case of an ordinary algebra, we call $\mathbf{p}_A$ and
$\mathbf{i}_A$ the \emph{homotopically projective resolution
functor} and the \emph{homotopically injective resolution functor},
respectively. Also, the objects of ${}^\perp\mathcal{Z}$ are called
\emph{homotopically projective} and those of $\mathcal{Z}^\perp$ are
called \emph{homotopically injective}.

When $A$ is a dg algebra, its \emph{opposite dg algebra}  $A^{op}$
is equal to $A$ as a graded $k$-module, but  the multiplication is
given by $b^o\cdot a^o=(-1)^{|a||b|}(ab)^o$. Here we denote by $x^o$
any homogeneous element $x\in A$ when viewed as an element of
$A^{op}$. Then, given dg algebras $A$ and $B$, a dg $B-A-$module $T$
is just a dg right $B^{op}\otimes A-$module, and these dg
$B-A-$modules form the category $\C(B^{op}\otimes A)$.

With some care on the use of signs, the reader will have no
difficulty in extending to dg algebras and their derived categories
the results in the earlier paragraphs of this section, simply by
defining correctly the total tensor and total $\text{Hom}$ functors
bifunctors.  Indeed, if $A$, $B$ and $C$ are dg algebras and
${}_BT^\bullet_A$, ${}_AM^\bullet_C$ and ${}_CN^\bullet_A$ are dg
bimodules, then $T^\bullet\Tt_AM^\bullet$ is just the classical
tensor product of a graded right $A$-module and a graded left
$A$-module and the differential
$d=d_{T\Tt_AM}:(T\Tt_AM)\longrightarrow (T\Tt_AM)$ takes $t\otimes
m\rightsquigarrow d_T(t)\otimes m+(-1)^{|t|}t\otimes d_M(m)$, for
all homogeneous elements $t\in T^\bullet$ and $m\in M^\bullet$.

We define
$\Th_A(T^\bullet,N^\bullet)=\oplus_{p\in\mathbb{Z}}\text{Gr}(A)(M^\bullet
,N^\bullet [p])$. This has an obvious structure of graded
$C-B-$bimodule, where the homogeneous component of degree $p$ is
$\Th_A(T^\bullet ,N^\bullet )^p=\text{Gr}(A)(M^\bullet ,N^\bullet
[p])$, for each $p\in\mathbb{Z}$. We then define the differential
$d:\Th_A(T^\bullet ,N^\bullet )\longrightarrow\Th_A(T^\bullet
,N^\bullet )$ by taking $f\rightsquigarrow d_N\circ f
-(-1)^{|f|}f\circ d_M$, for each homogeneous element
$f\in\Th_A(T^\bullet ,N^\bullet )$.

The following is a very well-known fact (see \cite[Theorem 5.3]{Ke}
and \cite{R}):

\begin{prop} \label{prop.compact objects in D(A)}
Let $A$ be a dg algebra. The compact objects of $\D(A)$ are
precisely the objects of $\text{thick}_{\D(A)}(A)=:\text{per}(A)$.
When $A$ is an ordinary algebra, those are precisely the complexes
which are quasi-isomorphic to bounded complexes of finitely
generated projective $A$-modules.
\end{prop}

\subsection{Some special objects and a generalization of Rickard theorem}

\begin{defi}
Let $\D$ be a triangulated category with coproducts. An object $T$
of $\D$ is called:

\begin{enumerate}
\item \emph{exceptional} when $\D(T,T[p])=0$, for all $p\neq 0$;
\item \emph{classical tilting} when $T$ is  exceptional 
and $\{T\}$ is a set of compact generators of $\D$;
\item \emph{self-compact} when $T$ is a compact object of
$\text{Tria}_{\D}(T)$.
\end{enumerate}
\end{defi}

The following result is a generalization of Rickard theorem
(\cite{R}), which, as can be seen in the proof, is an adaptation of
an argument of  Keller that the authors used in \cite{NS2}.

\begin{teor} \label{teor.Keller-Rickard generalization}
Let $A$ be a $k$-flat dg algebra and $T^\bullet$ be a self-compact
and exceptional object of $\D(A)$. If
$B=\text{End}_{\D(A)}(T^\bullet )$ is the endomorphism algebra then,
up to replacement of $T^\bullet$ by an isomorphic object in $\D(A)$,
we can view $T^\bullet$ as a dg $B$-$A$-bimodules such that the
restriction of $\Rh_A(T^\bullet ,?):\D(A)\longrightarrow\D(B)$
induces an equivalence of triangulated categories
$\text{Tria}_{\D(A)}(T^\bullet
)\stackrel{\cong}{\longrightarrow}\D(B)$.
\end{teor}
\begin{proof}
We can assume that $T^\bullet$ is homotopically injective in
$\K(A)$. Then $\hat{B}=\text{End}^\bullet_A(T^\bullet
):=\Th_A(T^\bullet ,T^\bullet )$ is a dg algebra and $T^\bullet$
becomes a dg $\hat{B}-A-$bimodule in a natural way. Moreover,   by
\cite[Corollary 2.5]{NS2},  we get that $?\Lt_{\hat{B}}T^\bullet
:\D(\hat{B})\longrightarrow\D(A)$ induces an equivalence of
triangulated categories
$\D(\hat{B})\stackrel{\cong}{\longrightarrow}\mathcal{T}:=\text{Tria}_{\D(A)}(T^\cdot
)$.

On the other hand, at the level of homology,  we have:

\begin{center}
$H^n(\hat{B})=\K(\hat{B})(\hat{B},\hat{B}[n])\cong\D(\hat{B})(\hat{B},\hat{B}[n])\cong\mathcal{T}(T^\bullet
,T^\bullet [n])\cong\D(A)(T^\bullet ,T^\bullet [n])$,
\end{center}
for all $n\in\mathbb{Z}$. It follows that $H^n(\hat{B})=0$, for
$n\neq 0$, while  $H^0(\hat{B})\cong B$. We take the canonical
truncation of $\hat{B}$ at $0$, i.e., the dg subalgebra $\tau^{\leq
0}\hat{B}$ of $\hat{B}$ given, as a complex of $k$-modules, by

\begin{center}
$...\hat{B}^{-n}\longrightarrow ...\longrightarrow
\hat{B}^{-1}\longrightarrow\text{Ker}(d^0)\longrightarrow 0...$,
\end{center}
where $d^0:\hat{B}^0\longrightarrow\hat{B}^1$ is the $0$-th
differential of $\hat{B}$. Then $B=H^0(\tau^{\leq
0}\hat{B})=H^0(\hat{B})$ and we have quasi-isomorphism of dg
algebras $B\stackrel{p}{\longleftarrow}\tau^{\leq
0}\hat{B}\stackrel{j}{\hookrightarrow}\hat{B}$, where $p$ and $j$
are the projection and inclusion, respectively. Replacing $\hat{B}$
by $\tau^{\leq 0}\hat{B}$ and the dg bimodule
${}_{\hat{B}}T^\bullet_A$ by ${}_{\tau^{\leq 0}\hat{B}}T^\bullet_A$,
we can assume, without loss of generality, that $\hat{B}^p=0$, for
all $p>0$. We assume this in the sequel.

It is convenient now to take the homotopically projective resolution
$\pi :\mathbf{p}_{\hat{B}^{op}\otimes A}T^\bullet\longrightarrow
T^\bullet$ in $\K(\hat{B}^{op}\otimes A)$. We  then replace
$T^\bullet$ by $\mathbf{p}T^\bullet
:=\mathbf{p}_{\hat{B}^{op}\otimes A}T^\bullet$. In this way we lose
the homotopically injective condition of $T^\bullet_A$ in $\K(A)$,
but we win that ${}_{\hat{B}}T^\bullet$ is homotopically flat in
$\K(\hat{B}^{op})$ (this follows by the extension of lemma
\ref{lem.restrictions of homotopically ...} to dg algebras). Note,
however, that $\pi :\mathbf{p}T^\bullet\longrightarrow T^\bullet$
induces a natural isomorphism
$?\Lt_{\hat{B}}\mathbf{p}T^\bullet\cong ?\Lt_{\hat{B}}T^\bullet$ of
triangulated functors $\D(\hat{B})\longrightarrow\D(A)$.

We can view $p:\hat{B}\longrightarrow B$ as a quasi-isomorphism in
$\K(\hat{B})$ and then
$T^\bullet\cong\hat{B}\Tt_{\hat{B}}T^\bullet\stackrel{p\otimes
1_T}{\longrightarrow}B\Tt_{\hat{B}}T^\bullet$ is a quasi-isomorphism
since ${}_{\hat{B}}T^\bullet$ is homotopically flat. By this same
reason, we have that $B\Tt_{\hat{B}}T^\bullet\cong
B\Lt_{\hat{B}}T^\bullet$ (see lemma \ref{lem.homotopically flat
resolution}). Then we get a composition of triangulated equivalences

\begin{center}
$\D(B)\stackrel{p_*}{\longrightarrow}\D(\hat{B})\stackrel{?\Lt_{\hat{B}}T^\bullet}{\longrightarrow}\mathcal{T}$,
\end{center}
where the first arrow is the restriction of scalars along the
projection $p:\hat{B}\longrightarrow B$. This composition of
equivalences is clearly identified with the functor
$?\Lt_B(B\Lt_{\hat{B}}T^\bullet )=?\Lt_B(B\Tt_{\hat{B}}T^\bullet )$.

Replacing $\hat{T}^\bullet   $ by $B\Tt_{\hat{B}}T^\bullet$, we can
then assume that $T^\bullet$ is a dg $B$-$A$-bimodule such that
$?\Lt_BT^\bullet :\D(B)\longrightarrow\D(A)$ induces an equivalence
$\D(B)\stackrel{\cong}{\longrightarrow}\text{Tria}_{\D(A)}(T^\bullet
)$. A quasi-inverse of this equivalence is then the restriction of
$\Rh_A(T^\bullet ,?)$ to $\text{Tria}_{\D(A)}(T^\bullet )$.
\end{proof}


\section{Main results}
All throughout this section $A$ and $B$ are arbitrary $k$-algebras
and $T^\bullet$ is a complex of $B-A-$bimodules. We want to give
necessary and sufficient conditions of the functors $\Rh_A(T^\bullet
,?):\D(A)\longrightarrow\D(B)$ or $?\Lt_BT^\bullet
:\D(B)\longrightarrow\D(A)$ to be fully faithful. We start with two
very helpful auxiliary results.

The following is the final result of \cite{NS}. Note that condition
2 of the proposition was given only for $\D(B^{op})$, but this
condition holds if, and only if, it holds for $\D(k)$ since the
forgetful functor $\D(B^{op})\longrightarrow\D(k)$  reflects
isomorphisms.

\begin{prop} \label{prop.upper bounded of f.g. projectives}
Suppose that $H^p(T^\bullet )=0$, for $p>>0$. The following
assertions are equivalent:

\begin{enumerate}
 \item  $T^\bullet$ is isomorphic in $\D(B^{op})$ to an upper bounded
  complex of finitely generated projective left $B$-modules;
  \item The canonical morphism $B^\alpha\Lt_BT^\bullet\longrightarrow (B\Lt_BT^\bullet
  )^\alpha=T^{\bullet\alpha}$ is an isomorphism (in $\D(A)$, $\D(B^{op}$ or
  $\D(k)$), for each cardinal $\alpha$.
\end{enumerate}
\end{prop}

We can go  further in the direction of the previous proposition.
Recall that if $X^\bullet =(X^\bullet ,d)\in\C(\mathcal{A})$ is any
complex, where $\mathcal{A}$ is an abelian category, then its
\emph{left (resp. right) stupid truncation at $m$} is the complex
$\sigma^{\leq m}X^\bullet =\sigma^{<m+1}X^\bullet$:
$...X^k\stackrel{d^k}{\longrightarrow}...X^{m-1}\stackrel{d^{m-1}}{\longrightarrow}
X^m\longrightarrow 0\longrightarrow 0...$ (resp. $\sigma^{\geq
m}X^\bullet =\sigma^{>m-1}X^\bullet$: $...0\longrightarrow
0\longrightarrow X^{m}\stackrel{d^{m}}{\longrightarrow}
X^{m+1}...X^n\stackrel{d^n}{\longrightarrow}...$). Note that we have
a conflation $0\rightarrow \sigma^{>m}X^\bullet\hookrightarrow
X^\bullet\longrightarrow\sigma^{\leq m}X^\bullet\rightarrow 0$ in
$\C(\mathcal{A})$, and, hence, an induced triangle

\begin{center}
$\sigma^{\geq m}X^\bullet\longrightarrow
X^\bullet\longrightarrow\sigma^{<
m}X^\bullet\stackrel{+}{\longrightarrow}$
\end{center}
in $\K(\mathcal{A})$ and in $\D(\mathcal{A})$.

\begin{prop} \label{prop.tensor by compacts preserves products}
The following assertions are equivalent:

\begin{enumerate}
\item ${}_BT^\bullet$ is compact in $\D(B^{op})$;
\item The functor $?\Lt_BT^\bullet :\D(B)\longrightarrow\D(A)$ preserves
products;
\item The functor  $?\Lt_BT^\bullet :\D(B)\longrightarrow\D(A)$ has
a left adjoint.
\end{enumerate}

\end{prop}
\begin{proof}
$(2)\Longleftrightarrow (3)$ is a direct consequence of corollary
\ref{cor.adjoints with comp.generated domain}.

$(1)\Longrightarrow (2)$ It is enough to prove that the composition
$\D(B)\stackrel{?\Lt_BT^\bullet}{\longrightarrow}\D(A)\stackrel{\text{forgetful}}{\longrightarrow}\D(k)$
preserves products since the forgetful functor preserves products
and reflects isomorphisms.

Abusing of notation, we still denote by $?\Lt_BT^\bullet$ the
mentioned composition and note that, when doing so, we have
isomorphisms $(?\Lt_BN^\bullet )(X^\bullet
)\cong\mathbf{T}_B(X^\bullet ,N^\bullet )\cong
(X^\bullet\Lt_B?)(N^\bullet )$ in $\D(k)$, for all
$X^\bullet\in\D(B)$ and $N^\bullet\in\D(B^{op})$ (see proposition
\ref{prop.balance of derived bifunctors}).
 If $(X_i^\bullet )_{i\in I}$
is any family of objects of $\D(B)$, we consider the full
subcategory $\mathcal{T}$ of $\D(B^{op})$ consisting of the
$N^\bullet$ such that the canonical morphism $(\prod_{i\in
I}X^\bullet_i)\Lt_BN^\bullet\longrightarrow\prod_{i\in
I}X^\bullet_i\Lt_BN^\bullet$ is an isomorphism. It is a thick
subcategory of $\D(B^{op})$ which contains $_BB$. Then it contains
$\text{per}(B^{op})$ and, in particular, it contains $_BT^\bullet$.

$(2)\Longrightarrow (1)$ Without loss of generality, we can assume
that $A=k$. The canonical morphism
$\coprod_{p\in\mathbb{Z}}B[p]\longrightarrow\prod_{p\in\mathbb{Z}}B[p]$
is an isomorphism in $\D(B^{op})$. By applying $?\Lt_BT^\bullet$ and
using the hypothesis, we then have an isomorphism in $\D(k)$

\begin{center}
$\coprod_{p\in\mathbb{Z}}T^\bullet[p]\cong(?\Lt_BT^\bullet
)(\coprod_{p\in\mathbb{Z}}B[p])\stackrel{\cong}{\longrightarrow}(?\Lt_BT^\bullet
)(\prod_{p\in\mathbb{Z}}B[p])\cong\prod_{p\in\mathbb{Z}}(B[p]\Lt_BT^\bullet
)\cong\prod_{p\in\mathbb{Z}}T^\bullet [p]$,
\end{center}
which can be identified with the canonical morphism from the
coproduct to the product. It follows that the canonical map

\begin{center}
$\coprod_{p\in\mathbb{Z}}H^p(T^\bullet )\cong
H^0(\coprod_{p\in\mathbb{Z}}T^\bullet [p])\longrightarrow
H^0(\prod_{p\in\mathbb{Z}}T^\bullet
[p])\cong\prod_{p\in\mathbb{Z}}H^p(T^\bullet )$
\end{center}
is an isomorphism. This implies that $H^p(T^\bullet )=0$, for all
but finitely many $p\in\mathbb{Z}$.  

 By proposition \ref{prop.upper bounded of f.g.
projectives}, replacing  $T^\bullet$ by its homotopically projective resolution
in $\K(B^{op})$, we can assume without loss of generality that
$T^\bullet$ is an upper bounded complex of finitely generated projective left
$B$-modules.  Let us put $m:=\text{min}\{p\in\mathbb{Z}:$
$H^p(T)\neq 0\}$. We then consider the triangle in $\D(B^{op})$
induced by the stupid truncation at $m$

\begin{center}
$\sigma^{\geq m}T^\bullet\longrightarrow
T^\bullet\longrightarrow\sigma^{<m}T^\bullet\stackrel{+}{\longrightarrow}$.
\end{center}
Then $\sigma^{\geq m}T^\bullet$ is compact in $\D(B^{op})$ while
$\sigma^{<m}T^\bullet$ has homology concentrated in degree $m-1$.
Then we have an isomorphism $\sigma^{<m}T^\bullet\cong M[1-m]$ in
$\D(B^{op})$, where $M=H^{m-1}(\sigma^{<m}T^\bullet )$. By the
implication $1)\Longrightarrow 2)$ we know that $?\Lt_B\sigma^{\geq
m}T^\bullet :\D(B)\longrightarrow\D(k)$ preserves products and, by
hypothesis, also $?\Lt_BT^\bullet :\D(B)\longrightarrow\D(k)$ does.
It follows that $?\Lt_B\sigma^{<m}T^\bullet\cong
?\Lt_BM[1-m]:\D(B)\longrightarrow\D(k)$ preserves products.

Note that $M$ admits admits a projective resolution with finitely
generated terms, namely,  the canonical quasi-isomorphism
$P^\bullet:=\sigma^{<m}T^\bullet[m-1]\longrightarrow M=M[0]$.
Therefore our task reduces to check that if $M$ is a left $B$ module
which admits a projective resolution with finitely generated terms
and such that $?\Lt_BM:\D(B)\longrightarrow\D(k)$ preserves
products, then $M$ has finite projective dimension. For that it is
enough to prove that there is an integer $n\geq0$ such that
$\text{Tor}_{n+1}^B(?,M)\equiv 0$. Indeed, if  this is proved then
$\Omega^n(M):=\text{Im}(d^{-n}:P^{-n}\longrightarrow P^{-n+1})$ will
be a flat module, and hence projective (see \cite[Corollaire
1.3]{L}), thus ending the proof.

Let us assume by way of contradiction that
$\text{Tor}_{n}^B(?,M)\not\equiv 0$, for all $n>0$. For each such
$n$, choose a right $B$-module $X_n$ such that
$\text{Tor}_n^B(X_n,M)\neq 0$. Then the canonical morphism
$\coprod_{n>0}X_n[-n]\longrightarrow\prod_{n>0}X_n[-n]$ is an
isomorphism in $\D(B)$. Our hypothesis then guarantees that the
canonical morphism

\begin{center}
$\coprod_{n>0}(X_n[-n]\Lt_BM)\cong (\coprod_{n>0}X_n[-n])\Lt_BM\cong
(\prod_{n>0}X_n[-n])\Lt_BM\longrightarrow\prod_{n>0}(X_n[-n]\Lt_BM)$
\end{center}
is an isomorphism. When applying the $0$-homology functor $H^0$, we
obtain an isomorphism

\begin{center}
$\coprod_{n>0}\text{Tor}_n^B(X_n,M)\cong\coprod_{n>0}H^0(X_n[-n]\Lt_BM)\longrightarrow\prod_{n>0}H^0(X_n[-n]\Lt_BM)\cong\prod_{n>0}\text{Tor}_n^B(X_n,M)$
\end{center}
which is identified with the canonical morphism from the coproduct
to the product in $\text{Mod}-K$. It follows that
$\text{Tor}_n^B(X_n,M)=0$, for almost all $n>0$, which is a
contradiction.

\end{proof}

\begin{rem} \label{rem.Rickard}
The argument in the last two paragraphs of the proof of proposition
\ref{prop.tensor by compacts preserves products} was communicated to
us by Rickard, to whom we deeply thank for it. When passing to the
context of dg algebras or even dg categories, the implication
$(1)\Longrightarrow (2)$ in that proposition still holds (see
\cite{NS}), essentially with the same proof. However, we do not know
if $(2)\Longrightarrow (1)$ holds for dg algebras $A$ and $B$.
\end{rem}

\subsection{Statements and proofs}

\begin{prop} \label{prop.fully faithful RHom}
Let $\delta:(?\Lt_BT^\bullet )\circ\Rh_A(T^\bullet,?)\longrightarrow
1_{\D(A)}$ be the counit of the adjoint pair $(?\Lt_BT^\bullet
,\Rh_A(T^\bullet ,?))$. The following assertions are equivalent:

\begin{enumerate}
\item $\Rh_A(T^\bullet ,?):\D(A)\longrightarrow\D(B)$ is fully
faithful;
\item The  map $\delta_A:[(?\Lt_BT)\circ\Rh_A(T,?)](A)\longrightarrow A$ is
an isomorphism in $\D(A)$ and the functor
$(?\Lt_BT)\circ\Rh_A(T,?):\D(A)\longrightarrow\D(A)$ preserves
coproducts.
\end{enumerate}

In that case, the functor $?\Lt_BT^\bullet$ induces a triangulated
equivalence $\D(B)/\text{Ker}(?\Lt_BT^\bullet
)\stackrel{\cong}{\longrightarrow}\D(A)$.
\end{prop}
\begin{proof}
Assertion (1) is equivalent to saying that $\delta:(?\Lt_BT^\bullet
)\circ\Rh_A(T^\bullet,?)\longrightarrow 1_{\D(A)}$ is a natural
isomorphism (see the dual of \cite[Proposition II.7.5]{HS}). As a
consequence, the implication $(1)\Longrightarrow (2)$ is automatic.
Conversely, if assertion (2) holds, then the full subcategory
$\mathcal{T}$ of $\D(A)$ consisting of the $M^\bullet\in\D(A)$ such
that $\delta_{M^\bullet }$ is an isomorphism is a triangulated
subcategory closed under taking coproducts and containing $A$. It
follows that $\mathcal{T}=\D(A)$, so that assertion (1) holds.

For the final statement, note that $\Rh_A(T^\bullet ,?)$ gives an
equivalence of triangulated categories
$\D(A)\stackrel{\cong}{\longrightarrow}\text{Im}(\Rh_A(T^\bullet
,?))$. On the other hand, by proposition \ref{prop.semi-orthogonal
pair from adjunction}, we know that $(\text{Ker}(?\Lt_BT^\bullet
),\text{Im}(\Rh_A(T^\bullet ,?)))$ is a semi-orthogonal
decomposition of $\D(B)$. Then, by proposition \ref{prop.Verdier
localization by triangulated}, we have a triangulated equivalence
$\D(B)/\text{Ker}(?\Lt_BT^\bullet )\cong\text{Im}(\Rh_A(T^\bullet
,?))$. We then get a triangulated equivalence
$\D(B)/\text{Ker}(?\Lt_BT^\bullet
)\stackrel{\cong}{\longrightarrow}\D(A)$, which is easily seen to be
induced by $?\Lt_BT^\bullet$.
\end{proof}

We now pass to study the recollement situations where one of the
fully faithful functors is $\Rh_A(T^\bullet ,?)$.

\begin{cor} \label{cor.recollement D(A)=D(B)=D by RHom}
Let $T^\bullet$ a complex of $B-A-$bimodules. The following
assertions hold:

\begin{enumerate}
\item There is  a triangulated category $\D'$ and a recollement $\D(A)\equiv\D(B)\equiv
\D'$,  with $i_*=\Rh_A(T^\bullet ,?)$;
\item There is  a triangulated category $\D'$ and a recollement $\D(A)\equiv\D(B)\equiv
\D'$,  with $i^*=?\Lt_BT^\bullet$;
\item $T^\bullet_A$ is compact in $\D(A)$  and $\delta_A:[(?\Lt_BT^\bullet )\circ\Rh_A(T^\bullet ,?)](A)\longrightarrow
A$ is an isomorphism in $\D(A)$, where $\delta$ is the counit of the
adjoint pair $(?\Lt_BT^\bullet ,\Rh_A(T^\bullet ,?))$.

\end{enumerate}
\end{cor}
\begin{proof}
$(1)\Longleftrightarrow (2)$ is clear.

$(1)\Longrightarrow (3)$ If the
recollement exists, then $\Rh_A(T^\bullet
,?):\D(A)\longrightarrow\D(B)$ is fully faithful, so that $\delta_A$
is an isomorphism. Moreover, $\Rh_A(T^\bullet ,?)$ preserves
coproducts since it it is a left adjoint functor. This preservation
of coproducts is equivalent to  having
$T^\bullet_A\in\text{per}(A)$.

$(3)\Longrightarrow (1)$ We clearly have that the functor
$(?\Lt_BT^\bullet )\circ\Rh_A(T^\bullet
,?):\D(A)\longrightarrow\D(A)$ preserves coproducts.  Then assertion
(2) of proposition \ref{prop.fully faithful RHom} holds, so that
$\Rh_A(T^\bullet ,?)$ is fully faithful. On the other hand, by
proposition \ref{prop.Brown representability theorem}, we get that
$\Rh_A(T^\bullet ,?):\D(A)\longrightarrow\D(B)$ has a right adjoint,
so that assertion (1) holds.
\end{proof}

\begin{teor} \label{teor.recollement D=D(B)=D(A) with Rh}
Let ${}_BT^\bullet_A$ be a complex of $B-A-$bimodules. Consider the
following assertions:

\begin{enumerate}
\item There is a recollement $\D'\equiv\D(B)\equiv\D(A)$, with $j_*=\Rh_A(T^\bullet
,?)$, for some triangulated category (which is equivalent to
$\D(C)$, where $C$ is a dg algebra);
\item There is a recollement $\D'\equiv\D(B)\equiv\D(A)$, with
$j^*=j^!=?\Lt_BT^\bullet$, for some triangulated category (which is
equivalent to $\D(C)$, where $C$ is a dg algebra);
\item The following three conditions hold:

\begin{enumerate}
\item The counit map $\delta_A:[(?\Lt_BT^\bullet )\circ\Rh_A(T^\bullet ,?)](A)\longrightarrow
A$ is an isomorphism;
\item The functor $(?\Lt_BT^\bullet )\circ\Rh_A(T^\bullet
,?):\D(A)\longrightarrow\D(A)$ preserves coproducts;
\item The functor $?\Lt_BT^\bullet :\D(B)\longrightarrow\D(A)$
preserves products.
\end{enumerate}

\item ${}_BT^\bullet$ is compact and exceptional in $\D(B^{op})$ and
the canonical algebra morphism
$A\longrightarrow\text{End}_{\D(B^{op})}(T^\bullet )^{op}$ is an
isomorphism.
\end{enumerate}
The implications $(1)\Longleftrightarrow (2)\Longleftrightarrow
(3)\Longleftarrow (4)$ hold true and, when $A$ is $k$-projective, all
assertions are equivalent.  Moreover, if $B$ is $k$-flat, then the
dg algebra $C$ can be chosen together with a homological epimorphism
$f:B\longrightarrow C$ such that $i_*$ is the restriction of scalars
$f_*:\D(C)\longrightarrow\D(B)$.
\end{teor}
\begin{proof}
$(1)\Longleftrightarrow (2)$ is clear.

$(1)\Longleftrightarrow (3)$ By proposition \ref{prop.functor in a
recollement}, we know that the recollement in (3) exists if, and only
if, $?\Lt_BT^\bullet :\D(B)\longrightarrow\D(A)$ has a left adjoint
and $\Rh_A(T^\bullet ,?):\D(A)\longrightarrow\D(B)$ is fully
faithful. Apply now corollary \ref{cor.adjoints with comp.generated
domain} and proposition \ref{prop.fully faithful RHom}.

$(4)\Longrightarrow (3)$ Condition (3)(c) follows from proposition
\ref{prop.tensor by compacts preserves products}. By proposition
\ref{prop.fully faithful RHom}, proving conditions (3)(a) and (3)(b) is
equivalent to proving  that $\Rh_A(T^\bullet ,?)$ is fully faithful.
This is in turn equivalent to proving that the counit $\delta
:(?\Lt_BT^\bullet )\circ\Rh_A(T^\bullet ,?)\longrightarrow
1_{\D(A)}$ is a natural isomorphism.

In order to prove this, we apply proposition \ref{prop.all
nat.transformations are isos} to $T^\bullet$, when viewed as a
complex of left $B$-modules (equivalently, of $B-k-$bimodules). Then
$T^{\bullet
*}$ is obtained from $T^\bullet$ by applying
$\Rh_{B^{op}}(?,B):\D(B^{op})^{op}\longrightarrow\D(B)$ and,
similarly, we obtain $T^{\bullet **}$ from $T^{\bullet *}$. By the
mentioned proposition, we know that $T^\bullet \cong T^{\bullet **}$
in $\D(B^{op})$. Moreover, applying assertion (1) of that proposition,
with $A$ and $C$ replaced by $k$ and $A$, respectively, and putting
$Y^\bullet =T^\bullet\in\C(B^{op}\otimes A)$, we obtain isomorphisms
in $\D(A)$:

\begin{center}
$T^{\bullet *}\Lt_BT^\bullet :=[(?\Lt_BT^\bullet
)\circ\Rh_{B^{op}}(?,B)](T^\bullet )\cong\Rh_{B^{op}}(?,T^\bullet
)(T^\bullet )\cong A_A$,
\end{center}
where the last isomorphism follows from the exceptionality of
${}_BT^\bullet$ in $\D(B^{op})$ and the fact  that the canonical
algebra morphism $A\longrightarrow\text{End}_{\D(B^{op})}(T^\bullet
)^{op}$ is an isomorphism.

Using the previous paragraph, proposition \ref{prop.all
nat.transformations are isos} and adjunction, for each object
$M^\bullet\in\D(A)$ we  get a chain of isomorphisms in $\D(k)$:

\begin{center}
$[(?\Lt_BT^\bullet )\circ\Rh_A(T^\bullet ,?)](M^\bullet )=
(?\Lt_BT^\bullet)(\Rh_A(T^\bullet ,M^\bullet
))\stackrel{\cong}{\longrightarrow}\Rh_B(T^{\bullet
*},\Rh_A(T^\bullet ,M^\bullet ))\stackrel{\cong}{\longrightarrow}\Rh_A(T^{\bullet *}\Lt_BT^\bullet ,M^\bullet )
\stackrel{\cong}{\longrightarrow}\Rh_A(A,M^\bullet )
\stackrel{\cong}{\longrightarrow}M^\bullet $.
\end{center}
It is routine now  to see that the composition of these isomorphisms
is obtained from the counit map $\delta_{M^\bullet
}:[(?\Lt_BT^\bullet )\circ\Rh_A(T^\bullet ,?)](M^\bullet
)\longrightarrow M^\bullet $ by applying the forgetful functor
$\D(A)\longrightarrow\D(k)$. Since this last functor reflects
isomorphisms we get that $\delta$ is a natural isomorphism, so that
$\Rh_A(T^\bullet ,?):\D(A)\longrightarrow\D(B)$ is a fully faithful
functor.

$(1)-(3)\Longrightarrow (4)$ (Assuming that $A$ is $k$-projective).  By
proposition  \ref{prop.tensor by compacts preserves products}, we
know that ${}_BT^\bullet\in\text{per}(B^{op})$. Then, by proposition
\ref{prop.all nat.transformations are isos}, we have a natural
isomorphism  $?\Lt_BT^{\bullet}\cong\Rh_B(T^{\bullet *},?)$ of
triangulated functors $D(B)\longrightarrow\D(A)$, where now we are
considering $T^{\bullet
*}$ as obtained from $T^\bullet$ by applying $(?)^*:=\Rh_{B^{op}}(?,B):\D(B^{op}\otimes A)^{op}\longrightarrow\D(A^{op}\otimes
B)$. By lemma \ref{lem.fully-faitful right iff left}, the fully
faithful condition of $\Rh_A(T^\bullet ,?)$ implies the same
condition  for $?\Lt_AT^{\bullet *}$. Then corollary
\ref{cor.Chen-Xi} below says that $T^{\bullet *}_B$ is compact and
exceptional in $\D(B)$ and the canonical algebra morphism
$A\longrightarrow\text{End}_{\D(B)}(T^{\bullet *})$ is an
isomorphism.

Note now that, due to the $k$-projectivity of $A$, when applying to
${}_BT^\bullet$ the functor
$\Rh_{B^{op}}(?,B):\D(B^{op})^{op}\longrightarrow\D(B)$, we obtain
an object isomorphic to $T^{\bullet *}_B$ in $\D(B)$, and conversely.
Applying now corollary \ref{cor.duality between perfect complexes}
with $A=k$, we have that ${}_BT^\bullet\cong {}_BT^{\bullet
**}$ is exceptional in $\D(B^{op})$ and that the algebra map
$A\longrightarrow\text{End}_{\D(B^{op})}(T^\bullet
)^{op}\cong\text{End}_{\D(B)}(T^{\bullet *})$ is an isomorphism.

For the final statement note that, when the recollement of
assertions (1) or (2) exists, $\D'$ has a compact generator, namely
$i^!(A)$. Then, by \cite[Theorem 4.3]{Ke}, we know that
$\D'\cong\D(C)$, for some dg algebra $C$. When $B$ is $k$-flat, the
fact that this dg algebra can be chosen together with a homological
epimorphism $f:B\longrightarrow C$ satisfying the requirements is a
direct consequence of \cite[Theorem 4]{NS2}.
\end{proof}

\begin{rem} \label{rem.Bazzoni-Pavarin}
Due to the results in \cite{NS}, except for the implication
$(1)-(3)\Longrightarrow (4)$,  theorem \ref{teor.recollement D=D(B)=D(A)
with Rh} is also true in the context of dg categories, with the
proof adapted. In that case its implication $(4)\Longrightarrow (3)$
partially generalizes \cite[Theorem 4.3]{BP} in the sense that we
explicitly prove that the recollement exists with
$j^*=j^!=?\Lt_BT^\bullet$. Note, however, that if $T^{\bullet
*}$ is the dg $A-B-$module obtained from $T^\bullet$ by application
of the functor $\Rh_{B^{op}}(?,B):\D(B^{op}\otimes
A)\longrightarrow\D(A^{op}\otimes B)$, we cannot guarantee that the
left adjoint of $?\Lt_BT^\bullet:\D(B)\longrightarrow\D(A)$ is
(naturally isomorphic to) $?\Lt_AT^{\bullet *}$. Due to the version
of proposition \ref{prop.all nat.transformations are isos} for dg
algebras, we can guarantee that when $A$ is assumed to be
$k$-projective. Note that, for the entire theorem
\ref{teor.recollement D=D(B)=D(A) with Rh} to be true in the context
of dg algebras(or even dg categories), one only needs to prove that
the implication $(2)\Longrightarrow (1)$ of proposition
\ref{prop.tensor by compacts preserves products} holds in this more
general context. The rest of the proof of theorem
\ref{teor.recollement D=D(B)=D(A) with Rh} can be extended without
problems.
\end{rem}

\begin{exem}
Let $A$ be a  hereditary  Artin  algebra and let $S$ be a
non-projective simple module. Then $T=A\oplus S$ is a right
$A$-module, so that $T$ becomes a $B-A-$bimodule, where
$B=\text{End}(T_A)\cong\begin{pmatrix} A & 0\\ S & D \end{pmatrix}$,
where $D=\text{End}(S_A)$. There  are  a recollement
$\D(A)\equiv\D(B)\equiv\D'$, with $i_*=\Rh_A(T,?)$, and a
recollement $\D''\equiv\D(B)\equiv\D(A)$, with $j_*=\Rh_A(T,?)$, for
some triangulated categories $\D'$ and $\D''$. However $T_A$ is not
exceptional in $\D(A)$.

\begin{proof}
It is well-known that $\text{Ext}_A^1(S,S)=0$, which implies that
$\text{Ext}_A^1(T,T)\cong\text{Ext}_A^1(S,A)\neq 0$ and, hence, that
$T_A$ is not exceptional in $\D(A)$.

 We denote by $e_i$ ($i=1,2$) the
canonical idempotents of $B$.  We readily see that ${}_BT\cong
Be_1$, that $\text{Hom}_A(T,A)\cong e_1B$ and that
$\text{Ext}_A^1(T,A)$ is isomorphic to $\begin{pmatrix} 0 &
\text{Ext}_A^1(S,A)\end{pmatrix}$, when viewed as a right $B$-module
in the usual way (see, e.g., \cite[Proposition III.2.2]{ARS}). In
particular, we have $\text{Ext}_A^1(T,A)e_1=0$. We then get a
triangle in $\D(B)$:

\begin{center}
$e_1B[0]\longrightarrow\Rh_A(T,A)\longrightarrow
\text{Ext}_A^1(T,A)[-1]\stackrel{+}{\longrightarrow}$.
\end{center}
When applying $?\Lt_BT=?\Lt_BBe_1=?\otimes_BBe_1$,  we get  a
triangle in $\D(A)$:

\begin{center}
 $A=e_1Be_1[0]\longrightarrow \Rh_A(T,A)\Lt_BT\longrightarrow
 \text{Ext}_A^1(T,A)e_1[-1]=0\stackrel{+}{\longrightarrow}$.
\end{center}
We then get an isomorphism
$A\stackrel{\cong}{\longrightarrow}\Rh_A(T,A)\Lt_BT$ in $\D(A)$,
which is easily seen to be inverse to $\delta_A$.  Then assertion (3)
of corollary \ref{cor.recollement D(A)=D(B)=D by RHom} holds.

On the other hand, also condition (4) of theorem \ref{teor.recollement
D=D(B)=D(A) with Rh} holds.
\end{proof}

\end{exem}

Although the exceptionality property is not needed, it helps to
extract more information about $T^\bullet$, when $\Rh_A(T^\bullet
,?)$ is fully faithful. The following is an example:

\begin{prop} \label{prop.faithful Rh for exceptional}
Let $T^\bullet$ be a complex of $B-A-$ bimodules such that
$\Rh_A(T^\bullet ,?):\D(A)\longrightarrow\D(B)$ is fully faithful.
The following assertions hold:

\begin{enumerate}
\item If $\D(A)(T^\bullet ,T^\bullet [n])=0$, for all but finitely
many $n\in\mathbb{Z}$, then $H^p(T^\bullet )=0$, for all but
finitely many $p\in\mathbb{Z}$;
\item If $T^\bullet_A$ is exceptional in $\D(A)$ and the algebra morphism $B\longrightarrow\text{End}_{\D(A)}(T^\bullet
)$ is an isomorphism, then ${}_BT^\bullet$ is isomorphic in
$\D(B^{op})$ to an upper bounded complex of finitely generated
projective left $B$-modules (with  bounded homology).
\end{enumerate}
\end{prop}
\begin{proof}
(1) Let us put $X^\bullet =\Rh_A(T^\bullet ,T^\bullet )$. The
hypothesis says that $H^n(X^\bullet )=0$, for all but finitely many
$n\in\mathbb{Z}$. In particular the canonical morphism
$f:\coprod_{n\in\mathbb{Z}}X^\bullet
[n]\longrightarrow\prod_{n\in\mathbb{Z}}X^\bullet [n]$ is an
isomorphism in $\D(B)$ since $H^p(f)$ is an isomorphism, for all
$p\in\mathbb{Z}$. On the other hand, due to the adjunction equations
and the fact the the counit $\delta : (?\Lt_BT^\bullet
)\circ\Rh_A(T^\bullet ,?)\longrightarrow 1_{\D(A)}$ is a natural
isomorphism, the unit $\lambda
:1_{\D(B)}\longrightarrow\Rh_A(T^\bullet ,?)\circ (?\Lt_BT^\bullet
)$ satisfies that
$\lambda_{M^\bullet}\Lt_B1_{T^\bullet}:=(?\Lt_BT^\bullet
)(\lambda_{M^\bullet})$ is an isomorphism, for each
$M^\bullet\in\D(B)$. Applying this to the map
$\lambda_B:B\longrightarrow [\Rh_A(T^\bullet ,?)\circ
(?\Lt_BT^\bullet )](B)=X^\bullet$, we conclude that
$\lambda_B\Lt_B1_{T^\bullet}:T^\bullet\cong
B\Lt_BT^\bullet\longrightarrow X^\bullet\Lt_BT^\bullet$ is an
isomorphism in $\D(A)$.

We now have the following chain of isomorphisms in $\D(A)$:

\begin{center}
$\coprod_{n\in\mathbb{Z}}T^\bullet
[n]\stackrel{\cong}{\longrightarrow}\coprod_{n\in\mathbb{Z}}(X^\bullet\Lt_BT^\bullet
[n])\cong (\coprod_{n\in\mathbb{Z}}X^\bullet
[n])\Lt_BT^\bullet\stackrel{f\otimes
1}{\longrightarrow}(\prod_{n\in\mathbb{Z}}X^\bullet
[n])\Lt_BT^\bullet\cong(\prod_{n\in\mathbb{Z}}\Rh_A(T^\bullet
,T^\bullet [n]))\Lt_BT^\bullet
\stackrel{\cong}{\longrightarrow}\Rh_A(T^\bullet
,\prod_{n\in\mathbb{Z}}T^\bullet
[n])\Lt_BT^\bullet\stackrel{\delta}{\longrightarrow}\prod_{n\in\mathbb{Z}}T^\bullet
[n]$.
\end{center}
It is routine to check that the composition of these isomormorphisms
is the canonical morphism from the coproduct to the product. Arguing
now as in the proof of proposition \ref{prop.tensor by compacts
preserves products}, we conclude that $H^p(T^\bullet )=0$, for all
but finitely many $p\in\mathbb{Z}$.

\vspace*{0.3cm}

(2)   The counit gives an isomorphism

\begin{center}
$B^\alpha\Lt_BT^\bullet\cong \Rh_A(T^\bullet
,T^{\bullet})^\alpha\Lt_BT^\bullet\stackrel{\cong}{\longrightarrow}\Rh_A(T^\bullet
,T^{\bullet\alpha})\Lt_BT^\bullet\stackrel{\delta}{\longrightarrow}T^{\bullet\alpha}$,
\end{center}
for each cardinal $\alpha$. Now apply proposition \ref{prop.upper
bounded of f.g. projectives} to end the proof.
\end{proof}

Unlike the case of $\Rh_A(T^\bullet ,?)$, one-sided exceptionality
is a consequence of the fully faithfulness of $?\Lt_BT^\bullet$.

\begin{prop} \label{prop.fully faithful derived tensor}
Let $T^\bullet$ be a complex of $B-A-$bimodules. Consider the
following assertions:

\begin{enumerate}
\item $?\Lt_BT^\bullet :\D(B)\longrightarrow\D(A)$ is fully faithful;
\item $T_A^\bullet$ is exceptional in $\D(A)$, the canonical algebra morphism $B\longrightarrow\text{End}_{\D(A)}(T^\bullet )$ is an isomorphism, and
 the functor $\Rh_A(T^\bullet
,?)\circ (?\Lt_BT^\bullet ):\D(B)\longrightarrow\D(B)$ preserves
coproducts.
\item  $T_A^\bullet$ is exceptional and self-compact in $\D(A)$, and the canonical algebra morphism $B\longrightarrow\text{End}_{\D(A)}(T^\bullet )$ is an isomorphism
\item $T^\bullet_A$ satisfies the following conditions:

\begin{enumerate}
\item The canonical morphism of algebras $B\longrightarrow\text{End}_{\D(A)}(T^\bullet
)$ is an isomorphism;
\item For each cardinal $\alpha$, the canonical morphism $\D(A)(T^\bullet ,T^\bullet )^{(\alpha )}\longrightarrow\D(A)(T^\bullet ,T^{\bullet (\alpha
)})$ is an isomorphism and $\D(A)(T^\bullet ,T^{\bullet
(\alpha)}[p])=0$, for all $p\in\mathbf{Z}\setminus\{0\}$;
\item $\text{Susp}_{\D(A)}(T^\bullet)^\perp\cap\text{Tria}_{\D(A)}(T^\bullet)$ is closed under
taking coproducts.
\end{enumerate}
\end{enumerate}
The implications $(1)\Longleftrightarrow (2)\Longleftrightarrow
(3)\Longrightarrow (4)$ hold true. When $T_A^\bullet$ is
quasi-isomorphic to a bounded complex of projective $A$-modules, all
assertions are equivalent.
\end{prop}
\begin{proof}
Let $\lambda :1_{\D(B)}\longrightarrow\Rh_A(T^\bullet ,?)\circ
(?\Lt_BT^\bullet )$ be the unit of the adjoint pair
$(?\Lt_BT^\bullet ,\Rh_A(T^\bullet ,?))$. Assertion (1) is equivalent
to saying that $\lambda$ is a natural isomorphism.

$(1)\Longrightarrow (2)$ In particular,  $\lambda_B:B\longrightarrow
[\Rh_A(T^\bullet ,?)\circ (?\Lt_BT^\bullet )](B)=:\Rh_A(T^\bullet
,B\Lt_BT^\bullet )\cong\Rh_A(T^\bullet ,T^\bullet )$ is an
isomorphism. This implies that $T^\bullet$ is exceptional and the
$0$-homology map $B=H^0(B)\longrightarrow H^0(\Rh_A(T^\bullet
,T^\bullet ))=\D(A)(T^\bullet ,T^\bullet )$ is an isomorphism.  But
this latter map is easily identified with the canonical algebra
morphism $B\longrightarrow\text{End}_{\D(A)}(T^\bullet )$. That
$\Rh_A(T^\bullet ,?)\circ (?\Lt_BT^\bullet
):\D(B)\longrightarrow\D(B)$ preserves coproducts is automatic since
this functor is naturally isomorphic to the identity.

$(2)\Longrightarrow (1)$ The full subcategory $\D$ of $\D(B)$
consisting of the objects $X^\bullet$ such that
$\lambda_{X^\bullet}$ is an isomorphism is a triangulated
subcategory, closed under coproducts, which contains $B_B$. Then we
have $\D =\D(B)$.

$(1),(2)\Longrightarrow (3)$ The functor $?\Lt_BT^\bullet
:\D(B)\longrightarrow\D(A)$ induces a triangulated equivalence
$\D(B)\stackrel{\cong}{\longrightarrow}\text{Tria}_{\D(A)}(T^\bullet
)$. The exceptionality (in $\D(A)$ or in
$\text{Tria}_{\D(A)}(T^\bullet )$) and the compactness of
$T^\bullet$ in $\text{Tria}_{\D(A)}(T^\bullet )$ follow from the
exceptionality and compactness of $B$ in $\D(B)$. Moreover, we get
an algebra isomorphism:

\begin{center}
$B\cong\text{End}_{\D(B)}(B)\stackrel{\cong}{\longrightarrow}\text{End}_{\D(A)}(B\Lt_BT^\bullet)\cong\text{End}_{\D(A)}(T^\bullet)$.
\end{center}

$(3)\Longrightarrow (2)$ The functor $\Rh_A(T^\bullet ,?)\circ
(?\Lt_BT^\bullet ):\D(B)\longrightarrow\D(B)$ coincides with the
following composition:

\begin{center}
$\D(B)\stackrel{?\Lt_BT^\bullet
}{\longrightarrow}\text{Tria}_{\D(A)}(T^\bullet
)\stackrel{\Rh_A(T^\bullet ,?)}{\longrightarrow}\D(B)$.
\end{center}
The self-compactness of $T^\bullet$ implies that the second functor
in this composition preserves coproducts. It then follows that the
composition itself preserves coproducts since so does
$?\Lt_BT^\bullet$.

$(3)\Longrightarrow (4)$ Condition (4)(a) is in the hypthesis, and from
the self-compactness and the exceptionality of $T^\bullet$
conditions (4)(b) and (4)(c) follow immediately.

$(4)\Longrightarrow (3)$ Without loss of generality, we assume that
that $T^\bullet$ is a bounded complex of projective $A$-modules in
$\C^{\leq 0}(A)$. We just need to prove the self-compactness of
$T^\bullet$ in $\D(A)$. We put
$\mathcal{U}:=\text{Susp}_{\D(A)}(T^\bullet)$ and
$\mathcal{T}:=\text{Tria}_{\D(A)}(T^\bullet )$. Note that
$(\mathcal{U},\mathcal{U}^\perp [1])$ is a t-structure in $\D(A)$
and that $\mathcal{U}^\perp$ consists of the  $Y^\bullet\in\D(A)$
such that $\D(A)(T^\bullet [k],Y^\bullet )=0$, for all $k\geq 0$
(see proposition \ref{prop.triangulated subcategory is aisle}). Let
$(X^\bullet_i)_{i\in I}$ be any family of objects in $\mathcal{T}$.
We want to check that the canonical morphism $\cp\D(A)(T^\bullet
,X^\bullet_i)\longrightarrow\D(A)(T^\bullet ,\cp X^\bullet_i)$ is an
isomorphism, which is equivalent to proving that it is an
epimorphism. We consider  the triangle associated to the t-structure
$(\mathcal{U},\mathcal{U}^\perp [1])$

\begin{center}
$\cp\tau_\mathcal{U}( X_i)\longrightarrow\cp
X_i\longrightarrow\cp\tau^{\mathcal{U}^\perp}(
X_i)\stackrel{+}{\longrightarrow}$.
\end{center}
This triangle is in $\mathcal{T}$ because its central and left terms
are in $\mathcal{T}$. Note also that we have
$\mathcal{U}\subseteq\D^{\leq 0}(A)$. In particular,
$\tau_\mathcal{U}(X_i)$ is in $\D^{\leq 0}(A)$, for each $i\in I$.

From    \cite[Theorem 3 and its proof]{NS1}, we know that the
inclusion $\mathcal{T}\cap\D^{-}(A)\hookrightarrow\D^-(A)$ has a
right adjoint and that $T^\bullet$ is a compact object of
$\mathcal{T}\cap\D^{-}(A)$. That is, the functor $\D(A)(T^\bullet
,?)$ preserves coproducts of objects in $\mathcal{T}\cap\D^{-}(A)$
whenever the coproduct is in $\D^{-}(A)$. In our case, this implies
that the canonical morphism $\cp\D(A)(T^\bullet
,\tau_\mathcal{U}(X^\bullet_i))\longrightarrow\D(A)(T^\bullet
,\cp\tau_\mathcal{U}(X_i))$ is an isomorphism. On the other hand,
condition 4.c says that $\D(A)(T^\bullet
,\cp\tau^{\mathcal{U}^\perp}(X^\bullet_i))=0$. It follows that the
induced map
$\D(A)(T,\cp\tau_\mathcal{U}(X_i))\longrightarrow\D(A)(T,\cp X_i)$
is an epimorphism. We then get the following commutative square:

\[
\xymatrix{
\coprod_{i\in I}\D(A)(T^\bullet,\tau_{\mathcal{U}}X^\bullet_{i})\ar[d]^{\wr}\ar[r]^{\sim} & \coprod_{i\in I}\D(A)(T^\bullet,X^\bullet_{i})\ar[d] \\
\D(A)(T^\bullet,\coprod_{i\in
I}\tau_{\mathcal{U}}X^\bullet_{i})\ar@{->>}[r] &
\D(A)(T^\bullet,\coprod_{i\in I}X^\bullet_{i}) }
\]

Its upper horizontal and its left vertical arrows are isomorphism,
while the lower horizontal one is an epimorphism. It follows that
the right vertical arrow is an epimorphism.
\end{proof}

Recall that if $A$ and $B$ are dg algebras and $f:B\longrightarrow
A$ is morphism of dg algebras, then $f$ is called a
\emph{homological epimorphism} when the morphism
$(?\Lt_BA)(A)\longrightarrow A$ in $\D(A)$, defined by the
multiplication map $A\Tt_BA\longrightarrow A$, is an isomorphism.
This is also equivalent to saying that the left-right symmetric
version $(A\Lt_B?)(A)\longrightarrow A$ is an isomorphism
$\D(A^{op})$ (see \cite{P}). When $A$ and $B$ are ordinary algebras,
$f$ is a homological epimorphism if, and only if,
$\text{Tor}_i^B(A,A)=0$, for $i\neq 0$, and the multiplication map
$A\otimes_BA\longrightarrow A$ is an isomorphism.

\begin{rem}
Under the hypothesis that $T^\bullet_A$ is quasi-isomorphic to a
bounded complex of projective modules, we do not know if condition
(4)(c) is needed for the implication $(4)\Longrightarrow (3)$ to hold
(see the questions in the next subsection). On the other hand, if in
that same assertion one replaces
$\text{Susp}_{\D(A)}(T^\bullet)^\perp\cap\text{Tria}_{\D(A)}(T^\bullet)$
by just $\text{Susp}_{\D(A)}(T^\bullet)^\perp$, then the implication
$(3)\Longrightarrow (4)$ need not be true, as the following example
shows.
\end{rem}

\begin{exem} \label{exem.LtQ}
The functor
$?\Lt_\mathbb{Q}\mathbb{Q}:\D(\mathbb{Q})\longrightarrow\D(\mathbb{Z})$
is fully faithful, but $\text{Susp}_{\D(\mathbb{Z})}(Q)^\perp$ is
not closed under taking coproducts in $\D(\mathbb{Z})$.
\end{exem}
\begin{proof}
The fully faithful condition of $?\Lt_\mathbb{Q}\mathbb{Q}$  follows
from theorem \ref{teor.recollement D(B)=D(A)=D with Lt} below since
the inclusion $\mathbb{Z}\hookrightarrow\mathbb{Q}$ is a homological
epimorphism. On the other hand, if $I$ is the set of prime natural
numbers and we consider the family of stalk complexes
$(\mathbb{Z}_p[1])_{p\in I}$, we clearly have that
$\D(\mathbb{Z})(Q[k],\mathbb{Z}_p[1])=0$, for all $p\in I$ and
integers $k\geq 0$, so that
$\mathbb{Z}_p[1]\in\text{Susp}_{\D(\mathbb{Z})}(\mathbb{Q})^\perp$.

We claim that $\D(\mathbb{Z})(\mathbb{Q}, \coprod_{p\in
I}\mathbb{Z}_p[1])=\text{Ext}_\mathbb{Z}^1(\mathbb{Q},\coprod_{p\in
I}\mathbb{Z}_p)\neq 0$. Indeed, consider the map
$\epsilon:\mathbb{Q}/\mathbb{Z}\longrightarrow\mathbb{Q}/\mathbb{Z}$
 given by multiplication by the 'infinite product' of all
prime natural numbers. Concretely, if $\frac{a}{p_1^{m_1}\cdot
...\cdot p_t^{m_t}}$ is a fraction of integers, where $m_j>0$ and
$p_j\in I$, for all $j=1,...,t$,  then $\epsilon
(\frac{a}{b}+\mathbb{Z})=\frac{a}{p_1^{m_1-1}\cdot ...\cdot
p_t^{m_t-1}}+\mathbb{Z}$. It is clear that $\text{Ker}(\epsilon )$
is the subgroup of $\mathbb{Q}/\mathbb{Z}$ generated by the elements
 $\frac{1}{p}+\mathbb{Z}$, with $p\in I$, which is clearly
 isomorphic to $\coprod_{p\in I}\mathbb{Z}_p$. In particular $\text{Ext}_\mathbb{Z}^1(\mathbb{Q},\coprod_{p\in
 I}Z_p)$ is isomorphic to the cokernel of the induced map
 $\epsilon_*:\text{Hom}_\mathbb{Z}(Q,\mathbb{Q}/\mathbb{Z})\longrightarrow\text{Hom}_\mathbb{Z}(Q,\mathbb{Q}/\mathbb{Z})$,
 which takes $\alpha\rightsquigarrow\epsilon\circ\alpha$.
The reader is invited to check that the canonical projection
$\pi:\mathbb{Q}\longrightarrow\mathbb{Q}/\mathbb{Z}$ is not in the
image of $\epsilon_*$, thus proving our claim.
\end{proof}

We know pass to study the recollement situations in which one of the
fully faithful functors is $?\Lt_BT^\bullet$.

\begin{teor} \label{teor.recollement D(B)=D(A)=D with Lt}
Let $T^\bullet$ be a complex of $B-A-$bimodules. Consider the
following assertions:

\begin{enumerate}
\item There is a recollement $\D(B)\equiv\D(A)\equiv\D'$, for some
triangulated category $\D'$, where $i_*=?\Lt_BT^\bullet$;
\item There is a recollement $\D(B)\equiv\D(A)\equiv\D'$, for some
triangulated category $\D'$, where $i^!=\Rh_A(T^\bullet ,?)$;
\item $T^\bullet_A$ is exceptional, self-compact, the canonical algebra
morphism $B\longrightarrow\text{End}_{\D(A)}(T^\bullet)$ is an
isomorphism and the functor $?\Lt_BT^\bullet
:\D(B)\longrightarrow\D(A)$ preserves products;
\item $T^\bullet_A$ is exceptional, self-compact, the canonical algebra
morphism $B\longrightarrow\text{End}_{\D(A)}(T^\bullet)$ is an
isomorphism and $\text{Tria}_{\D(A)}(T^\bullet )$ is closed under
taking products in $\D(A)$.
\item $T^\bullet_A$ is exceptional, self-compact, the canonical algebra
morphism $B\longrightarrow\text{End}_{\D(A)}(T^\bullet)$ is an
isomorphism and ${}_BT^\bullet\in\text{per}(B^{op})$.
\item There is a dg algebra $\hat{A}$, a homological epimorphism of
dg algebras $f:A\longrightarrow\hat{A}$ and a classical tilting
object $\hat{T}^\bullet\in\D(\hat{A})$ such that the following
conditions hold:

\begin{enumerate}
\item $T^\bullet_A\cong f_*(\hat{T}^\bullet )$, where $f_*:\D(\hat{A})\longrightarrow\D(A)$ is the restriction of scalars functor;
\item the canonical algebra morphism
$B\longrightarrow\text{End}_{\D(A)}(T)\cong\text{End}_{\D(\hat{A})}(\hat{T})$
is an isomorphism.
\end{enumerate}

\end{enumerate}
Then the implications $(1)\Longleftrightarrow (2)\Longleftrightarrow
(3)\Longleftrightarrow (4)\Longleftrightarrow (5)\Longleftarrow (6)$
hold true and, when $A$ is $k$-flat, all assertions are equivalent.
\end{teor}
\begin{proof}
$(1)\Longleftrightarrow (2)$ is clear.

 Note that the recollement of assertions (1) or (2) exists if,
and only if, the functor $?\Lt_BT^\bullet$ is fully faithful and has
a left adjoint.

$(1)\Longleftrightarrow (3)$ is then a direct consequence of theorem
\ref{prop.fully faithful derived tensor} and corollary
\ref{cor.adjoints with comp.generated domain}.

$(1),(3)\Longrightarrow (4)$ The functor $?\Lt_BT^\bullet$ induces an
equivalence of triangulated categories
$\D(B)\stackrel{\cong}{\longrightarrow}\mathcal{T}:=\text{Tria}_{\D(A)}(T^\bullet
)$. The fact that $?\Lt_BT^\bullet$ has a left adjoint functor
implies that also the inclusion functor
$j:\mathcal{T}\hookrightarrow\D(A)$ has a left adjoint. But then $j$
preserves products, so that $\mathcal{T}$ is closed under taking
products in $\D(A)$.

$(4)\Longrightarrow (3)$ By proposition \ref{prop.fully faithful
derived tensor} and its proof, we know that $?\Lt_BT^\bullet$ is a
fully faithful functor which establishes an equivalence of
triangulated categories
$\D(B)\stackrel{\cong}{\longrightarrow}\mathcal{T}:=\text{Tria}_{\D(A)}(T^\bullet
)$. In particular, this equivalence of categories preserves
products. This, together with the fact that $\mathcal{T}$ is closed
under taking products in $\D(A)$, implies that $?\Lt_BT^\bullet$
preserves products.

$(3)\Longleftrightarrow (5)$ is a direct consequence of proposition
\ref{prop.tensor by compacts preserves products}.

$(6)\Longrightarrow (4)$ By the properties of homological epimorphisms
(see \cite[Section 4]{NS2}), the functor
$f_*:\D(\hat{A})\longrightarrow\D(A)$ is fully faithful (and
triangulated). It then follows that

\begin{center}
$\text{Hom}_{\D(A)}(T^\bullet ,T^\bullet
[p])=\text{Hom}_{\D(A)}(f_*(\hat{T}^\bullet ),f_*(\hat{T}^\bullet
)[p])\cong\text{Hom}_{\D(\hat{A})}(\hat{T}^\bullet ,\hat{T}^\bullet
[p] )=0$,
\end{center}
for all integers $p\neq 0$.  By condition (6)(b), it also follows that
the canonical algebra morphism
$B\longrightarrow\text{End}_{\D(A)}(T^\bullet )$ is an isomorphism.

On the other hand, $f_*$ defines an equivalence
$\D(\hat{A})\stackrel{\cong}{\longrightarrow}\text{Im}(f_*)$.
Moreover $\hat{T}^\bullet$ is a compact generator of $\D(\hat{A})$,
which implies that $T^\bullet _A$ is a compact generator of
$\text{Im}(f_*)$. Since the last one is a full triangulated
subcategory of $\D(A)$ closed under taking coproducts  we get that
$\text{Im}(f_*)=\text{Tria}_{\D(A)}(T^\bullet)$. But
$\text{Im}(f_*)$ is also closed under taking products in $\D(A)$
(see \cite[Section 4]{NS2}). Then $T^\bullet_A$ is self-compact and
$\text{Tria}_{\D(A)}(T^\bullet )$ is closed under taking products in
$\D(A)$.

$(1),(4)\Longrightarrow (6)$ (assuming that   $A$ is $k$-flat)  We have
seen above that the inclusion functor
$\mathcal{T}:=\text{Tria}_{\D(A)}(T^\bullet )\hookrightarrow\D(A)$
has a left adjoint.  By proposition \ref{prop.Keller-Vossieck}, its
dual and by proposition \ref{prop.triangulated subcategory is
aisle}, we get that
$({}^\perp\mathcal{T},\mathcal{T},\mathcal{T}^\perp )$ is a TTF
triple in $\D(A)$. By \cite[Theorem 4]{NS2}, there exists a dg
algebra $\hat{A}$ and a homological epimorphism of dg algebras
$f:A\longrightarrow\hat{A}$ such that $\mathcal{T}=\text{Im}(f_*)$.
Bearing in mind that $f_*$ is fully faithful, it induces an
equivalence of categories
$f_*:\D(\hat{A})\stackrel{\cong}{\longrightarrow}\mathcal{T}$. If
now $\hat{T}^\bullet\in\D(\hat{A})$ is an object such that
$f_*(\hat{T}^\bullet )\cong T^\bullet _A$, then $\hat{T}^\bullet$ is
a compact generator of $\D(\hat{A})$ such that
$\D(\hat{A})(\hat{T}^\bullet ,\hat{T}^\bullet
[p])\cong\D(A)(T^\bullet ,T^\bullet [p])$, for each
$p\in\mathbb{Z}$. Now all requirements for $\hat{T}^\bullet$ are
easily verified.

\end{proof}

The following result deeply extends \cite[Lemma 4.2]{CX2}.

\begin{cor} \label{cor.Chen-Xi}
The following assertions are equivalent for a complex $T^\bullet$ of
$B-A-$bimodules:

\begin{enumerate}
\item There is  a  recollement $\D'\equiv\D(A)\equiv\D(B)$ such that
$j_!=?\Lt_BT^\bullet$, for some triangulated category $\D'$ (which
is equivalent to $\D(C)$, where $C$ is a dg algebra);
\item There is a recollement  $\D'\equiv\D(A)\equiv\D(B)$ such that $j^!=j^*=\Rh_A(T^\bullet
,?)$, for some triangulated category $\D'$ (which is equivalent to
$\D(C)$, where $C$ is a dg algebra)
\item $T^\bullet_A$ is compact and exceptional in $\D(A)$ and the
canonical algebra morphism
$B\longrightarrow\text{End}_{\D(A)}(T^\bullet )$ is an isomorphism.
\end{enumerate}
In such case, if $A$ is $k$-flat, then $C$ can be chosen together
with a homological epimorphism $f:A\longrightarrow C$ such that
$i_*$ is the restriction of scalars $f_*:\D(C)\longrightarrow\D(A)$.
\end{cor}
\begin{proof}
$(1)\Longleftrightarrow (2)$ is clear.

$(1)\Longleftrightarrow (3)$ The mentioned recollement exists if, and
only if, $?\Lt_BT$ is fully faithful and $\Rh_A(T^\bullet ,?)$ has a right
adjoint. Using proposition \ref{prop.fully faithful derived tensor}
and corollary \ref{cor.adjoints with comp.generated domain}, the
existence of such recollement is equivalent to assertion (3). This is
because that $\Rh_A(T^\bullet ,?)$ preserves coproducts if, and only
if, $T^\bullet_A$ is compact in $\D(A)$.

The statements about the dg algebra $C$ follow as at the end of the
proof of theorem \ref{teor.recollement D=D(B)=D(A) with Rh}.
\end{proof}

We then get the following consequence (compare with \cite[Theorem
2]{Ha}).

\begin{cor} \label{cor.right recollement implies left recollement}
Let $T^\bullet$ be a complex of $B-A-$bimodules, where $A$ and $B$
are ordinary algebras. If there is a recollement
$\D'\equiv\D(A)\equiv\D(B)$, with $j_!=?\Lt_BT^\bullet$, for some
triangulated category $\D'$, then there is a recollement
$\D''\equiv\D(A^{op})\equiv\D(B^{op})$, with
$j^{!}=j^*=T^\bullet\Lt_A?$, for a triangulated category $\D''$.
When $A$ is $k$-projective, the converse is also true.
\end{cor}
\begin{proof}
It is a direct consequence of   corollary \ref{cor.Chen-Xi} and the
left-right symmetric version of theorem \ref{teor.recollement
D=D(B)=D(A) with Rh}.
\end{proof}

The following is a rather general result on the existence of
recollements.

\begin{prop} \label{prop.general recollement}
Let $\Lambda$ be any dg algebra and $U^\bullet$, $V^\bullet$ be
exceptional objects of $\D(\Lambda )$. Put $A=\text{End}_{\D(\Lambda
)}(U^\bullet )$ and $B=\text{End}_{\D(\Lambda )}(V^\bullet )$. If
$V^\bullet\in\text{thick}_{\D(\Lambda )}(U^\bullet )$, then there is
a recollement $\D'\equiv\D(A)\equiv\D(B)$, for some triangulated
category $\D'$ (which is equivalent to $\D(C)$, for some dg algebra
$C$).
\end{prop}
\begin{proof}
We choose $U^\bullet$ and $V^\bullet$ to be homotopically
projective. We then put $\hat{A}=\text{End}_\Lambda^\bullet
(U^\bullet )$ and $\hat{B}=\text{End}_\Lambda^\bullet (V^\bullet )$,
which are dg algebras whose associated homology algebras are
concentrated in degree $0$ and satisfy that $H^0(\hat{A})\cong A$
and $H^0(\hat{B})\cong B$. Then we know that $\D(\hat{A})\cong\D(A)$
and $\D(\hat{B})\cong\D(B)$.

We next put $T^\bullet =\Th_\Lambda (U^\bullet ,V^\bullet )$ which
is a dg $\hat{B}-\hat{A}-$bimodule. It is easy to see that
$\Th_\Lambda (U^\bullet ,?)=\Rh_\Lambda (U^\bullet ,?):\D(\Lambda
)\longrightarrow\D(\hat{A})$ induces an equivalence of of
triangulated categories $\text{thick}_{\D(\Lambda )}(U^\bullet
)\stackrel{\cong}{\longrightarrow}\text{thick}_{\D(\hat{A})}(\hat{A})=\text{per}(\hat{A})$.
In particular, we get that $T^\bullet_{\hat{A}}$ is compact in
$\D(\hat{A})$. We claim now that there is a recollement
$\D'\equiv\D(\hat{A})\equiv\D(\hat{B})$, where
$j_!=?\Lt_{\hat{B}}T^\bullet$, an this will end the proof.

The desired recollement exists if, and only if,
$?\Lt_{\hat{B}}T^\bullet$ is fully faithful and
$\Rh_{\hat{A}}(T^\bullet ,?):\D(\hat{A})\longrightarrow \D(\hat{B})$
has a right adjoint. This second condition holds due to corollary
\ref{cor.adjoints with comp.generated domain}. As for the first
condition, we just need to check that the unit $\lambda
:1_{\D(\hat{B})}\longrightarrow\Rh_{\hat{A}}(T^\bullet ,?)\circ
(?\Lt_{\hat{B}}T^\bullet )$ is a natural isomorphism. For that, we
take the full subcategory $\mathcal{X}$ of $\D(\hat{B})$ consisting
of the objects $X^\bullet$ such that $\lambda_{X^\bullet}$ is an
isomorphism. It is clearly a triangulated subcategory closed under
taking coproducts. But if we apply to
$\lambda_{\hat{B}}:\hat{B}\longrightarrow [\Rh_{\hat{A}}(T^\bullet
,?)\circ (?\Lt_{\hat{B}}T^\bullet
)](\hat{B})\cong\Rh_{\hat{A}}(T^\bullet ,T^\bullet )$  the homology
functor, we obtain a map

\begin{center}
$H^p(\lambda_{\hat{B}}):0=H^p(\hat{B})\longrightarrow
H^p(\Rh_{\hat{A}}(T^\bullet ,T^\bullet ))\cong\D(\hat{A})(T^\bullet,
T^\bullet [p])\cong\D(\Lambda )(V^\bullet ,V^\bullet [p])=0$,
\end{center}
 so that $H^p(\lambda_B)$ is an isomorphism, for each
 $p\in\mathbb{Z}\setminus\{0\}$. As for $p=0$, we have

\begin{center}
$H^0(\lambda_{\hat{B}}):B\cong H^0(\hat{B})\longrightarrow
H^0(\Rh_{\hat{A}}(T^\bullet ,T^\bullet ))\cong\D(\hat{A})(T^\bullet,
T^\bullet )\cong\D(\Lambda )(V^\bullet ,V^\bullet)\cong B$,
\end{center}
which also an isomorphism. This proves $\lambda_{\hat{B}}$ is an
isomorphism in $\D(\hat{B})$, so that $\hat{B}\in\mathcal{X}$. It
follows that $\mathcal{X}=\D(\hat{B})$.
\end{proof}

The following is a particular case of last proposition. Compare with
\cite[Theorem 1.3]{CX2}
\begin{exem} \label{exem.Chen-Xi}
Let $\Lambda$ be an algebra, let $V$ be an injective cogenerator of
$\text{Mod}-\Lambda$ and let $B=\text{End}_\Lambda (V)$ be its
endomorphism algebra. Suppose that $U$ is a $\Lambda$-module
satisfying the following two conditions:

\begin{enumerate}
\item  $\text{Ext}_\Lambda^p(U,U)=0$, for all $p>0$;
\item There is an exact sequence $0\rightarrow U^{-n}\longrightarrow ...\longrightarrow U^{-1}\longrightarrow U^0\longrightarrow V\rightarrow
0$, where $U^k\in\text{add}_{\text{Mod}-\Lambda}(U)$, for each
$k=-n,...,-1,0$.
\end{enumerate}
Then there exists a recollement $\D'\equiv\D(A)\equiv\D(B)$, for
some triangulated category $\D'$, where $A=\text{End}_\Lambda (U)$.
Moreover, a slight modification of the proof of last proposition
shows that the recollement can be chosen in such a way that
$j_!=?\Lt_BT$, where $T=\text{Hom}_\Lambda (U,V)$, which is a
$B-A-$bimodule.

Note that when $\Lambda$ is an algebra with Morita duality (e.g. an
Artin algebra), the algebras $\Lambda$ and $A$ are Morita
equivalent, and so $\D(A)\cong\D(\Lambda )$.
\end{exem}

We will end the section by studying an interesting case of fully
faithfulness of $\Rh_A(T^\bullet ,?)$, where two recollement
situations come at once.

\begin{teor} \label{teor.main one}
Let $T^\bullet$ be a complex of $B-A-$bimodules such that
$T^\bullet_A$ is exceptional in $\D(A)$ and the algebra morphism
$B\longrightarrow\text{End}_{\D(A)}(T^\bullet )$ is an isomorphism.
The following assertion are equivalent:

\begin{enumerate}
\item ${}_BT^\bullet$ is compact and exceptional in $\D(B^{op})$ and the canonical algebra morphism $A\longrightarrow\text{End}_{\D(B^{op})}(T^\bullet )^{op}$
 is an isomorphism;
\item[$(1')$]  There is a recollement $\D'\equiv\D(B^{op})\equiv\D(A^{op})$,
with $j_!=T^\bullet\Lt_A?$, for some triangulated category $\D'$
(which is equivalent to $\D(C)$, for some dg algebra $C$);
\item $\Rh_A(T^\bullet ,?):\D(A)\longrightarrow\D(B)$ is fully
faithful and  preserves compact objects;
\item  $A_A$ is in the thick subcategory of $\D(A)$ generated by $T^\bullet_A$;
\item  $?\Lt_BT^\bullet :\D(B)\longrightarrow\D(A)$
has a fully faithful left adjoint;
\item There is a recollement $\D'\equiv\D(B)\equiv\D(A)$, with
 $j_*=\Rh_A(T^\bullet ,?)$, for  triangulated category $\D'$ (which is equivalent to $\D(C)$, for some dg algebra $C$).
\end{enumerate}

When $B$ is $k$-flat, the dg algebra in conditions (1') and (5) can be
chosen together with a homological epimorphism of dg algebras
$f:B\longrightarrow C$ such that $i_*$ is the restriction of scalars
$f_*:\D(C)\longrightarrow\D(B)$.
\end{teor}
\begin{proof}
 Note that the exceptionality of $T^\bullet_A$ plus the fact that
the algebra morphism $B\longrightarrow\text{End}_{\D(A)}(T^\bullet
)$ is an isomorphism is equivalent to saying that the unit map
$\lambda_B:B\longrightarrow[\Rh_A(T^\bullet ,?)\circ
(?\Lt_BT^\bullet )](B)=\Rh_A(T^\bullet ,T^\bullet )$ is an
isomorphism in $\D(B)$.

$(1)\Longleftrightarrow (1')$ follows from the left-right symmetric
version of corollary \ref{cor.Chen-Xi}.

$(1)\Longrightarrow (3)$ The hypothesis implies that the unit map
$\rho_A:A\longrightarrow [\Rh_{B^{op}}(T^\bullet ,?)\circ
(T^\bullet\Lt_A?)](A)$ is an isomorphism in $\D(A^{op})$. Then,
using lemma \ref{lem.units are isomorphisms2} and its terminology,
we know that $\tau_A$ is an isomorphism in $\D(A)$, which implies
that $A_A\cong\Rh_{B^{op}}(?,T^\bullet )(T^\bullet )$. The fact that
${}_BT^\bullet\in\text{per}(B^{op})=\text{thick}_{\D(B^{op})}(B)$
implies then that
$A_A\in\text{thick}_{\D(A)}([\Rh_{B^{op}}(?,T^\bullet)](B))=$ \newline
$\text{thick}_{\D(A)}(T^\bullet
)$.

$(3)\Longrightarrow (1)$  Since $\lambda_B$ is an isomorphism, using
lemma \ref{lem.units are isomorphisms} and its terminology, we get
that
$\tau_{T^\bullet}:T^\bullet\longrightarrow[\Rh_{B^{op}}(\Rh_A(?,T^\bullet),T^\bullet
)](T^\bullet)$ is an isomorphism in $\D(A)$. Since
$A_A\in\text{thick}_{\D(A)}(T^\bullet )$ we get that $\tau_A$ is an
isomorphism in $\D(A)$ which, by lemma \ref{lem.units are
isomorphisms2}, implies that $\rho_A$ is an isomorphism. That is,
${}_BT^\bullet$ is exceptional in $\D(B^{op})$ and the algebra
morphism $A\longrightarrow\text{End}_{\D(B^{op})}(T^\bullet )^{op}$
is an isomorphism.

On the other hand, the fact that
$A_A\in\text{thick}_{\D(A)}(T^\bullet )$ implies that
${}_BT^\bullet\cong\Rh_A(?,T^\bullet )(A)$ is in
$\text{thick}_{\D(B^{op})}(\Rh_A(?,T^\bullet )(T^\bullet
))=\text{thick}_{\D(B^{op})}(B)=\text{per}(B^{op})$.

$(1')\Longrightarrow (5)$ is  the left-right symmetric version of
corollary \ref{cor.right recollement implies left recollement}.

$(4)\Longleftrightarrow (5)$ Apply proposition \ref{prop.functor in a
recollement} to the functor $G=:?\Lt_BT^\bullet
:\D(B)\longrightarrow\D(A)$.

$(4)\Longrightarrow (3)$ Let $L :\D(A)\longrightarrow\D(B)$ be a fully
faithful left adjoint of $?\Lt_BT^\bullet$. One easily sees that $L$
preserves compact objects. Moreover, the unit
$1_{\D(A)}\longrightarrow (?\Lt_BT^\bullet )\circ L$ of the
adjunction $(L,?\Lt_BT^\bullet )$ is a natural isomorphism. It
follows that $A\cong [(?\Lt_BT^\bullet )\circ L](A)=(?\Lt_BT^\bullet
)(L(A))\in (?\Lt_BT^\bullet )(\text{per}(B))=(?\Lt_BT^\bullet
)(\text{thick}_{\D(B)}(B))\subseteq\text{thick}_{\D(A)}((?\Lt_BT^\bullet
)(B))=\text{thick}_{\D(A)}(T^\bullet )$.

$(5),(3)\Longrightarrow (2)$ From assertion (5) we get that
$\Rh_A(T^\bullet ,?)$ is fully faithful. On the other hand, the fact
that  $A_A\in\text{thick}_{\D(A)}(T^\bullet )$ and that $\lambda_B$
is an isomorphism imply that $\Rh_A(T^\bullet
?)(A)\in\text{thick}_{\D(B)}(\Rh_A(T^\bullet
,?)(T^\bullet))=\text{thick}_{\D(B)}(B)=\text{per}(B)$. It follows
from this that $\Rh_A(T^\bullet ,?)$ takes objects of
$\text{per}(A)=\text{thick}_{\D(A)}(A)$ to objects of
$\text{per}(B)$, thus proving assertion (2).

$(2)\Longrightarrow (3)$
 Bearing in mind
that the counit $\delta :(?\Lt_BT^\bullet )\circ\Rh_A(T^\bullet
,?)\longrightarrow 1_{\D(A)}$ is a natural isomorphism and that
$\Rh_A(T^\bullet ,?)(A)$ is a compact object of $\D(B)$, we get:

\begin{center}
$A\cong [?\Lt_BT^\bullet )\circ\Rh_A(T^\bullet
,?)](A)=(?\Lt_BT^\bullet ) )(\Rh_A(T^\bullet ,A))\in(?\Lt_BT^\bullet
)(\text{per}(B))=(?\Lt_BT^\bullet
)(\text{thick}_{\D(B)}(B))\subseteq\text{thick}_{\D(A)}(B\Lt_BT^\bullet
)=\text{thick}_{\D(A)}(T^\bullet )$.
\end{center}

\end{proof}

\begin{rem}
The precursor of theorem \ref{teor.main one} is \cite[Theorem
2.2]{BMT}, where the authors prove that if $T_A$ is a good tilting
module (see definition \ref{def.tilting module}) and
$B=\text{End}(T_A)$, then condition (2) in our theorem holds. It is a
consequence of theorem \ref{teor.main one} (see corollary
\ref{cor.fully faithful Rh-compact for bimodules} below) that the
converse is also true when one assume that $T_A$ has finite
projective dimension and $\text{Ext}_A^p(T,T^{(\alpha )})=0$, for
all integers $p>0$ and all cardinals $\alpha$. Another consequence
(see corollary \ref{cor.fully faithful Rh-compact for bimodules} and
example \ref{exem.tilting-versus-ffaithful Rh}(1)) is that there are
right $A$-modules, other than the good tilting ones, for which the
equivalent conditions of the theorem holds. In the case of good
$1$-tilting modules, it was proved in \cite[Theorem 1.1]{CX1} that
the dg algebra $C$ can be chosen to be an ordinary algebra.

The corresponding of the implication $(1)\Longrightarrow (5)$ in our
theorem was proved in \cite[Theorem 1]{Y} for dg algebras over
field. This result and its converse is then covered by the extension
of theorem \ref{teor.main one} to the context of dg categories,
which is proved in \cite{NS}.
\end{rem}

\subsection{Some natural questions}

As usual, $T^\bullet$ is a complex of $B-A-$bimodules. After the
previous subsection, some natural questions arise, starting with the
questions \ref{ques.recollement situations} of the introduction. Our
next list of examples gives negative answers to all questions
\ref{ques.recollement situations}.

For question 1.2(1)(a), the following is a counterexample:

\begin{exem}
Let $T_A$ be a good tilting module (see definition \ref{def.tilting
module}) which is not finitely generated (e.g.
$\mathbb{Q}\oplus\mathbb{Q}/\mathbb{Z}$ as $\mathbb{Z}$-module) and let 
$B=\text{End}(T_A)$ be its endomorphism algebra. The functor
$\Rh_A(T,?):\D(A)\longrightarrow\D(B)$ is fully faithful, but there
is no recollement $\D(A)\equiv\D(B)\equiv\D'$, with
$i_*=\Rh_A(T,?)$,  for any triangulated category $\D'$.
\end{exem}
\begin{proof}
That $\Rh_A(T,?)$ is fully faithful follows from corollary
\ref{cor.fully faithful Rh-compact for bimodules} below.  On the
other hand, if the desired recollement $\D(A)\equiv\D(B)\equiv\D'$
existed, then, by corollary \ref{cor.recollement D(A)=D(B)=D by
RHom}, we would have that $T_A$ is compact in $\D(A)$, and this is
not the case.
\end{proof}

For question 1.2(1.b), the following is a counterexample.

\begin{exem} \label{exem.fully faithful Rh and no D'=D(B)=D(A)}
If $f:B\longrightarrow A$ is a homological epimorphism of algebras,
and we take $T={}_BA_A$, then
$\Rh_A(A,?)=f_*:\D(A)\longrightarrow\D(B)$ is fully faithful, but
there need not exist a recollement $\D'\equiv\D(B)\equiv\D(A)$, with
$j_*=\Rh_A(A,?)$, for any triangulated category $\D'$.
\end{exem}
\begin{proof}
That $\Rh_A(A,?)=f_*$ is fully faithful follows from the properties
of homological epimorphisms. If the mentioned recollement exists,
then the functor $?\Lt_BA:\D(B)\longrightarrow\D(A)$ preserves
products and, by proposition \ref{prop.tensor by compacts preserves
products}, ${}_BA$ is compact in $\D(B^{op})$. There are obvious
homological epimorphisms which do not satisfy this last property.
\end{proof}

For question 1.2(2)(a), the following is a counterexample, inspired
by  theorem \ref{teor.recollement D(B)=D(A)=D with Lt}:

\begin{exem} \label{exem.fully faithful Lt without recollement}
Let $A$ be an algebra and let $P$ be a finitely generated projective
right $A$-module such that $P$ is not finitely generated as a left
module over $B:=\text{End}_A(P)$. Then
$?\Lt_BP:\D(B)\longrightarrow\D(A)$ is fully faithful, but there is
no recollement $\D(B)\equiv\D(A)\equiv\D'$, with
$i_*=?\Lt_BP=?\otimes_BP$, for any triangulated category $\D'$.

Concretely, if $k$ is a field, $V$ is an infinite dimensional
$k$-vector space and $A=\text{End}_k(V)^{op}$, then
$?\Lt_kV=?\otimes_kV:\D(k)\longrightarrow\D(A)$ is fully faithful,
but does not define the mentioned recollement.
\end{exem}
\begin{proof}
The final statement follows directly from the first part since $V$
is a simple projective right $A$-module such that
$\text{End}_A(V)\cong k$. On the other hand,  we get from
proposition \ref{prop.fully faithful derived tensor} that
$?\Lt_BP:\D(B)\longrightarrow\D(A)$ is fully faithful and, since
${}_BP$ is not compact in $\D(B^{op})$, theorem
\ref{teor.recollement D(B)=D(A)=D with Lt} implies that the
recollement does not exist.

\end{proof}

As a counter example to question 1.2(2)(b), we have:

\begin{exem}
The functor
$?\Lt_\mathbb{Q}\mathbb{Q}=?\otimes_\mathbb{Q}Q:\D(\mathbb{Q})\longrightarrow\D(\mathbb{Z})$
is fully faithful (see example \ref{exem.LtQ}), but there is no
recollement $\D'\equiv\D(\mathbb{Z})\equiv\D(\mathbb{Q})$, with
$j_!=?\otimes_\mathbb{Q}Q$, for any triangulated category $\D'$.
\end{exem}
\begin{proof}
If the recollement existed, then, by corollary \ref{cor.Chen-Xi},
$\mathbb{Q}$ would be compact in $\D(\mathbb{Z})$, which is absurd.
\end{proof}

  But, apart from  questions \ref{ques.recollement situations}, there are some other natural questions whose answer we do
not know even in the case of a bimodule.

\begin{quess} \label{ques.fully-faithfulness of Lt}
\begin{enumerate}
\item (Motivated by proposition \ref{prop.fully faithful derived
tensor}) Suppose that $T^\bullet_A$ is isomorphic in $\D(A)$ to a
bounded complex of projective right $A$-modules,  that the canonical
morphism $\text{Hom}_{D(A)}(T^\bullet ,T^\bullet )^{(\alpha
)}\longrightarrow\text{Hom}_{\D(A)}(T^\bullet ,T^{\bullet (\alpha
)})$ is an isomoprhism and that $\text{Hom}_{\D(A)}(T^\bullet
,T^{\bullet (\alpha )}[p])=0$, for all cardinals  $\alpha$ and all
integers $p\neq 0$. Is $T^\bullet_A$ self-compact in $\D(A)$?.

\item (Motivated by proposition \ref{prop.fully faithful derived tensor}) Suppose that $?\otimes_BT^\bullet :\D(B)\longrightarrow\D(A)$ is
fully faithful. Is $H^p(T^\bullet )=0$ for $p>>0?$. Is $T^\bullet_A$
quasi-isomorphic to a bounded complex of projective right
$A$-modules?
\end{enumerate}
\end{quess}

\begin{rem}
The converse of question \ref{ques.fully-faithfulness of Lt}(2) has
a negative answer, even for a bounded complex of projective
$A$-modules. For instance, if $P$ is a projective generator of
$\text{Mod}-A$ which is not finitely generated, then
$\text{Tria}_{\D(A)}(P)=\D(A)$ and $P$ is not compact in this
category. It follows from proposition \ref{prop.fully faithful
derived tensor} that $?\Lt_BP:\D(B)\longrightarrow\D(A)$ is not
fully faithful, where $B=\text{End}(P_A)$.
\end{rem}

Our next question concerns the relationship between proposition
\ref{prop.fully faithful RHom} and theorem \ref{teor.main one}. That
$\Rh_A(T^\bullet ,?):\D(A)\longrightarrow\D(B)$ be fully faithful
does not imply that it preserves compact objects (see example
\ref{exem.fully faithful Rh and no D'=D(B)=D(A)}).  The correct
question to answer, for which we do not have an aswer, is the
following:

\begin{ques} \label{ques.Rh faithful preserves compacts}
Suppose that $T^\bullet_A$ is exceptional in $\D(A)$,   the
canonical algebra morphism
$B\longrightarrow\text{End}_{\D(A)}(T^\bullet )$ is an isomorphism
and that $\Rh_A(T^\bullet ,?):\D(A)\longrightarrow\D(B)$ is fully
faithful. Due to theorems \ref{teor.recollement D=D(B)=D(A) with Rh}
and \ref{teor.main one}, each of the following questions has an
affirmative answer if, and only if, so do the other ones, but we do
not know the answer:

\begin{enumerate}
\item Does $\Rh_A(T^\bullet ,?)$ preserve compact objects?;
\item Is $A_A$ in
$\text{thick}_{\D(A)}(T^\bullet )$?; \item Is ${}_BT^\bullet$
compact in $\D(B^{op})$?
\end{enumerate}
Note that, by proposition \ref{prop.faithful Rh for exceptional},
${}_BT^\bullet$ is isomorphic in $\D(B^{op})$ to an upper bounded
complex of finitely generated projective left $B$-modules with
bounded homology.
\end{ques}

In next section, we will show that, in case $T^\bullet$ is a
$B-A-$bimodule, the question has connections with Wakamatsu tilting
problem.


\section{The case of a bimodule}

\subsection{Re-statement of the main results}
For the convenience of the reader, we make explicit  what some
results of the previous section say in the particular case when
$T^\bullet =T$ is just a $B-A-$bimodule. The statements show a close
connection with the theory of (not necessarily finitely generated)
tilting modules.

 Recall:

\begin{defi} \label{def.tilting module}
Consider the following conditions for an $A$-module $T$:

\begin{enumerate}
\item[a)]  $T$ has finite projective dimension;
\item[a')] $T$ admits a finite projective resolution with finitely
generated terms;
\item[b)] There is an exact sequence $0\rightarrow A\longrightarrow T^0\longrightarrow T^1\longrightarrow ....\longrightarrow T^m\rightarrow
0$ in $\text{Mod}-A$, with $T^i\in\text{Add}(T)$ for $i=0,1,...,m$;
\item[b')] There is an exact sequence $0\rightarrow A\longrightarrow T^0\longrightarrow T^1\longrightarrow ....\longrightarrow T^m\rightarrow
0$ in $\text{Mod}-A$, with $T^i\in\text{add}(T)$ for $i=0,1,...,m$;
\item[c)] $\text{Ext}_A^p(T,T^{(\alpha )})=0$, for all integers
$p>0$ and all cardinals $\alpha$.
\end{enumerate}
$T$ is called a \emph{$n$-tilting module} when conditions a), b) and
c) hold and $\text{pd}_A(T)=n$. Such a tilting module is
\emph{classical $n$-tilting} when it satisfies a'), b') and c) and
it is called a \emph{good $n$-tilting module} when it satisfies
conditions a), b') and c). We will simply say that $T$ is tilting
(resp.  classical tilting, resp. good tilting) when it is
$n$-tilting (resp. classical $n$-tilting, resp. good $n$-tilting),
for some $n\in\mathbf{N}$.
\end{defi}

\begin{rem}
Note that $T$ is a classical $n$-tilting module if, and only if, it
satisfies conditions a'), b') and  $\text{Ext}_A^p(T,T)=0$, for all
integers $p>0$.
\end{rem}

When $\Ext_A^p(T,T)=0$, for all $p>0$,  it is proved in \cite{NS}
that the condition that $A_A\in\text{thick}_{\D(A)}(T)$ is
equivalent condition b') in Definition \ref{def.tilting module}.

In the rest of the subsection, unless otherwise stated, $T$ will be
a  $B-A-$bimodule  and all statements  are given for it. The
following result is then a direct consequence of proposition
\ref{prop.fully faithful derived tensor}:

\begin{cor} \label{cor.fully faithful Lt for bimodule}
Consider the following assertions:

\begin{enumerate}
\item $\Lt_BT:\D(B)\longrightarrow\D(A)$ is fully faithful;
\item $\text{Ext}_A^p(T,T)=0$, for all $p>0$, the canonical algebra
morphism $B\longrightarrow\text{End}_A(T)$ is an isomorphism and
$T_A$ is a compact object of $\text{Tria}_{\D(A)}(T)$.

\item The following conditions hold:

\begin{enumerate}
\item the canonical map $B^{(\alpha
)}\longrightarrow\text{Hom}_A(T,T)^{(\alpha
)}\longrightarrow\text{Hom}_A(T,T^{(\alpha )})$ is an isomorphism,
for all cardinals $\alpha$
\item $\text{Ext}_A^p(T,T^{(\alpha )})=0$, for all cardinals $\alpha$ and integers $p>0$;
\item  for each family $(X_i^\bullet )_{i\in I}$ in $\text{Tria}_{\D(A)}(T)$ such that
$\D(A)(T[k],X^\bullet_i)=0$, for all $k\geq 0$ and all $i\in I$, one
has that $\D(A)(T,\coprod_{i\in I}X_i^\bullet )=0$.
\end{enumerate}
\end{enumerate}

The implications $(1)\Longleftrightarrow (2)\Longrightarrow (3)$ hold
true. When $T_A$ has finite projective dimension, all assertions are
equivalent.
\end{cor}

The next result is then a consequence of theorem
\ref{teor.recollement D(B)=D(A)=D with Lt}:

\begin{cor} \label{cor.recollement D(B)=D(A)=D with bimodule}
The following assertions are equivalent:

\begin{enumerate}

\item There is a recollement  $\D(B)\equiv\D(A)\equiv\D'$ with $i_*=?\Lt_BT$,
for some triangulated category $\D'$;

\item  $\text{Ext}_A^p(T,T)=0$, for all $p>0$, the canonical algebra
morphism $B\longrightarrow\text{End}_A(T)$ is an isomorphism,  $T$
is a compact object of $\text{Tria}_{\D(A)}(T)$ and this subcategory
is closed under taking products in $\D(A)$.
\item  $\text{Ext}_A^p(T,T)=0$, for all $p>0$, the canonical algebra
morphism $B\longrightarrow\text{End}_A(T)$ is an isomorphism,  $T$
is a compact object of $\text{Tria}_{\D(A)}(T)$ and $T$ admits a
finite projective resolution with finitely generated terms as a left
$B$-module.

When $A$ is $k$-flat, these conditions are also equivalent to:

\item  There is a dg algebra $\hat{A}$, a homological epimorphism of
dg algebras $f:A\longrightarrow\hat{A}$ and a classical tilting
object  $\hat{T}^\bullet\in\D(\hat{A})$ such that

\begin{enumerate}
\item $f_*(\hat{T}^\bullet )\cong T_A$, where $f_*:\D(\hat{A})\longrightarrow\D(A)$ is the restriction of scalars functor;
\item The canonical algebra morphism
$B\longrightarrow\text{End}_{A}(T)\cong\text{End}_{\D(\hat{A})}(\hat{T}^\bullet)$
is an isomorphism.
\end{enumerate}

\end{enumerate}
\end{cor}

The next result is a direct consequence of corollaries
\ref{cor.Chen-Xi} and \ref{cor.right recollement implies left
recollement}:

\begin{cor}
Consider the following assertions for the $B-A-$bimodule $T$:

\begin{enumerate}
\item $T_A$ admits a finite projective resolution with finitely
generated terms, $\text{Ext}_A^p(T,T)=0$, for all $p>0$, and the
algebra morphism $B\longrightarrow\text{End}(T_A)$ is an
isomorphism;
\item There is recollement $\D'\equiv D(A)\equiv\D(B)$, with $j_!=?\Lt_BT$,  for some
triangulated category $\D'$ (which is equivalent to $\D(C)$, for
some dg algebra $C$);
\item There is a recollement $\D'\equiv D(A^{op})\equiv\D(B^{op})$,
with $j^!=j^*=T\Lt_A?$, for some triangulated category $\D'$ (which
is equivalent to $\D(C^{op})$, for some dg algebra $C$).
\end{enumerate}
Then the implications $(1)\Longleftrightarrow (2)\Longrightarrow (3)$
hold true. When $A$ is $k$-projective, all assertions are
equivalent.
\end{cor}

The next result is a direct consequence of theorem \ref{teor.main
one} and the definition of good tilting module.

\begin{cor} \label{cor.fully faithful Rh-compact for bimodules}
Let $T$ be a right $A$-module such that $\text{Ext}_A^p(T,T)=0$, for
all $p>0$, and let $B=\text{End}_A(T)$. The following assertions are
equivalent:

\begin{enumerate}
\item $\text{Ext}_{B^{op}}^p(T,T)=0$, for all $p>0$, the canonical algebra
morphism $A\longrightarrow\text{End}_{B^{op}}(T)^{op}$ is an
isomorphism and $T$ admits a finite projective resolution with
finitely generated terms as a left $B$-module;
\item $\Rh_A(T,?):\D(A)\longrightarrow\D(B)$ is fully faithful and
preserves compact objects;
\item There exists an exact sequence $0\rightarrow A\longrightarrow
T^0\longrightarrow T^1 ...\longrightarrow T^n\rightarrow 0$ in
$\text{Mod}-A$, with $T^k\in\text{add}(T)$ for each $k=0,1,...,n$;

\item $?\Lt_B:\D(B)\longrightarrow\D(A)$ has a fully faithful left
adjoint.
\end{enumerate}
When in addition $T_A$ has finite projective dimension and
$\text{Ext}_A^p(T,T^{(\alpha )})=0$, for all cardinals $\alpha$ and
all integers $p>0$, the above conditions are also equivalent to:

\begin{enumerate}
\item[(5)] $T$ is a good tilting right $A$-module.
\end{enumerate}
\end{cor}

The last results show that the fully faithful condition of the
classical derived functors associated to an exceptional module is
closely related  to  tilting theory. However, this relationship
tends to be tricky, as the following examples show. They are
explained in detail in the final part of \cite{NS}.

\begin{exems} \label{exem.tilting-versus-ffaithful Rh}
\begin{enumerate}
\item If $A$ is a non-Noetherian hereditary algebra and $I$ is an
injective cogenerator of $\text{Mod}-A$ containing an isomorphic
copy of each cyclic module, then $T=E(A)\oplus\frac{E(A)}{A}\oplus
I$ satisfies the conditions (1)-(5) of corollary \ref{cor.fully faithful
Rh-compact for bimodules}, but $\text{Ext}_A^1(T,T^{(\aleph_0)})\neq
0$. Hence $T$ is not a tilting $A$-module.
\item Let  $A$ be a right Noetherian right hereditary algebra such
that $\text{Hom}_A(E(A/A,E(A)))=0$ and $E(A)/A$ contains an
indecomposable summand  with infinite multiplicity. If $I$ is the
direct sum of one isomorphic copy of each indecomposable summand of
$E(A)/A$, then  $T=E(A)\oplus I$ is a $1$-tilting module such that
$\Rh_A(T,?):\D(A)\longrightarrow\D(B)$ is not fully faithful. The
Weyl algebra $A_1(k)=k<x,y>/(xy-yx-1)$ over the field $k$ is an
example where the situation occurs.
\item If $A$ is a hereditary Artin algebra,  $T$ is a finitely generated projective
right $A$-module which is not a generator and $B=\text{End}_A(T)$,
then $T$ admits a finite projective resolution with finitely
generated terms as a left $B$-module, but
$\Rh_A(T,?):\D(A)\longrightarrow\D(B)$ is not fully faithful. Indeed
$\Rh_A(T,?)$ preserves compact objects, but  condition (3) of last
corollary does not hold.
\end{enumerate}
\end{exems}

\subsection{Connection with Wakamatsu tilting problem}

In this subsection we show a connection of question \ref{ques.Rh
faithful preserves compacts} with a classical problem in
Representation Theory.

\begin{defi}

Let $T_A$ be a module and $B=\text{End}(T_A)$. Consider the
following conditions:

\begin{enumerate}
\item $T_A$ admits a projective resolution with finitely generated terms;
\item $Ext_A^p(T,T)=0$, for all $p>0$;
\item There exists an exact sequence $0\rightarrow A\rightarrow T^0\rightarrow T^1\rightarrow
...\rightarrow T^n\rightarrow ...$ such that
\begin{enumerate}
\item $T^i\in add(T_A)$, for all $i\geq 0$;
\item The functor $Hom_A(?,T)$ leaves the sequence exact.
\end{enumerate}
\item There exists an exact sequence $0\rightarrow A\rightarrow T^0\rightarrow ...\rightarrow T^n\rightarrow 0$, with $T^i\in add(T_A)$, for all $i\geq 0$.
\end{enumerate}

We shall say that $T_A$ is

\begin{enumerate}
 \item[a)] \emph{Wakamatsu tilting} when (1), (2) and (3) hold;
 \item[b)] \emph{semi-tilting} when (1), (2) and (4) hold;
 \item[c)] \emph{generalized Wakamatsu tilting} when (2) and (3) hold;
 \item[d)] \emph{generalized semi-tilting} when (2) and (4) hold.
 \end{enumerate}
 \end{defi}

\begin{rem}
 Each classical tilting module is (generalized) semi-tilting
and each (generalized) semi-tilting module is (generalized)
Wakamatsu tilting.
\end{rem}

\begin{prop} \label{prop.Wakamatsu problem}
Let $T$ be a Wakamatsu tilting right $A$-module. The following
assertions are equivalent:

\begin{enumerate}
\item $T$ is classical tilting;
\item $T$ is semi-tilting of finite projective dimension;
\item $T$ has finite projective dimension, both as a right $A$-module and as left module over $B=\text{End}(T_A)$.
\end{enumerate}
\end{prop}
\begin{proof}
$(1)\Longrightarrow (2)$ is clear.

$(2)\Longrightarrow (3)$ By hypothesis, we have that
$\text{pd}(T_A)<\propto$. On the other hand, a finite projective
resolution for ${}_BT$ is obtained by applying the functor
$\text{Hom}_A(?,T)$ to the exact sequence $0\rightarrow A\rightarrow
T^0\rightarrow ...\rightarrow T^n\rightarrow 0$, with $T^i\in
add(T_A)$ given in the definition of semi-tilting module.

$(3)\Longrightarrow (1)$  This is known (see \cite[Section 4]{MR}).
\end{proof}

\begin{ques}
\begin{enumerate}
\item[1.]  Is statement (1) of last proposition
true, for all Wakamatsu tilting modules?.
\item[2.] We can ask an intermediate question, namely: is each Wakamatsu
tilting module a semi-tilting one?.
\end{enumerate}
\end{ques}

\begin{rem}
The answer to question 1 is negative in general (see \cite[Example
3.1]{Wa}). However it is still an open question, known as
\emph{Wakamatsu tilting problem}, whether each Wakamatsu tilting
module of finite projective dimension is classical tilting. Note
that, by proposition \ref{prop.Wakamatsu problem}, an affirmative
answer to question 2 above implies an affirmative answer to
Wakamatsu problem and, conversely.
\end{rem}

It turns out that question 2 is related to question \ref{ques.Rh
faithful preserves compacts}, as the following result shows:

\begin{prop}
Let us assume that $\text{Ext}_A^p(T,T)=0$, for all $p>0$, and that
$\text{RHom}_A(T,?):\mathcal{D}(A)\longrightarrow\mathcal{D}(B)$ is
fully faithful, where $B=\text{End}(T_A)$. Consider the following
assertions:

\begin{enumerate}
\item $\text{RHom}_A(T,?)$ preserves compact objects;
\item $T_A$ is a generalized semi-tilting module;
\item $T_A$ is a generalized Wakamatsu tilting module;
\item The structural algebra homomorphism $A\longrightarrow\text{End}_{B^{op}}(T)^{op}$ is an isomorphism
and $\text{Ext}_{B^{op}}^p(T,T)=0$, for all $p>0$.
\end{enumerate}
Then the implications $(1)\Longleftrightarrow (2)\Longrightarrow
(3)\Longleftrightarrow (4)$ hold true.
\end{prop}
\begin{proof}
$(1)\Longleftrightarrow (2)$ is a direct consequence of corollary
\ref{cor.fully faithful Rh-compact for bimodules} and the definition
of generalized semi-tllting module.

$(2)\Longrightarrow (3)$ is clear.

$(3)\Longrightarrow (4)$ Let us fix an exact sequence $0\rightarrow
A\longrightarrow
T^0\stackrel{d^0}{\longrightarrow}T^1\stackrel{d^1}{\longrightarrow}
...\stackrel{d^{n-1}}{\longrightarrow}
T^n\stackrel{d^n}{\longrightarrow} ...$ (*), with
$T^i\in\text{add}(T)$, for all $i\geq 0$. As shown in the proof of
proposition \ref{prop.Wakamatsu problem}, when we apply to it the
functor $\text{Hom}_A(?,T):\text{Mod}-A\longrightarrow
B-\text{Mod}$, we obtain a projective resolution of
${}_BT\cong\text{Hom}_A(A,T)$. Bearing in mind that the canonical
natural transformation $\sigma
:1_{\text{Mod}-A}\longrightarrow\text{Hom}_{B^{op}}(\text{Hom}_A(?,T),T)$
is an isomorphism, when evaluated at a module $T'\in\text{add}(T)$,
when we apply $\text{Hom}_{B^{op}}(?,T)$ to that projective
resolution of ${}_BT$, we obtain a sequence

\begin{center}
$0\rightarrow\text{Hom}_{B^{op}}(T,T)\longrightarrow
T^0\stackrel{d^0}{\longrightarrow}
T^1\stackrel{d^1}{\longrightarrow}
...\stackrel{d^{n-1}}{\longrightarrow}
T^n\stackrel{d^n}{\longrightarrow} ...$.
\end{center}
This sequence is exact due to the left exactness of
$\text{Hom}_{B^{op}}(?,T)$ and to the exactness of the sequence (*)
and, hence, both sequences are isomorphic. Then assertion (4) holds.

$(4)\Longrightarrow (3)$ By proposition \ref{prop.faithful Rh for
exceptional}, we know that ${}_BT$ admits a projective resolution
with finitely generated terms, say

\begin{center}
$...P^{-n}\longrightarrow ...\longrightarrow P^{-1}\longrightarrow
P^0\longrightarrow {}_BT\rightarrow 0$. \hspace*{1cm}(**)
\end{center}
The hypotheses imply that, when we apply to it the functor
$\text{Hom}_{B^{op}}(?,T)$, we obtain an exact sequence in
$\text{Mod}-A$

\begin{center}
$0\rightarrow
A\longrightarrow\text{Hom}_{B^{op}}(P^0,T)\longrightarrow\text{Hom}_{B^{op}}(P^1,T)\longrightarrow
...\longrightarrow\text{Hom}_{B^{op}}(P^{-n},T)\longrightarrow ...$.
\end{center}
Note that $\text{Hom}_{B^{op}}(P^{-i},T)$ is a direct summand of
$\text{Hom}_{B^{op}}(B^{(r)},T)\cong T_A^{(r)}$, for some
$r\in\mathbb{N}$, so that $\text{Hom}_{B^{op}}(P^{-i},T)=:T^i$ is in
$\text{add}(T_A)$, for each $i\geq 0$. Note also that the canonical
natural transformation $\sigma
:1_{B-\text{Mod}}\longrightarrow\text{Hom}_{A}(\text{Hom}_{B^{op}}(?,T),T)$
is an isomorphism when evaluated at any finitely generated
projective left $B$-module, because $\text{Hom}_{A}(T,T)\cong
{}_BB$. It follows from this and the fact that
$\text{Ext}_A^p(T,T)=0$, for all $p>0$, that when we apply
$\text{Hom}_A(?,T)$ to the last exact sequence we obtain, up to
isomorphism,  the initial projective resolution (**). Then the exact
sequence

\begin{center}
$0\rightarrow A_A\longrightarrow T^0\longrightarrow
T^1\longrightarrow ...\longrightarrow T^n\longrightarrow ...$
\end{center}
is kept exact when applying $\text{Hom}_{A}(?,T)$. Therefore $T_A$
is a generalized Wakamatsu tilting module.
\end{proof}

As an immediate consequence, we get:

\begin{cor}
Each of the following statements is true if, and only if, so is the
other:

\begin{enumerate}
\item If $T_A$ is a generalized Wakamatsu tilting module such that
$\Rh_A(T,?):\D(A)\longrightarrow\D(B)$ is fully faithful, where
$B=\text{End}_A(T)$,  then $T_A$ is generalized semi-tilting.
\item Let ${}_BT_A$ be a bimodule such that
$\text{Ext}_{A}^p(T,T)=0=\text{Ext}_{B^{op}}^p(T,T)$, for all $p>0$
and  the algebra morphisms $B\longrightarrow\text{End}_{A}(T)$ and
$A\longrightarrow\text{End}_{B^{op}}(T)^{op}$ are isomorphisms. If
the functor $\Rh_A(T,?):\D(A)\longrightarrow\D(B)$ is fully
faithful, then it preserves compact objects.
\end{enumerate}
\end{cor}

\vspace*{0.5cm}

\emph{ACKNOWLEGMENTS:} The preparation of this paper started in a
visit of the second named author to the University of Shizuoka and
his subsequent participation in the 46th Japan Symposium on Ring
Theory and Representation Theory, held in Tokyo on October 12-14,
2014. Saor\'in thanks Hideto Asashiba and all the organizers of the
Symposium for their invitation and for their extraordinary
hospitality.

We also thank Jeremy Rickard for his help in the proof of
proposition \ref{prop.tensor by compacts preserves products}.

Both authors are supported by research projects from the Spanish
Ministry of Education (MTM2010-20940-C02-02) and from the
Fundaci\'on 'S\'eneca' of Murcia (04555/GERM/06), with a part of
FEDER funds. We thank these institutions for their help.

\ifx\undefined\bysame 
\newcommand{\bysame}{\leavevmode\hbox to3em{\hrulefill}\,} 
\fi

\end{document}